\documentclass[11pt]{amsart}
\usepackage[margin=1.35in]{geometry}

\usepackage[T1]{fontenc}
\usepackage{lmodern}
\usepackage{mathtools,amssymb,mathrsfs,booktabs}
\usepackage{microtype}
\usepackage{siunitx}
\usepackage{xcolor}
\usepackage{tikz}
\usepackage{fancyvrb}
\usepackage[colorlinks=true,urlcolor=blue,linkcolor=black,citecolor=black]{hyperref}
\hypersetup{pdftitle={The stable Bernstein theorem in R7},
pdfauthor={Han Hong, Haizhong Li, and Gaoming Wang}}

\DefineVerbatimEnvironment{codeblock}{Verbatim}{
  frame=single,framerule=0.3pt,framesep=5pt,
  rulecolor=\color{black!40},fontsize=\small}

\numberwithin{equation}{section}
\numberwithin{table}{section}
\newtheorem{theorem}{Theorem}[section]
\newtheorem{proposition}[theorem]{Proposition}
\newtheorem{lemma}[theorem]{Lemma}
\newtheorem*{nonlinearendpointlemma}{Lemma~\ref{lem:nonlinear-endpoints}}
\newtheorem*{terminalspectrallemma}{Lemma~\ref{lem:spectral-scalar-lower}}
\newtheorem{corollary}[theorem]{Corollary}
\theoremstyle{remark}
\newtheorem{remark}[theorem]{Remark}
\allowdisplaybreaks[1]

\title{The stable Bernstein theorem in $\mathbb{R}^{7}$}
\author{Han Hong}
\address{Department of Mathematics and Statistics, Beijing Jiaotong University,
Beijing 100044, China}
\email{hanhong@bjtu.edu.cn}

\author{Haizhong Li}
\address{Department of Mathematical Sciences, Tsinghua University, 100084, Beijing, China}
\email{lihz@tsinghua.edu.cn}

\author{Gaoming Wang}
\address{Beijing Institute of Mathematical Sciences and Applications,
Beijing 101408, China}
\email{wanggaoming@bimsa.cn}

\begin{document}

\begin{abstract}
We give a Green-function proof of the stable Bernstein theorem in
$\mathbb R^7$ for smooth, connected, complete, two-sided stable
minimal hypersurfaces, thus resolving the last case of the stable Bernstein problem.
\end{abstract}

\maketitle
\setcounter{tocdepth}{1}
\tableofcontents

\section{Introduction}\label{sec:introduction}

The stable Bernstein problem asks in $\mathbb{R}^{n+1}$, $n\leq 6$, whether a complete,
two-sided stable minimal hypersurface in Euclidean space must be a hyperplane (see \cite[Problem 102]{Yauproblemlist}).
Let $F:(M^n,g)\to\mathbb R^{n+1}$ be a smooth minimal immersion, and let
$A$ denote its second fundamental form with respect to a global unit normal $\nu_M$.
Stability is the nonnegativity of the second variation of area under
compactly supported normal variations, equivalently
\begin{equation}\label{eq:stability}
 \int_M |A|^2f^2\,dV_g\leq\int_M|\nabla f|^2\,dV_g,
 \qquad f\in C_c^\infty(M).
\end{equation}
By approximation this inequality also holds for compactly supported
Lipschitz functions.

For surfaces in $\mathbb R^3$, the theorem was proved by do Carmo--Peng~\cite{doCarmoPeng1979}, Fischer Colbrie--Schoen~\cite{FischerColbrieSchoen1980} and Pogorelov~\cite{Pogorelov1981} independently.
Simons' identity~\cite{Simons1968} provides a basic tool for higher-dimensional
rigidity. Under a Euclidean volume growth assumption, the curvature estimates
of Schoen--Simon--Yau~\cite{SchoenSimonYau1975} imply flatness for complete
two-sided stable minimal immersed hypersurfaces of dimensions at most five.
Schoen--Simon~\cite{SchoenSimon1981} established the corresponding result in
dimension six for embedded hypersurfaces, and
Bellettini~\cite{Bellettini2025} subsequently extended it to immersed
hypersurfaces through a De Giorgi iteration argument.
Related work on the regularity and compactness of stable minimal hypersurfaces
includes contributions by Schoen--Simon~\cite{SchoenSimon1981},
Wickramasekera~\cite{Wickramasekera2014}, Bellettini~\cite{Bellettini2025},
Hong--Li--Wang~\cite{HongLiWang2024}, Minter--Xiao~\cite{MinterXiao2026},
and Wang--Zhang~\cite{WangZhang2026}.

A central difficulty in the stable Bernstein problem is to prove flatness
without assuming Euclidean volume growth.
Chodosh and Li~\cite{ChodoshLi2024} resolved the first case beyond surfaces
by proving the stable Bernstein theorem in $\mathbb R^4$ without any
volume growth assumption. Their proof studies level-set integrals of the Green gradient
and combines stability with the geometry of two-dimensional level sets
of the Green function. They subsequently gave another proof
in~\cite{ChodoshLi2023Anisotropic}, combining the Gulliver--Lawson conformal
change with warped $\mu$-bubbles to obtain intrinsic cubic volume growth.
Chodosh, Li, Minter, and Stryker~\cite{ChodoshLiMinterStryker2026} extended
this strategy to $\mathbb R^5$, using spectral bi-Ricci curvature and a
spectral volume comparison theorem to control the three-dimensional
$\mu$-bubbles. Mazet~\cite{Mazet2024} further refined this approach to prove
the stable Bernstein theorem in $\mathbb R^6$, using the spectral volume
comparison of Antonelli and Xu~\cite{AntonelliXu2024}.

In a 2022 preprint, subsequently published as~\cite{CatinoMastroliaRoncoroni2024},
Catino, Mastrolia, and Roncoroni gave a conformal proof in $\mathbb R^4$. More recently, Hong and Wang~\cite{HongWang2026} gave a short volume-growth proof of the same result. 
They use mixed radial volume comparison under a spectral Ricci bound to obtain polynomial weighted volume growth, 
and then combine it with a weighted $L^4$ curvature inequality.
Using Green functions, Cabr\'e, Catino, Mari, Mastrolia, and
Roncoroni~\cite{CabreEtAl2026} developed weighted curvature inequalities
and Green-gradient estimates under spectral Ricci bounds, obtaining another
proof in dimension $4$.

We use cubic Green-function level integrals of the type appearing in the
monotonicity formulas of Colding~\cite{Colding2012} and
Colding--Minicozzi~\cite{ColdingMinicozzi2014}.
We combine Green identities, the Codazzi equation, and divergence identities for curvature tensors to prove the stable Bernstein theorem in intrinsic dimension six.
The algebraic inequalities used in the proof are formally proved in Lean
and independently verified in Mathematica using exact rational arithmetic.

\begin{theorem}\label{thm:main}
Let $F:(M^6,g)\to\mathbb R^7$ be a smooth, connected, complete,
two-sided stable minimal immersion.
Then $A\equiv0$, and $F(M)$ is a
hyperplane.
\end{theorem}

The lower-dimensional conclusions (already well-known) follow by taking products with Euclidean space.
\begin{corollary}\label{cor:lower-dimensions}
Let $2\leq n\leq6$, and let $F:(M^n,g)\to\mathbb R^{n+1}$ be a smooth,
connected, complete, two-sided stable minimal immersion.
Then $A\equiv0$, and $F(M)$ is an affine $n$-plane.
\end{corollary}
For $n<6$, the product immersion $(x,y)\mapsto(F(x),y)$ from
$M\times\mathbb R^{6-n}$ into $\mathbb R^7$ is complete, two-sided, and
minimal. It is stable by applying the stability inequality on each
$M$-slice and integrating over $\mathbb R^{6-n}$.
Theorem~\ref{thm:main} makes this product flat, so $A\equiv0$ on $M$.

\begin{samepage}
Theorem~\ref{thm:main} also gives the following curvature estimate by a
standard blow-up argument.
\begin{corollary}[Curvature estimate]\label{cor:curvature-estimate}
There exists a universal constant $C>0$ such that every smooth, two-sided
stable minimal immersion $F:(M^6,g)\to\mathbb R^7$, with $M$ compact and
$\partial M\ne\varnothing$, satisfies
\[
 |A|(p)\,d_M(p,\partial M)\leq C,
 \qquad p\in M\setminus\partial M,
\]
where $d_M$ denotes the intrinsic distance on $M$.
\end{corollary}
\end{samepage}

\subsection{Outline of the proof}

After the bounded-curvature reduction, let $G$ be the minimal positive
Green function of $-\Delta$ with a unit pole.
For a fixed small $t_0>0$, define
\[
 J(t):=\int_{\{G=t\}}|\nabla G|^3\,dA_g,
 \qquad \mathcal M_{\sigma}:=\int_0^{t_0}J(t)t^{-\sigma -1}\,dt.
\]
Here $\sigma$ denotes the moment exponent.
Unit flux and the gradient estimate $|\nabla G| \leq CG$ initially give
$\mathcal M_{\sigma}<\infty$ only for $\sigma <2$.
The exponent $5/2$ is critical because, on a Euclidean six-plane,
$J(t)=32\sqrt{|\mathbb S^5|}\,t^{5/2}$, so $\mathcal M_{5/2}$ diverges logarithmically.
We first prove that all subcritical moments are finite and then deduce flatness.

\smallskip
\noindent\textbf{Part I: Finiteness of all subcritical moments.}
We prove that
\[
 \mathcal M_\sigma<\infty\qquad\text{for every }\sigma<5/2.
\]
This integrability allows us to remove the exterior cutoffs in the
critical-energy argument in the second part.
We extend the moment bounds using Corollary~\ref{cor:continuation}.
If the cutoff inequality \eqref{eq:abstract-cutoff-ineq} holds uniformly
for $\sigma\in[\sigma_0,\sigma_1]$, then
$\mathcal M_{\sigma_0}<\infty$ implies $\mathcal M_{\sigma_1}<\infty$.
Starting from the initial range $\sigma<2$, we reach the full subcritical
range in three steps.

\smallskip
\noindent\emph{Step 1: Reaching $\mathcal M_{2.46}<\infty$.}
We combine stability with weighted identities derived from Bochner's
formula and the Codazzi equation.
A Schoen--Simon--Yau (SSY) estimate~\cite{SchoenSimonYau1975}, obtained from
Simons' identity and stability, controls the curvature integrals and cutoff
errors. Together, these estimates yield the required cutoff inequality
and hence $\mathcal M_{2.46}<\infty$.

\smallskip
\noindent\emph{Step 2: From $2.46$ to $2.4962$.}
The next step applies the stability inequality to a nonlinear combination of two test functions.
With a compactly supported cutoff $\phi(G)$, set
\[
 f_1:=|\nabla G| G^{(1-\sigma)/2}\phi^2(G),
 \qquad f_2:=|A|G^{(3-\sigma)/2}\phi^2(G).
\]
For $\mathbf f:=(f_1,f_2)^{\mathsf T}$ and a positive definite matrix
$\mathsf M=(m_{ij})_{i,j=1}^2$, apply stability to
$\mathcal F:=(\mathbf f^{\mathsf T}\mathsf M\mathbf f)^{1/2}$.
Differentiating this norm gives the additional nonnegative gradient term in \eqref{eq:norm-id}.
Combining this term with the Green and curvature identities yields a
coercive inequality on eleven rational parameter intervals covering
$[2.46,2.4962]$.
The corresponding estimates with compactly supported cutoff functions imply
$\mathcal M_{2.4962}<\infty$.

\smallskip
\noindent\emph{Step 3: From $2.4962$ to every exponent below $5/2$.}
For the last interval, we add the divergence identities of the three tensors
\[
 |A|^{-2}A^4,\qquad |A|^{-2}\operatorname{tr}(A^3)A,\qquad
 |A|^{-2}\operatorname{tr}(A^4)g
\]
to the earlier relations. These additional identities yield the coercive
estimate needed to continue from $\mathcal M_{2.4962}<\infty$ to
$\mathcal M_\sigma<\infty$ for every $\sigma<5/2$, completing the
first part of the proof.

Steps 2 and 3 use the algebraic estimates in
Lemmas~\ref{lem:nonlinear-endpoints} and~\ref{lem:spectral-scalar-lower},
respectively. These two lemmas are formally proved in Lean and independently
verified in Mathematica; see Appendix~\ref{sec:coefficients-verification}
for details.

\smallskip
\noindent\textbf{Part II: From subcritical moments to flatness.}
At the critical exponent $\sigma=5/2$, the coercive estimate from Step 3 still controls
the curvature terms. Combining this estimate with the subcritical moment
bounds, we use exponentially decaying weights to absorb the cutoff errors
and prove finiteness of the critical curvature energies.

Bochner's formula then relates the critical Hessian and curvature integrals
to a Green boundary term. Brendle's sharp isoperimetric
inequality~\cite{Brendle2021}, applied to the Green superlevel sets, shows
that this boundary term is nonpositive. Removing the cutoffs in the critical
coercive estimate therefore forces the curvature energies to vanish,
yielding $A\equiv0$.

\vskip.3cm

\emph{The paper is organized as follows.}
Section~\ref{sec:preparation} presents the geometric preliminaries and the
moment-continuation criterion.
Section~\ref{sec:first-estimates} proves $\mathcal M_{2.46}<\infty$.
Sections~\ref{sec:compact-system} and~\ref{sec:nonlinear-moments} combine
further weighted identities and spectral estimates with nonlinear stability
to prove $\mathcal M_{2.4962}<\infty$.
Section~\ref{sec:all-subcritical-moments} extends the moment bounds to every
$\sigma<5/2$.
Section~\ref{sec:critical-rigidity} establishes the critical energy bounds
and completes the proof of the main theorem.
The appendices contain the coefficient tables, formal verification details,
and coefficient expansions.

\subsection{On the usage of AI}
AI assistance (ChatGPT 5.5, 5.6 and Astra) played a substantial role in
developing this work. Through repeated interaction, we used AI to explore
several approaches, including weighted volume growth, weighted
SSY inequalities, and Green function methods, as well as to
derive and compare identities, search for admissible parameters and smaller
sufficient values of $\delta$ in the $\delta$-stability inequality, and develop
the accompanying verification code. The authors directed the investigation,
examined gaps in proposed arguments, and determined which approaches to
use or revise. This exploration led to a proof of the main result, which
the authors subsequently modified and expanded substantially. Codex was
also used to assist with the exposition of this paper. The authors
take full responsibility for the mathematical content and conclusions of
the paper.

\subsection{Acknowledgements}
The authors greatly appreciate Professor S.-T. Yau's interest in this work.

The first author is supported by the Fundamental Research Funds for the
Central Universities, grant no.~YA26JBMC00040, and by the National Natural
Science Foundation of China, grant no.~12401058. The second author is supported by the National Natural Science Foundation of China,
grant no.~12471047.
The third author is supported by the National Natural Science Foundation of China,
grant no.~12601090.

\section{Preliminaries}\label{sec:preparation}
\subsection{Geometric reduction and the Green function}\label{subsec:green-preparation}

Let $M$ be a $6$-dimensional stable minimal hypersurface. The metric $g$ is induced by the immersion.
We use $\Delta:=\operatorname{div}\nabla$, and identify symmetric two-tensors with
self-adjoint endomorphisms by means of $g$.
Thus $A^2(v,w)=\langle A(v,\cdot),A(w,\cdot)\rangle$ and
$\operatorname{Ric}=-A^2$ by the Gauss equation.
Tensor norms and contractions use $g$.
Unless otherwise indicated, an integral on $M$ uses its volume measure
$dV_g$, and a level integral uses $\,dA_g$.
The letter $C$ denotes a positive constant whose value may change.
We write $\omega_m:=|\mathbb S^m|$ for the area of the Euclidean unit $m$-sphere.

\begin{lemma}\label{lem:reduction}
If Theorem~\ref{thm:main} fails, there exists a complete, two-sided stable
minimal immersion with bounded second fundamental form and $|A|(o)=1$
at some point $o\in M$.
\end{lemma}
\begin{proof}
This follows from the standard point-picking, rescaling, and limiting
argument; see \cite[Lecture~3]{White2016} and
\cite[beginning of Section~5]{ChodoshLi2024}.
\end{proof}
We work with this hypersurface throughout the proof.

\begin{lemma}\label{lem:Green}
The minimal
positive Green function $G$ with pole $o$ exists on $M$.
It satisfies
\begin{equation}\label{eq:Green-flux}
 -\Delta G=\delta_o,\qquad G\longrightarrow0\text{ at infinity},\qquad
 \int_{\{G=t\}}|\nabla G|\,dA_g=1
\end{equation}
for almost every $t>0$.
For every $0<a<b<\infty$, the set $\{x\in M\setminus\{o\}:a\leq G(x)\leq b\}$ is compact.
There are $t_0>0$ and $C<\infty$ such that
\begin{equation}\label{eq:gradient-bound}
 |\nabla G|\leq CG\qquad\text{on }\{G<2t_0\}.
\end{equation}
Near the pole, with $r_o:=d(o,\cdot)$,
\begin{equation}\label{eq:pole-asymptotic}
 G=\frac{r_o^{-4}}{4\omega_5}(1+o(1)),\qquad
 |\nabla G|=\frac{r_o^{-5}}{\omega_5}(1+o(1)).
\end{equation}
\end{lemma}
\begin{proof}
The submanifold Sobolev inequality of Michael and Simon~\cite{MichaelSimon1973} for a
complete minimal immersion in dimension six implies nonparabolicity and hence the
existence of the minimal Green function.
The associated Gaussian heat-kernel bound, integrated in time, implies $G(o,x)\to0$ as
$d(o,x)\to\infty$.
These facts hold in general dimensions; see \cite[Section 2,
before (2.1)]{CabreEtAl2026}; the heat-kernel implication there is attributed to
\cite{Davies1987}.
The manifold $M$ is complete, so vanishing at infinity implies compactness of positive
superlevel sets, with the pole included.
Integrating $-\Delta G=\delta_o$ across a regular superlevel gives the unit flux.

Boundedness of $|A|$ gives a uniform lower Ricci bound.
The Cheng--Yau gradient estimate~\cite{ChengYau1975}, applied on balls of
fixed radius away from the pole, therefore yields \eqref{eq:gradient-bound}
for sufficiently small $t_0$; see also
\cite[Section~2, before (2.4)]{CabreEtAl2026}.

The local fundamental-solution expansion and elliptic estimates on shrinking annuli yield
\eqref{eq:pole-asymptotic}.
\end{proof}
In particular, with
\begin{equation}\label{eq:moments}
 J(t):=\int_{\{G=t\}}|\nabla G|^3\,dA_g,\qquad
 \mathcal M_{\sigma}:=\int_0^{t_0}J(t)t^{-\sigma -1}\,dt,
\end{equation}
combining the gradient estimate \eqref{eq:gradient-bound} with the unit flux
identity in \eqref{eq:Green-flux} gives, for almost every $0<t<t_0$,
\begin{equation}\label{eq:J-quadratic-bound}
 J(t)=\int_{\{G=t\}}|\nabla G|^2|\nabla G|\,dA_g
 \leq C^2t^2\int_{\{G=t\}}|\nabla G|\,dA_g=C^2t^2.
\end{equation}
Consequently
\begin{equation}\label{eq:initial-seed}
 \mathcal M_{\sigma}<\infty\qquad\text{for all }\sigma<2.
\end{equation}

\subsection{Moment extension}\label{sec:moment-continuation}\label{subsec:moment-continuation}
We first prove a moment-continuation criterion for a nonnegative Borel measure.

Let $\mu$ be a nonnegative Borel measure on $[0,\infty)$.
Define
\begin{equation}\label{eq:borel-integral}
 I(\sigma):=\int_{[0,\infty)}e^{\sigma s}\,d\mu(s),
 \qquad \sigma\in\mathbb R.
\end{equation}

\begin{proposition}\label{prop:borel-continuation}
Let $1\leq p<\infty$, $C>0$, and $\sigma_0\leq\sigma_1$. Suppose
$I(\sigma_0)<\infty$ and that
\begin{equation}\label{eq:borel-cutoff-ineq}
 \int \zeta^p e^{\sigma s}\,d\mu
 \leq C\int\bigl(|\partial_s\zeta|^p+|\partial_s^2\zeta|^p\bigr)e^{\sigma s}\,d\mu
\end{equation}
for every nonnegative $\zeta\in C_c^\infty((0,\infty))$ and every
$\sigma\in[\sigma_0,\sigma_1]$, with $C$ independent of $\zeta$ and $\sigma$.
Then $I(\sigma)<\infty$ for every $\sigma\in[\sigma_0,\sigma_1]$.
\end{proposition}

For comparison, the standard Laplace--Stieltjes transform is
\begin{equation}\label{eq:laplace-sign}
 L_\mu(\lambda):=\int_{[0,\infty)}e^{-\lambda s}\,d\mu(s),
 \qquad I(\sigma)=L_\mu(-\sigma).
\end{equation}
Thus $I$ uses the usual Laplace convention
with the parameter sign reversed.
We keep the positive sign so that increasing $\sigma$ corresponds to
testing integrability against a more rapidly growing weight.
\begin{proof}
We claim that there exists $\delta>0$, independent of
$\tau\in[\sigma_0,\sigma_1]$, such that if $I(\tau)<\infty$, then
$I(\tau')<\infty$ for every
$\tau'\in[\tau,\min\{\tau+\delta,\sigma_1\}]$.

Fix a smooth nondecreasing function $\psi:\mathbb R\to[0,1]$ such that
$\psi=0$ on $(-\infty,0]$ and $\psi=1$ on $[1,\infty)$.
For $R\geq1$ and $S>R+2$, apply \eqref{eq:borel-cutoff-ineq} to
$
 \zeta(s)=\psi(s-R)\psi(S+1-s).
$
Since $\zeta=1$ on $[R+1,S]$ and its derivatives are supported in the
two transition intervals, we obtain
\[
 \int_{[R+1,S]}e^{\tau s}\,d\mu(s)
 \leq C\left(
   \int_{[R,R+1)}e^{\tau s}\,d\mu(s)
  +\int_{[S,S+1)}e^{\tau s}\,d\mu(s)\right).
\]
Here the bounds for the derivatives of the fixed cutoff have been
absorbed into $C$, which depends only on the original constant, $p$,
and $\psi$, and is independent of $\tau,R,S$.
For fixed $R$, let $S\to\infty$. Since $I(\tau)<\infty$, the last
integral tends to zero and hence
\[
 \int_{[R+1,\infty)}e^{\tau s}\,d\mu(s)
 \leq C\int_{[R,R+1)}e^{\tau s}\,d\mu(s).
\]
Rearranging, we have
\[
 \int_{[R+1,\infty)}e^{\tau s}\,d\mu(s)
 \leq\frac{C}{1+C}\int_{[R,\infty)}e^{\tau s}\,d\mu(s).
\]
Consequently, for every integer $n\geq0$,
\[
 \int_{[n,\infty)}e^{\tau s}\,d\mu(s)
 \leq\left(\frac{C}{1+C}\right)^{n-1}I(\tau).
\]
Choose $\delta:=\frac12\log\left(1+\frac1C\right)>0$,
which is independent of $\tau$ and satisfies
$e^\delta C/(1+C)<1$. For any
$\tau'\in[\tau,\min\{\tau+\delta,\sigma_1\}]$, splitting the integral
into unit intervals gives
\begin{align*}
 I(\tau')
 &=\sum_{n=0}^\infty\int_{[n,n+1)}e^{(\tau'-\tau)s}e^{\tau s}\,d\mu(s)
 \leq\sum_{n=0}^\infty e^{\delta(n+1)}
       \int_{[n,\infty)}e^{\tau s}\,d\mu(s)\\
 &\leq I(\tau)\sum_{n=0}^\infty e^{\delta(n+1)}
       \left(\frac{C}{1+C}\right)^{n-1}
 =\frac{e^\delta(1+C)}{C}I(\tau)
       \sum_{n=0}^\infty\left(\frac{e^\delta C}{1+C}\right)^n\\
 &=\frac{e^\delta(1+C)^2}{C(1+C-e^\delta C)}I(\tau)<\infty.
\end{align*}
The series converges because $e^\delta C/(1+C)<1$, proving the claim.

Starting from $I(\sigma_0)<\infty$, repeated application of the claim
reaches $\sigma_1$ after finitely many steps, since $\delta>0$ is fixed.
\end{proof}

We apply Proposition~\ref{prop:borel-continuation} to the Green moments.
Set $D:=t\partial_t$, with primes denoting differentiation in $t$, and define
\[
 d\mu^M_\sigma:=|\nabla G|^4G^{-\sigma-1}\,dV_g.
\]
For a measurable function $h$ of the Green value, define
\begin{equation}\label{eq:cutoff-moment}
 \mathcal M_\sigma[h]:=\int_{\{0<G<t_0\}}|h(G)|^4\,d\mu^M_\sigma,
 \qquad \mathcal M_\sigma[1]=\mathcal M_\sigma.
\end{equation}
All cutoff derivatives are taken before composition with $G$.
The weighted identities below hold on each compact region $\{a\leq G\leq b\}$ with $0<a<b<\infty$;
when using $\mathcal M_\sigma[\phi]$, we take
$\phi\in C_c^\infty((0,t_0))$, so their integrals are on this exterior region.
By the coarea formula, the pushforward of $d\mu^M_\sigma$ under $G$ on
$\{0<G<t_0\}$ is $J(t)t^{-\sigma-1}\,dt$.
\begin{corollary}[Moment continuation]\label{cor:continuation}
Let $\sigma_0\leq\sigma_1$. Suppose there is a constant $C>0$ such that
\begin{equation}\label{eq:abstract-cutoff-ineq}
 \mathcal M_\sigma[\phi]
 \leq C\bigl(\mathcal M_\sigma[|D\phi|]+\mathcal M_\sigma[|D^2\phi|]\bigr)
\end{equation}
for every $0\leq\phi\in C_c^\infty((0,t_0))$, uniformly for
$\sigma\in[\sigma_0,\sigma_1]$.
If $\mathcal M_{\sigma_0}<\infty$, then
$\mathcal M_{\sigma_1}<\infty$.
\end{corollary}
\begin{proof}
By the coarea formula, apply Proposition~\ref{prop:borel-continuation}
with $p=4$, $s=\log(t_0/t)$, and
$d\mu(s)=J(t_0e^{-s})\,ds$. Here
$I(\sigma)=t_0^\sigma\mathcal M_\sigma$, and $D\phi=-\partial_s\zeta$,
$D^2\phi=\partial_s^2\zeta$ for $\phi(t)=\zeta(\log(t_0/t))$, so
\eqref{eq:abstract-cutoff-ineq} is precisely \eqref{eq:borel-cutoff-ineq}.
\end{proof}

\section{Initial moment improvement}\label{sec:first-estimates}
We establish weighted Green identities and combine stability with an SSY
estimate to prove $\mathcal M_{2.46}<\infty$.

\subsection{Notation and elementary algebra}
On $\{|\nabla G|>0\}$, define
\[
 \nu:=\frac{\nabla G}{|\nabla G|},\qquad
 \rho:=\frac{|A|^2G^2}{|\nabla G|^2}.
\]
Define the shifted Hessian by
\begin{equation}\label{eq:shifted-hessian}
 Z_\sigma:=\frac{G}{|\nabla G|^2}\nabla^2G
 -\frac{\sigma}{10}(6\nu\otimes\nu-g).
\end{equation}
Here $\nu$ is tangent to $M$ and normal to its regular Green level sets.
We have $\operatorname{tr}Z_\sigma=0$. Where $|A|>0$, set
\[
 \widehat A:=A/|A|.
\]
All weighted densities are interpreted in their undivided form. For example,
\[
 \rho\,d\mu^M_\sigma=|A|^2|\nabla G|^2G^{1-\sigma}dV_g,
 \qquad \rho^2\,d\mu^M_\sigma=|A|^4G^{3-\sigma}dV_g.
\]
Likewise, the Hessian energy density is
\[
 |Z_\sigma|^2d\mu^M_\sigma
 =G^{1-\sigma}\left|\nabla^2G-
 \frac{\sigma}{10G}(6dG\otimes dG-|\nabla G|^2g)\right|^2dV_g.
\]
If $M$ is flat there is nothing to prove. Otherwise, analyticity of the
minimal immersion and of $G$ implies that the zero sets of $|A|$ and
$|\nabla G|$ have volume zero. The undivided formulas and locally Lipschitz
approximation justify all calculations on each compact region $\{a\leq G\leq b\}$ with $0<a<b<\infty$.

\begin{remark}
The shifted Hessian is motivated by the Euclidean Green function.
On a Euclidean six-plane, $G=(4\omega_5)^{-1}r_o^{-4}$ satisfies
\[
 \nabla^2G=\frac{1}{4G}
 \left(6\,dG\otimes dG-|\nabla G|^2g\right)
\]
away from the pole. Thus $Z_{5/2}=0$ on the Euclidean model,
and $Z_{5/2}$ measures the deviation from this radial Hessian identity.
The critical tensor is a scalar multiple of the trace-free Hessian used in
Colding's Green-function monotonicity formulas~\cite{Colding2012}.
More precisely,
\[
 Z_{5/2}=-\frac{2G^{5/2}}{|\nabla G|^2}
 \left[\nabla^2(G^{-1/2})-\frac{\Delta(G^{-1/2})}{6}g\right].
\]
The family in \eqref{eq:shifted-hessian} extends this critical construction
to the varying exponent $\sigma$.
\end{remark}

\begin{lemma}\label{lem:early-algebra}
Let $A$ be the second fundamental form of a six-dimensional minimal
hypersurface in $\mathbb R^7$. Then
\begin{equation}\label{eq:global-Newton}
 |\nabla A|^2\geq\frac43|\nabla|A||^2,
 \qquad \operatorname{tr}(A^4)\leq\frac7{10}|A|^4.
\end{equation}
\end{lemma}
\begin{proof}
The refined Kato inequality is classical; see \cite{SchoenSimonYau1975}.
The Gauss equation and the orthogonal Ricci--Weyl decomposition
\cite[Proposition 5.1 and (5.11)]{Viaclovsky2007} give
\[
 0\leq|W_M|^2=\frac{21}{10}|A|^4-3\operatorname{tr}(A^4),
\]
where $W_M$ is the Weyl tensor of the induced metric.
This proves \eqref{eq:global-Newton}.
\end{proof}

\begin{remark}\label{rem:algebraic-quartic}
The quartic inequality in Lemma~\ref{lem:early-algebra} is purely algebraic.
More generally, if $n\geq2$ and $a_1,\ldots,a_n\in\mathbb R$ satisfy
$\sum_{i=1}^n a_i=0$, then
\begin{equation}\label{eq:algebraic-quartic}
 \sum_{i=1}^n a_i^4
 \leq\frac{n^2-3n+3}{n(n-1)}
 \left(\sum_{i=1}^n a_i^2\right)^2.
\end{equation}
Indeed, for $n\geq3$, let $B=\operatorname{diag}(a_1,\ldots,a_n)$ and
form the algebraic curvature tensor
$R_{ijkl}=B_{ik}B_{jl}-B_{il}B_{jk}$.
The same orthogonal curvature decomposition gives
\[
 0\leq|W|^2
 =\frac{2(n^2-3n+3)}{(n-1)(n-2)}|B|^4
 -\frac{2n}{n-2}\operatorname{tr}(B^4),
\]
which proves \eqref{eq:algebraic-quartic}; the case $n=2$ follows from
$a_2=-a_1$.
In particular, $n=6$ and $n=5$ give the constants $7/10$ and $13/20$,
respectively.
\end{remark}

Simons' identity~\cite{Simons1968} reads
\begin{equation}\label{eq:Simons-identity}
 \frac12\Delta|A|^2=|\nabla A|^2-|A|^4.
\end{equation}
The corresponding Simons' inequality~\cite{SchoenSimonYau1975} is
\begin{equation}\label{eq:Simons-inequality}
 |A|\Delta|A|+|A|^4\geq\frac13|\nabla|A||^2
 \qquad\text{on }\{|A|>0\}.
\end{equation}

\begin{lemma}[Weighted Laplacian]\label{lem:weighted-Laplacian}
Let $a\in\mathbb R$ and $v\in C^2((0,\infty))$, and write $D:=t\partial_t$.
Then
\begin{equation}\label{eq:weighted-second-derivative}
\begin{aligned}
 \frac{d^2}{dt^2}(t^av)
 &=t^{a-2}\{t^2v''+2atv'+a(a-1)v\}\\
 &=t^{a-2}\{D^2v+(2a-1)Dv+a(a-1)v\}.
\end{aligned}
\end{equation}
Consequently, on $M\setminus\{o\}$,
\begin{equation}\label{eq:weighted-Laplacian}
 \Delta(G^av(G))=|\nabla G|^2G^{a-2}
 \bigl[D^2v+(2a-1)Dv+a(a-1)v\bigr]_{t=G}.
\end{equation}
\end{lemma}
\begin{proof}
The identities $Dv=tv'$ and $D^2v=tv'+t^2v''$ give
\eqref{eq:weighted-second-derivative} by the product rule.
Since $\Delta G=0$ away from the pole, the chain rule gives
\[
 \Delta(G^av(G))
 =|\nabla G|^2\left.\frac{d^2}{dt^2}(t^av(t))\right|_{t=G}.
\]
Substituting \eqref{eq:weighted-second-derivative} proves
\eqref{eq:weighted-Laplacian}.
\end{proof}

\subsection{The SSY estimate and cutoff errors}
\begin{lemma}\label{lem:cutoff-remainders}
Let $\phi\in C_c^\infty((0,t_0))$ and suppose that
\[
 |\mathcal R|\leq C\int_M(1+\rho)
 \bigl[|\phi|^3(|D\phi|+|D^2\phi|)+\phi^2|D\phi|^2\bigr]\,d\mu^M_\sigma.
\]
Then, for every $\varepsilon>0$,
\begin{equation}\label{eq:cutoff-remainder-young}
 |\mathcal R|\leq\varepsilon\int_M\rho^2\phi^4\,d\mu^M_\sigma
 +\varepsilon\mathcal M_\sigma[\phi]
 +C_\varepsilon\bigl(\mathcal M_\sigma[|D\phi|]
                     +\mathcal M_\sigma[|D^2\phi|]\bigr).
\end{equation}
The constant $C_\varepsilon$ is uniform for $\sigma$ in a fixed compact
interval whenever the constant in the hypothesis is uniform.
\end{lemma}
\begin{proof}
Young's inequality gives, for every $\delta>0$,
\[
\begin{aligned}
 |\mathcal R|
 &\leq\delta\int_M(1+\rho^2)\phi^4\,d\mu^M_\sigma
 +C_\delta\int_M\bigl[\phi^2(|D\phi|^2+|D^2\phi|^2)+|D\phi|^4\bigr]\,d\mu^M_\sigma\\
 &\leq2\delta\int_M(1+\rho^2)\phi^4\,d\mu^M_\sigma
 +C_\delta\bigl(\mathcal M_\sigma[|D\phi|]
                +\mathcal M_\sigma[|D^2\phi|]\bigr).
\end{aligned}
\]
In the second inequality we applied Young's inequality once more to
$\phi^2|D\phi|^2$ and $\phi^2|D^2\phi|^2$.
Choosing $\delta$ sufficiently small proves \eqref{eq:cutoff-remainder-young}.
\end{proof}

Define the stability form
\[
 \mathcal Q(f,g):=\int_M(\langle\nabla f,\nabla g\rangle-|A|^2fg)\,dV_g.
\]
Stability gives $\mathcal Q(f,f)\geq0$ for compactly supported test functions.

\begin{lemma}\label{lem:ssy-lemma}
For a compactly supported smooth function $h$ on $M$,
\begin{equation}\label{eq:curvature-product}
 \mathcal Q(|A|h,|A|h)
 =\int_M |A|^2|\nabla h|^2\,dV_g
 -\int_M h^2\bigl(|\nabla A|^2-|\nabla|A||^2\bigr)\,dV_g.
\end{equation}
Consequently, for a compactly supported smooth function $f$ on $M$,
\begin{equation}\label{eq:retained-ssy}
 \int_M |A|^4f^2\,dV_g\leq4\int_M |A|^2|\nabla f|^2\,dV_g-\frac12\int_M |A|^2\Delta(f^2)\,dV_g.
\end{equation}
Hence, for $0\leq\phi\in C_c^\infty((0,t_0))$ and every $\varepsilon>0$,
\begin{equation}\label{eq:physical-ssy}
\begin{aligned}
 \int_M\rho^2\phi^4\,d\mu^M_\sigma
 &\leq\frac{(3-\sigma)(4-\sigma)}2\int_M\rho\phi^4\,d\mu^M_\sigma
 +\varepsilon\mathcal M_\sigma[\phi]\\
 &\quad+C_\varepsilon\bigl(\mathcal M_\sigma[|D\phi|]
                       +\mathcal M_\sigma[|D^2\phi|]\bigr).
\end{aligned}
\end{equation}
In particular,
\begin{equation}\label{eq:ssy}
 \int_M\rho^2\phi^4d\mu^M_\sigma
 \leq C\bigl(\mathcal M_\sigma[\phi]
       +\mathcal M_\sigma[|D\phi|]+\mathcal M_\sigma[|D^2\phi|]\bigr).
\end{equation}
The constants are uniform for $\sigma$ in a fixed compact interval.
\end{lemma}

\begin{proof}
Expanding $|\nabla(|A|h)|^2$ and integrating the mixed term
$\frac12\langle\nabla|A|^2,\nabla(h^2)\rangle$ by parts gives
\[
 \mathcal Q(|A|h,|A|h)
 =\int_M |A|^2|\nabla h|^2\,dV_g
 -\int_M h^2\left(\frac12\Delta|A|^2+|A|^4-|\nabla|A||^2\right)\,dV_g.
\]
Substituting Simons' identity \eqref{eq:Simons-identity} proves
\eqref{eq:curvature-product}.
By stability and the refined Kato inequality in \eqref{eq:global-Newton},
\begin{equation}\label{eq:ssy-gradient-bound}
 \frac14\int_M h^2|\nabla A|^2\,dV_g
 \leq\int_M h^2\bigl(|\nabla A|^2-|\nabla|A||^2\bigr)\,dV_g
 \leq\int_M |A|^2|\nabla h|^2\,dV_g.
\end{equation}
Using Simons' identity \eqref{eq:Simons-identity}, integrating by parts, and applying
\eqref{eq:ssy-gradient-bound} with $h=f$, we obtain
\[
\begin{aligned}
 \int_M |A|^4f^2\,dV_g
 &=\int_M f^2|\nabla A|^2\,dV_g-\frac12\int_M |A|^2\Delta(f^2)\,dV_g\\
 &\leq4\int_M |A|^2|\nabla f|^2\,dV_g-\frac12\int_M |A|^2\Delta(f^2)\,dV_g.
\end{aligned}
\]
This proves \eqref{eq:retained-ssy}.

Take $f:=G^{(3-\sigma)/2}\phi^2$. The product rule gives
\begin{equation}\label{eq:ssy-test-gradient}
 \nabla f=G^{(1-\sigma)/2}
 \left(\frac{3-\sigma}{2}\phi^2+2\phi D\phi\right)\nabla G.
\end{equation}
Using \eqref{eq:weighted-Laplacian} with $a=3-\sigma$ and $v=\phi^4$ gives
\begin{equation}\label{eq:ssy-test-Laplacian}
 \Delta(f^2)=|\nabla G|^2G^{1-\sigma}
 \bigl[D^2(\phi^4)+(5-2\sigma)D(\phi^4)
 +(3-\sigma)(2-\sigma)\phi^4\bigr].
\end{equation}
Substituting \eqref{eq:ssy-test-gradient} and \eqref{eq:ssy-test-Laplacian}
into \eqref{eq:retained-ssy} yields
\begin{equation}\label{eq:ssy-weighted-with-error}
 \int_M\rho^2\phi^4\,d\mu^M_\sigma
 \leq\frac{(3-\sigma)(4-\sigma)}2\int_M\rho\phi^4\,d\mu^M_\sigma
 +\mathcal E,
\end{equation}
where the terms involving $D\phi$ or $D^2\phi$ are
\[
 \mathcal E=\int_M\rho\bigl[(14-4\sigma)\phi^3D\phi
 +10\phi^2(D\phi)^2-2\phi^3D^2\phi\bigr]\,d\mu^M_\sigma.
\]
Applying Lemma~\ref{lem:cutoff-remainders} to these terms gives
\begin{equation}\label{eq:ssy-cutoff-error}
\begin{aligned}
 |\mathcal E|
 &\leq C\int_M\rho
 \bigl(\phi^3|D\phi|+\phi^2|D\phi|^2+\phi^3|D^2\phi|\bigr)\,d\mu^M_\sigma\\
 &\leq\varepsilon\int_M\rho^2\phi^4\,d\mu^M_\sigma
 +\varepsilon\mathcal M_\sigma[\phi]
 +C_\varepsilon\bigl(\mathcal M_\sigma[|D\phi|]
                     +\mathcal M_\sigma[|D^2\phi|]\bigr).
\end{aligned}
\end{equation}
First fix a small error parameter in \eqref{eq:ssy-cutoff-error}.
Applying Young's inequality to the $\rho\phi^4$ term in
\eqref{eq:ssy-weighted-with-error} and absorbing the resulting small
multiples of $\int_M\rho^2\phi^4\,d\mu^M_\sigma$ proves \eqref{eq:ssy}.
Next, substitute \eqref{eq:ssy} into the curvature error in
\eqref{eq:ssy-cutoff-error}. Choosing its error parameter sufficiently small
and using \eqref{eq:ssy-weighted-with-error} gives \eqref{eq:physical-ssy}.
\end{proof}

\begin{corollary}\label{cor:cutoff-remainders}
Under stability, let $\phi\in C_c^\infty((0,t_0))$ and suppose that
\[
 |\mathcal R|\leq C\int_M(1+\rho)
 \bigl[|\phi|^3(|D\phi|+|D^2\phi|)+\phi^2|D\phi|^2\bigr]\,d\mu^M_\sigma.
\]
Then, for every $\varepsilon>0$,
\[
 |\mathcal R|\leq\varepsilon\mathcal M_\sigma[\phi]
 +C_\varepsilon\bigl(\mathcal M_\sigma[|D\phi|]
                     +\mathcal M_\sigma[|D^2\phi|]\bigr).
\]
The constant $C_\varepsilon$ is uniform for $\sigma$ in a fixed compact
interval whenever the constant in the hypothesis is uniform.
\end{corollary}
\begin{proof}
Substitute \eqref{eq:ssy}, whose derivation also applies to signed $\phi$,
into \eqref{eq:cutoff-remainder-young} and choose the small parameter
sufficiently small.
\end{proof}

\subsection{Stability and two weighted identities}
We compute the stability form for
$f_1:=|\nabla G|G^{(1-\sigma)/2}\phi^2(G)$ and derive two further weighted
relations from Bochner's formula and divergence identities.
The quantities $\mathcal R_i$ depend on $\sigma$ and the cutoff $\phi$.
We suppress this dependence in the notation.

\begin{lemma}\label{lem:first-relations}
For every compactly supported cutoff as above, define
\begin{align}
 \mathcal R_1&:=\mathcal Q(f_1,f_1)
 +\int_M\bigl[\rho-|Z_\sigma\nu|^2-1/4\bigr]\phi^4\,d\mu^M_\sigma,
 \label{eq:early-stability}\\
 \mathcal R_2&:=\int_M\bigl[|Z_\sigma|^2-\rho|\widehat A\nu|^2
       +\sigma(1/2-\sigma/5)\bigr]\phi^4\,d\mu^M_\sigma,
 \label{eq:early-bochner}\\
 \mathcal R_3&:=\int_M\rho\bigl[\langle\widehat A^2,Z_\sigma\rangle
       +2(1-\sigma/5)|\widehat A\nu|^2-(1-2\sigma/5)\bigr]
       \phi^4\,d\mu^M_\sigma.\label{eq:early-curvature}
\end{align}
For $\phi\in C_c^\infty((0,t_0))$, these quantities satisfy, for every $\varepsilon>0$,
\begin{equation}\label{eq:early-source-bound}
 |\mathcal R_1|+|\mathcal R_2|+|\mathcal R_3|
 \leq\varepsilon\mathcal M_\sigma[\phi]
 +C_\varepsilon\bigl(\mathcal M_\sigma[|D\phi|]
                       +\mathcal M_\sigma[|D^2\phi|]\bigr),
\end{equation}
uniformly for $\sigma$ in a fixed compact interval.
\end{lemma}

\begin{remark}\label{rem:first-relations-motivation}
The quantity in \eqref{eq:early-stability} is defined using the stability
form evaluated at $f_1$.
Since $\mathcal Q(f_1,f_1)\geq0$, the resulting curvature bound contains
the Hessian term $|Z_\sigma\nu|^2$. Since $Z_\sigma$ is trace-free,
$|Z_\sigma\nu|^2\leq5|Z_\sigma|^2/6$, so we use Bochner's identity
\eqref{eq:early-bochner} to relate the full Hessian energy to curvature
and the Green moment. This introduces $\rho|\widehat A\nu|^2$, which we
eliminate using the divergence identity \eqref{eq:early-curvature} for $A^2$.
The remaining mixed term $\rho\langle\widehat A^2,Z_\sigma\rangle$ is
estimated by the Cauchy--Schwarz inequality and the SSY estimate.
\end{remark}

\begin{proof}[Proof of Lemma~\ref{lem:first-relations}]
\textbf{Step 1. Computation of $\mathcal R_1$.}
The divergence identity
\[
 \operatorname{div}(|\nabla G|^2G^{-\sigma}\nabla G)
 =2Z_\sigma(\nu,\nu)|\nabla G|^4G^{-\sigma-1}
\]
gives, for every smooth $v(G)$ supported in a compact region $\{a\leq G\leq b\}$ with $0<a<b<\infty$,
\begin{equation}\label{eq:radial-ibp}
 \int Z_\sigma(\nu,\nu)v\,d\mu^M_\sigma
 =-\frac12\int Dv\,d\mu^M_\sigma.
\end{equation}
For \eqref{eq:early-stability}, we compute $\mathcal Q(f_1,f_1)$.
By the definition of $Z_\sigma$,
\begin{equation}\label{eq:Green-gradient-norm}
 \nabla|\nabla G|=\nabla^2G(\nu,\cdot)
 =\frac{|\nabla G|^2}{G}
  \left(Z_\sigma\nu+\frac\sigma2\nu\right).
\end{equation}
The product rule therefore gives
\begin{equation}\label{eq:gradient-f1}
 \nabla f_1=|\nabla G|^2G^{-(1+\sigma)/2}
 \left[\phi^2\left(Z_\sigma\nu+\frac12\nu\right)+2\phi D\phi\,\nu\right].
\end{equation}
Using the definition of $\mathcal Q$ and recalling
$\rho=|A|^2G^2/|\nabla G|^2$, we obtain
\[
\begin{aligned}
 &\mathcal Q(f_1,f_1)+\int_M\bigl[\rho-|Z_\sigma\nu|^2-1/4\bigr]\phi^4\,d\mu^M_\sigma\\
 ={}&\int[(\phi^4+4\phi^3D\phi)Z_\sigma(\nu,\nu)
      +2\phi^3D\phi+4\phi^2(D\phi)^2]d\mu^M_\sigma.
\end{aligned}
\]
Apply \eqref{eq:radial-ibp} with $v=\phi^4+4\phi^3D\phi$.
We have $Dv=4\phi^3D\phi+12\phi^2(D\phi)^2+4\phi^3D^2\phi$.
Thus the quantity defined in \eqref{eq:early-stability} is given by
\[
 \mathcal R_1=-2\int_M[\phi^2(D\phi)^2+\phi^3D^2\phi]\,d\mu^M_\sigma.
\]

\medskip
\noindent\textbf{Step 2. Computation of $\mathcal R_2$.}
Bochner's formula and the traced Gauss equation give
\[
 \frac12\Delta |\nabla G|^2=|\nabla^2G|^2-|A|^2|\nabla G|^2|\widehat A\nu|^2.
\]
Multiplying by $G^{1-\sigma}\phi^4$ and integrating by parts twice gives
\[
\begin{aligned}
&\int G^{1-\sigma}\phi^4
 \bigl(|\nabla^2G|^2-|A|^2|\nabla G|^2|\widehat A\nu|^2\bigr)\,dV_g\;=\frac12\int G^{1-\sigma}\phi^4\,\Delta|\nabla G|^2\,dV_g\\
={} &-\frac12\int
 \langle\nabla(G^{1-\sigma}\phi^4),\nabla|\nabla G|^2\rangle\,dV_g\;=\frac12\int|\nabla G|^2\Delta(G^{1-\sigma}\phi^4)\,dV_g.
\end{aligned}
\]
Applying Lemma~\ref{lem:weighted-Laplacian} with $a=1-\sigma$ and $v=\phi^4$ gives
\[
 \Delta(G^{1-\sigma}\phi^4)
 =|\nabla G|^2G^{-1-\sigma}
 \bigl[D^2(\phi^4)+(1-2\sigma)D(\phi^4)+\sigma(\sigma-1)\phi^4\bigr].
\]
Thus, using $d\mu^M_\sigma=|\nabla G|^4G^{-\sigma-1}\,dV_g$,
we obtain
\[
\begin{aligned}
 &\phantom{{}={}}\int G^{1-\sigma}\phi^4
 \bigl(|\nabla^2G|^2-|A|^2|\nabla G|^2|\widehat A\nu|^2\bigr)\,dV_g\\
 &=\frac12\int[D^2(\phi^4)+(1-2\sigma)D(\phi^4)+\sigma (\sigma -1)\phi^4]\,d\mu^M_\sigma.
\end{aligned}
\]
Now $|Z_\sigma|^2=G^2|\nabla G|^{-4}|\nabla^2G|^2-(6\sigma /5)G|\nabla G|^{-2}\nabla^2G(\nu,\nu)+3\sigma ^2/10$.
Thus
\[
\begin{aligned}
 &\phantom{{}={}}\int\left(|Z_\sigma|^2-\rho |\widehat A\nu|^2
 +\frac{6\sigma}{5}Z_\sigma(\nu,\nu)+\frac{3\sigma^2}{10}\right)
 \phi^4\,d\mu^M_\sigma\\
 &=\frac12\int
 \left[D^2(\phi^4)+(1-2\sigma)D(\phi^4)+\sigma(\sigma-1)\phi^4\right]
 \,d\mu^M_\sigma.
\end{aligned}
\]
By \eqref{eq:radial-ibp}, the integral defined in
\eqref{eq:early-bochner} is given by
\begin{equation}\label{eq:sources}
 \mathcal R_2=\frac12\int_M
 [D^2(\phi^4)+(1-4\sigma/5)D(\phi^4)]\,d\mu^M_\sigma.
\end{equation}

\medskip
\noindent\textbf{Step 3. Computation of $\mathcal R_3$.}
The Codazzi equation and minimality give
$\operatorname{div}(A^2)=\nabla|A|^2/2$.
Integration by parts gives
\[
\begin{aligned}
&\frac12\int G^{2-\sigma}\phi^4
 \langle\nabla|A|^2,\nabla G\rangle\,dV_g
 =\int G^{2-\sigma}\phi^4\langle\operatorname{div}(A^2),\nabla G\rangle\,dV_g\\
={} &-\int G^{2-\sigma}\phi^4\langle A^2,\nabla^2G\rangle\,dV_g
 -\int\langle A^2\nabla G,\nabla(G^{2-\sigma}\phi^4)\rangle\,dV_g\\
={} &-\int\rho\Bigl\{\Bigl[\langle\widehat A^2,Z_\sigma\rangle
 +\frac\sigma{10}(6|\widehat A\nu|^2-1)\Bigr]\phi^4
 +|\widehat A\nu|^2[(2-\sigma)\phi^4+D(\phi^4)]\Bigr\}\,d\mu^M_\sigma\\
={} &-\int\rho\bigl[\langle\widehat A^2,Z_\sigma\rangle
 +2(1-\sigma/5)|\widehat A\nu|^2-\sigma/10\bigr]\phi^4\,d\mu^M_\sigma
 -\!\int\!\rho|\widehat A\nu|^2D(\phi^4)\,d\mu^M_\sigma.
\end{aligned}
\]
The divergence of $|A|^2G^{2-\sigma}\phi^4\nabla G$ also gives
\begin{equation}\label{eq:scalar-curvature-divergence}
\begin{aligned}
 \frac12\int G^{2-\sigma}\phi^4
 \langle\nabla|A|^2,\nabla G\rangle\,dV_g
 =\frac12\int\rho[(\sigma-2)\phi^4-D(\phi^4)]d\mu^M_\sigma.
\end{aligned}
\end{equation}
Equating the right-hand sides of the last two identities and collecting terms
shows that the integral defined in \eqref{eq:early-curvature} is given by
\begin{equation}\label{eq:curvature-direction-remainder}
 \mathcal R_3=\int_M\rho(1/2-|\widehat A\nu|^2)D(\phi^4)\,d\mu^M_\sigma.
\end{equation}

\medskip
\noindent\textbf{Step 4. Estimates for $\mathcal R_1,\mathcal R_2,\mathcal R_3$.}
Expand $D(\phi^4)$ and $D^2(\phi^4)$ in the expressions for
$\mathcal R_1$, $\mathcal R_2$, and $\mathcal R_3$ obtained in Steps 1--3.
Since $0\leq|\widehat A\nu|^2\leq1$, their sum
$|\mathcal R_1|+|\mathcal R_2|+|\mathcal R_3|$ satisfies the hypothesis of
Corollary~\ref{cor:cutoff-remainders}, which gives \eqref{eq:early-source-bound}.
\end{proof}

\begin{remark}\label{rem:exact-bochner-error}
The exact coefficients of $\mathcal R_2$ in \eqref{eq:sources} will be needed
in the rigidity argument in Section~\ref{sec:critical-rigidity}.
They are used in the proofs of Proposition~\ref{prop:energies} and
Lemma~\ref{lem:flux}.
\end{remark}

\subsection{Extension to exponent $2.46$}\label{sec:seed-section}
We combine the preceding estimates to prove $\mathcal M_{2.46}<\infty$.
For $0\leq\phi\in C_c^\infty((0,t_0))$, define the Hessian integral
\begin{equation}\label{eq:seed-Hessian-integral}
 \mathcal I_\sigma[\phi]:=\int_M|Z_\sigma|^2\phi^4\,d\mu^M_\sigma.
\end{equation}

\begin{proposition}\label{prop:ordinary-seed}
There is a constant $C>0$ such that
\begin{equation}\label{eq:cutoff-ineq}
 \mathcal M_\sigma[\phi]
 \leq C\bigl(\mathcal M_\sigma[|D\phi|]+\mathcal M_\sigma[|D^2\phi|]\bigr)
\end{equation}
for every $0\leq\phi\in C_c^\infty((0,t_0))$ and every
$\sigma\in[19/10,123/50]$.
In particular,
\begin{equation}\label{eq:ordinary-seed}
 \mathcal M_{123/50}=\mathcal M_{2.46}<\infty.
\end{equation}
\end{proposition}
\begin{proof}
By \eqref{eq:global-Newton},
\begin{equation}\label{eq:normalized-Newton}
 \left|\widehat A^2-\frac16g\right|^2
 =\frac{\operatorname{tr}(A^4)}{|A|^4}-\frac16
 \leq\frac7{10}-\frac16=\frac8{15}.
\end{equation}
Since $\operatorname{tr}Z_\sigma=0$,
$\langle Z_\sigma,\widehat A^2\rangle
=\langle Z_\sigma,\widehat A^2-g/6\rangle$.
By the Cauchy--Schwarz inequality, \eqref{eq:normalized-Newton}, and Young's inequality,
\begin{equation}\label{eq:seed-mixed-cauchy}
\begin{aligned}
 &\left|\int_M\rho\langle\widehat A^2,Z_\sigma\rangle\phi^4\,d\mu^M_\sigma\right|\\
 \leq{}&\left(\int_M\rho^2\phi^4\,d\mu^M_\sigma\right)^{1/2}
 \left(\int_M\left|\widehat A^2-\frac16g\right|^2
 |Z_\sigma|^2\phi^4\,d\mu^M_\sigma\right)^{1/2}\\
 \leq{}&\left(\frac8{15}\mathcal I_\sigma[\phi]
                   \int_M\rho^2\phi^4\,d\mu^M_\sigma\right)^{1/2}
 \leq\frac89\mathcal I_\sigma[\phi]
       +\frac3{20}\int_M\rho^2\phi^4\,d\mu^M_\sigma.
\end{aligned}
\end{equation}
Multiplying \eqref{eq:early-bochner} by $2(1-\sigma/5)$ and adding
\eqref{eq:early-curvature} cancels the terms involving $\rho|\widehat A\nu|^2$.
We estimate the mixed integral by \eqref{eq:seed-mixed-cauchy} and move its
Hessian term to the left. Controlling
$2(1-\sigma/5)\mathcal R_2+\mathcal R_3$ by
\eqref{eq:early-source-bound} gives
\begin{equation}\label{eq:seed-combined-identities}
\begin{aligned}
 &\left[2\left(1-\frac\sigma5\right)-\frac89\right]\mathcal I_\sigma[\phi]
 +\sigma\left(1-\frac\sigma5\right)\left(1-\frac{2\sigma}5\right)
 \mathcal M_\sigma[\phi]\\
 \leq{}&\left(1-\frac{2\sigma}5\right)\int_M\rho\phi^4\,d\mu^M_\sigma
 +\frac3{20}\int_M\rho^2\phi^4\,d\mu^M_\sigma
 +\varepsilon\mathcal M_\sigma[\phi]\\
 &+C_\varepsilon\bigl(\mathcal M_\sigma[|D\phi|]+\mathcal M_\sigma[|D^2\phi|]\bigr).
\end{aligned}
\end{equation}
Also, the zero trace and the Cauchy--Schwarz inequality imply
$|Z_\sigma\nu|^2\leq5|Z_\sigma|^2/6$.
Thus \eqref{eq:early-stability} and \eqref{eq:early-source-bound} give
\begin{equation}\label{eq:seed-curvature-bound}
\begin{aligned}
 \int_M\rho\phi^4\,d\mu^M_\sigma
 &\leq\frac56\mathcal I_\sigma[\phi]+\frac14\mathcal M_\sigma[\phi]
 +\varepsilon\mathcal M_\sigma[\phi]\\
 &\quad+C_\varepsilon\bigl(\mathcal M_\sigma[|D\phi|]+\mathcal M_\sigma[|D^2\phi|]\bigr).
\end{aligned}
\end{equation}
Substituting \eqref{eq:seed-curvature-bound} into \eqref{eq:physical-ssy}
and choosing the error parameters sufficiently small gives
\begin{equation}\label{eq:seed-ssy-substitution}
\begin{aligned}
 \int_M\rho^2\phi^4\,d\mu^M_\sigma
 &\leq(3-\sigma)(4-\sigma)
 \left(\frac5{12}\mathcal I_\sigma[\phi]
                    +\frac18\mathcal M_\sigma[\phi]\right)
 +\varepsilon\mathcal M_\sigma[\phi]\\
 &\quad+C_\varepsilon\bigl(\mathcal M_\sigma[|D\phi|]
                          +\mathcal M_\sigma[|D^2\phi|]\bigr).
\end{aligned}
\end{equation}
Substituting \eqref{eq:seed-curvature-bound} and
\eqref{eq:seed-ssy-substitution} into \eqref{eq:seed-combined-identities},
collecting coefficients, and moving the remaining principal terms to the
left gives, after relabelling $\varepsilon$,
\begin{equation}\label{eq:seed-scalar}
\begin{aligned}
 &\left(-\frac{\sigma^2}{16}+\frac{89\sigma}{240}-\frac{17}{36}\right)
   \mathcal I_\sigma[\phi]
 +\left(\frac{2\sigma^3}{25}-\frac{99\sigma^2}{160}
                  +\frac{197\sigma}{160}-\frac{19}{40}\right)
   \mathcal M_\sigma[\phi]\\
 \leq{}&\varepsilon\mathcal M_\sigma[\phi]
 +C_\varepsilon\bigl(\mathcal M_\sigma[|D\phi|]
                    +\mathcal M_\sigma[|D^2\phi|]\bigr).
\end{aligned}
\end{equation}
Direct computation shows that, on $[19/10,123/50]$, the first coefficient
on the left of \eqref{eq:seed-scalar} is increasing and positive at $19/10$,
while the second is decreasing and positive at $123/50$.
Both therefore have uniform positive lower bounds.
Choosing $\varepsilon$ sufficiently small and absorbing the moment error gives
\[
 \mathcal I_\sigma[\phi]+\mathcal M_\sigma[\phi]
 \leq C\bigl(\mathcal M_\sigma[|D\phi|]+\mathcal M_\sigma[|D^2\phi|]\bigr).
\]
This proves \eqref{eq:cutoff-ineq}.
Finally, \eqref{eq:initial-seed} gives $\mathcal M_{19/10}<\infty$.
Applying Corollary~\ref{cor:continuation} on $[19/10,123/50]$ proves
\eqref{eq:ordinary-seed}.
\end{proof}

\section{Further weighted identities and spectral estimates}\label{sec:compact-system}
\subsection{Further weighted identities}\label{subsec:compact-identities}
We derive further weighted identities for the nonlinear stability argument, using the notation and cutoff conventions of Section~\ref{sec:first-estimates} and the zero-set interpretation in Remark~\ref{rem:zeros}.
The integral identities in this section use compactly supported test functions.
The Codazzi equation and minimality make $\nabla A$ fully symmetric and trace-free.

For the integrands appearing in the weighted identities of
Section~\ref{sec:first-estimates}, we write
\begin{align*}
 F_1&:=\rho-|Z_\sigma\nu|^2-1/4,\\
 F_2&:=|Z_\sigma|^2-\rho|\widehat A\nu|^2+\sigma(5-2\sigma)/10,\\
 F_3&:=\rho\bigl[\langle\widehat A^2,Z_\sigma\rangle
       +2(1-\sigma/5)|\widehat A\nu|^2-(1-2\sigma/5)\bigr],\\
 F_4&:=Z_\sigma(\nu,\nu).
\end{align*}
At $\sigma=5/2$, the weighted Bochner identity associated with $F_2$
is a reformulation of Colding's identity~\cite[Lemma~2.7]{Colding2012}.
Using \eqref{eq:early-stability}--\eqref{eq:early-curvature}, we rewrite
the previously defined quantities as
\begin{align}
 \mathcal R_1&=\mathcal Q(f_1,f_1)+\int_M F_1\phi^4\,d\mu^M_\sigma,
 \label{eq:F1-stability}\\
 \mathcal R_i&=\int_M F_i\phi^4\,d\mu^M_\sigma\qquad\text{for }i=2,3.
 \label{eq:F23-integrals}
\end{align}
We also define
\begin{equation}\label{eq:F4-integral}
 \mathcal R_4:=\int_M F_4\phi^4\,d\mu^M_\sigma,
\end{equation}
which can also be expressed in terms of the cutoff derivative by
\eqref{eq:radial-ibp} as
$\mathcal R_4=-\frac12\int_M D(\phi^4)\,d\mu^M_\sigma$.
Applying Corollary~\ref{cor:cutoff-remainders} to $\mathcal R_4$ using
\eqref{eq:radial-ibp}, and combining with \eqref{eq:early-source-bound}, gives
\begin{equation}\label{eq:first-four-source-bound}
 \sum_{i=1}^4|\mathcal R_i|
 \leq\varepsilon\mathcal M_\sigma[\phi]
 +C_\varepsilon\bigl(\mathcal M_\sigma[|D\phi|]+\mathcal M_\sigma[|D^2\phi|]\bigr).
\end{equation}
For the curvature identities, set
\begin{equation}\label{eq:stability-test-functions}
 f_2:=|A|G^{(3-\sigma)/2}\phi^2(G).
\end{equation}

The product identity \eqref{eq:curvature-product} gives the stability
relation for $f_2$. Simons' identity \eqref{eq:Simons-identity} and the scalar divergence identity
\eqref{eq:scalar-curvature-divergence} give two further weighted relations.

\begin{lemma}\label{lem:compact}
Define
\begin{align*}
 F_5&:=\frac{\rho G^2}{|A|^2|\nabla G|^2}\bigl(|\nabla A|^2-|\nabla|A||^2\bigr)-\frac14(3-\sigma)^2\rho,\\
 F_6&:=\rho^2-\rho\frac{G^2|\nabla A|^2}{|A|^2|\nabla G|^2}-\frac12(3-\sigma)(\sigma-2)\rho,\\
 F_7&:=\rho\left(\frac{G\langle\nabla|A|,\nu\rangle}{|A||\nabla G|}-\frac{\sigma-2}{2}\right).
\end{align*}
For every compactly supported cutoff $\phi$, define
\begin{align}
 \mathcal R_5&:=\mathcal Q(f_2,f_2)+\int_M F_5\phi^4\,d\mu^M_\sigma,
 \label{eq:F5-stability}\\
 \mathcal R_i&:=\int_M F_i\phi^4\,d\mu^M_\sigma\qquad\text{for }i=6,7.
 \label{eq:F67-integrals}
\end{align}
For $0\leq\phi\in C_c^\infty((0,t_0))$ and every $\varepsilon>0$,
\begin{equation}\label{eq:additional-source-bound}
 \sum_{i=5}^7|\mathcal R_i|
 \leq\varepsilon\mathcal M_\sigma[\phi]
 +C_\varepsilon\bigl(\mathcal M_\sigma[|D\phi|]+\mathcal M_\sigma[|D^2\phi|]\bigr),
\end{equation}
uniformly for $\sigma$ in a fixed compact interval.
\end{lemma}
\begin{proof}
\textbf{The stability relation for $F_5$}. Apply \eqref{eq:curvature-product}
with $h:=G^{(3-\sigma)/2}\phi^2$, so that $|A|h=f_2$.
Using \eqref{eq:ssy-test-gradient} and
$|A|^2|\nabla G|^2G^{1-\sigma}\,dV_g=\rho\,d\mu^M_\sigma$, we obtain
\[
\begin{aligned}
 \mathcal Q(f_2,f_2)
 ={}&\int_M\rho\left[\frac{(3-\sigma)^2}{4}\phi^4
 +2(3-\sigma)\phi^3D\phi+4\phi^2(D\phi)^2\right]d\mu^M_\sigma\\
 &-\int_MG^{3-\sigma}\phi^4
 \bigl(|\nabla A|^2-|\nabla|A||^2\bigr)\,dV_g.
\end{aligned}
\]
By the definition of $F_5$, the terms containing $\phi^4$ combine to
$-\int_MF_5\phi^4\,d\mu^M_\sigma$.
Moving this integral to the left and using \eqref{eq:F5-stability} gives
\begin{equation}\label{eq:F5-remainder}
 \mathcal R_5
 =2(3-\sigma)\int\rho\phi^3D\phi\,d\mu^M_\sigma
 +4\int\rho\phi^2(D\phi)^2\,d\mu^M_\sigma.
\end{equation}
\vskip.2cm
\textbf{Simons' identity for $F_6$}. Multiplying Simons' identity \eqref{eq:Simons-identity}
by $G^{3-\sigma}\phi^4$ and integrating by parts twice gives
\[
 \int G^{3-\sigma}\phi^4\bigl(|\nabla A|^2-|A|^4\bigr)\,dV_g
 =\frac12\int|A|^2\Delta(G^{3-\sigma}\phi^4)\,dV_g.
\]
Substituting \eqref{eq:ssy-test-Laplacian} and recalling
$\rho=|A|^2G^2/|\nabla G|^2$, we find that the integral defined in
\eqref{eq:F67-integrals} is given by
\begin{equation}\label{eq:simons-remainder}
 \mathcal R_6=-\frac12\int\rho
 [D^2(\phi^4)+(5-2\sigma)D(\phi^4)]\,d\mu^M_\sigma.
\end{equation}
\vskip.2cm
\textbf{The divergence identity for $F_7$}. Rewriting
\eqref{eq:scalar-curvature-divergence} using $\rho$ and $\nu$ gives the following
expression for the integral defined in \eqref{eq:F67-integrals}
\begin{equation}\label{eq:curv-radial}
 \mathcal R_7=\int\rho\left(\frac{G\langle\nabla|A|,\nu\rangle}{|A||\nabla G|}-\frac{\sigma-2}{2}\right)
 \phi^4\,d\mu^M_\sigma
 =-\frac12\int\rho D(\phi^4)\,d\mu^M_\sigma.
\end{equation}

The formulas \eqref{eq:F5-remainder}, \eqref{eq:simons-remainder}, and
\eqref{eq:curv-radial} imply
\begin{equation}\label{eq:curvature-source-coarse}
 \sum_{i=5}^7|\mathcal R_i|
 \leq C\int\rho\bigl(
 |\phi|^3|D\phi|+\phi^2(D\phi)^2+|\phi|^3|D^2\phi|\bigr)\,d\mu^M_\sigma,
\end{equation}
with $C$ uniform for $\sigma$ in a fixed compact interval.
The sum in \eqref{eq:curvature-source-coarse} satisfies the hypothesis of
Corollary~\ref{cor:cutoff-remainders}, which gives \eqref{eq:additional-source-bound}.
\end{proof}
For $i\in\{2,3,4,6,7\}$, the identities also hold with $\phi^4$
replaced throughout by any $w\in C_c^\infty((0,\infty))$, possibly signed,
evaluated at $G$. Their cutoff terms are obtained by replacing
$D(\phi^4)$ and $D^2(\phi^4)$ by $Dw$ and $D^2w$ in the corresponding proofs.

\begin{remark}\label{rem:zeros}
The approximation argument in Section~\ref{sec:first-estimates}
also applies to these identities across the zero sets. After substituting the definitions of $\rho$ and
$d\mu^M_\sigma$, the weighted derivative densities can be written
without division by $|A|$ or $|\nabla G|$. The identities are derived on compact regions $\{a\leq G\leq b\}$ with $0<a<b<\infty$
and require only local integrability.
\end{remark}

\subsection{Spectral estimates}\label{subsec:compact-estimates}
The following lemma estimates $|\widehat A^2-g/6|^2$ when one normalized
principal curvature is prescribed, refining \eqref{eq:normalized-Newton}.
By the Gauss equation, $\widehat A^2-g/6$ is the negative of the
trace-free Ricci tensor divided by $|A|^2$.

\begin{lemma}\label{lem:sharp-spectra}
If $\sum_{j=1}^5z_j=0$, then
\begin{equation}\label{eq:five-moments}
 \left|\sum_{j=1}^5z_j^3\right|
 \leq\frac3{\sqrt{20}}\left(\sum_{j=1}^5z_j^2\right)^{3/2},\qquad
 \sum_{j=1}^5z_j^4\leq\frac{13}{20}\left(\sum_{j=1}^5z_j^2\right)^2.
\end{equation}
If $q\geq0$ is an eigenvalue of $\widehat A$, then $q\leq\sqrt{5/6}$ and
\begin{equation}\label{eq:sharp-envelope}
 |(\widehat A^2-\frac16g)|^2\leq\Lambda(q):=\frac{207}{125}q^4-\frac{33}{25}q^2+\frac{29}{60}
 +\frac6{125}q(5-6q^2)^{3/2}.
\end{equation}
The estimate also applies with $q=|\widehat\lambda_i|$ for any eigenvalue
$\widehat\lambda_i$, by replacing $\widehat A$ with $-\widehat A$ if necessary.
\end{lemma}
\begin{proof}
If $\sum_{j=1}^5z_j^2=0$, then all $z_j$ vanish. Otherwise, by homogeneity,
it suffices to assume $\sum_{j=1}^5z_j^2=1$.

The set $\{z\in\mathbb R^5:\sum_{j=1}^5z_j=0,\ \sum_{j=1}^5z_j^2=1\}$ is compact,
so $\sum_{j=1}^5z_j^3$ attains its extrema. The gradients of the two constraints
are independent there. At each extremum, Lagrange multipliers give
\[
 3z_i^2=2\alpha z_i+\gamma\qquad\text{for }1\leq i\leq5.
\]
Every coordinate is a root of the same quadratic, so there are at most two
coordinate values. Let the two possible values be $a,b$, with multiplicities
$j,5-j$. The constraints give
\[
 ja+(5-j)b=0,\qquad ja^2+(5-j)b^2=1.
\]
In particular, $j\notin\{0,5\}$. Solving gives
$a=\pm\sqrt{(5-j)/(5j)}$ and $b=\mp\sqrt{j/[5(5-j)]}$.
Substitution into the cubic sum gives
\[
 \left|ja^3+(5-j)b^3\right|
 =\frac{|5-2j|}{\sqrt{5j(5-j)}}
 =\begin{cases}
 3/\sqrt{20},&j=1,4,\\
 1/\sqrt{30},&j=2,3.
 \end{cases}
\]
By homogeneity, this proves the cubic estimate in \eqref{eq:five-moments}.
The quartic estimate follows from
Remark~\ref{rem:algebraic-quartic} with $n=5$ and homogeneity.

For \eqref{eq:sharp-envelope}, relabel the selected nonnegative eigenvalue as
$\widehat\lambda_6=q$. Subtract the mean of the remaining eigenvalues by writing
$z_j:=\widehat\lambda_j+q/5$ for $1\leq j\leq5$.
The two eigenvalue constraints then give
\[
 \sum_{j=1}^5z_j=0,\qquad
 |z|^2:=\sum_{j=1}^5z_j^2=1-\frac65q^2\geq0,
 \qquad 0\leq q\leq\sqrt{5/6}.
\]
Expanding the fourth powers and using \eqref{eq:five-moments} and
$|z|^2=1-6q^2/5$, we obtain
\[
\begin{aligned}
 &\left|\widehat A^2-\frac16g\right|^2
 =q^4+\sum_{j=1}^5\left(z_j-\frac q5\right)^4-\frac16\\
 ={}&\frac{126}{125}q^4-\frac16+\sum_{j=1}^5z_j^4
 -\frac{4q}{5}\sum_{j=1}^5z_j^3+\frac{6q^2}{25}|z|^2\\
 \leq{}&\frac{126}{125}q^4-\frac16+\frac{13}{20}|z|^4
 +\frac{12q}{5\sqrt{20}}|z|^3+\frac{6q^2}{25}|z|^2\\
 ={}&\frac{126}{125}q^4-\frac16
 +\frac{13}{20}\left(1-\frac65q^2\right)^2
 +\frac{12q}{5\sqrt{20}}\left(1-\frac65q^2\right)^{3/2}
 +\frac{6q^2}{25}\left(1-\frac65q^2\right)\\
 ={}&\frac{207}{125}q^4-\frac{33}{25}q^2+\frac{29}{60}
 +\frac6{125}q(5-6q^2)^{3/2}=\Lambda(q).
\end{aligned}
\]
\end{proof}

\section{Further moment improvement via nonlinear stability}\label{sec:nonlinear-moments}
\subsection{A nonlinear stability inequality}\label{sec:nonlinear-stability}
Recall the test functions
\[
 f_1=|\nabla G|G^{(1-\sigma)/2}\phi^2(G),\qquad
 f_2=|A|G^{(3-\sigma)/2}\phi^2(G).
\]
They satisfy $f_2=\sqrt\rho\,f_1$, where
$\rho=|A|^2G^2/|\nabla G|^2$.
The identities \eqref{eq:early-stability} and \eqref{eq:F5-stability} contain the terms $\mathcal Q(f_1,f_1)$ and $\mathcal Q(f_2,f_2)$, respectively.
We combine these functions in a single nonlinear test function.
For $\mathbf f:=(f_1,f_2)^{\mathsf T}$, choose a positive definite matrix
\begin{equation}\label{eq:stability-matrix}
 \mathsf M:=\begin{pmatrix}m_{11}&m_{12}\\m_{12}&m_{22}\end{pmatrix}>0,
 \qquad \det\mathsf M=m_{11}m_{22}-m_{12}^2.
\end{equation}
Define the nonnegative function $\mathcal F$ by
\[
 \mathcal F:=\sqrt{\mathbf f^{\mathsf T}\mathsf M\mathbf f}
 =\sqrt{m_{11}f_1^2+2m_{12}f_1f_2+m_{22}f_2^2}.
\]
Recall the stability form
\[
 \mathcal Q(f,g)=\int_M(\langle\nabla f,\nabla g\rangle-|A|^2fg)\,dV_g.
\]
Direct computation gives
\begin{equation}\label{eq:norm-id}
 \sum_{i,j=1}^2m_{ij}\mathcal Q(f_i,f_j)
 =\mathcal Q(\mathcal F,\mathcal F)
 +\int_M\frac{\det\mathsf M\,|f_2\nabla f_1-f_1\nabla f_2|^2}{\mathcal F^2}\,dV_g.
\end{equation}
The quotient is defined to be zero where $f_2=f_1=0$.
The weak gradients of $f_2,f_1$ and $\mathcal F$ vanish almost everywhere on this set.
Since $\mathcal F$ is compactly supported and locally Lipschitz, stability gives
$\mathcal Q(\mathcal F,\mathcal F)\geq0$. Thus the left side of \eqref{eq:norm-id}
is bounded below by the integral on its right side.

In vector expressions, we regard $Z_\sigma$ as the self-adjoint endomorphism
associated with the symmetric tensor by the metric $g$.
\begin{proposition}\label{prop:normalized-square}
For the test functions $f_1,f_2$ and matrix $\mathsf M$ from
\eqref{eq:stability-matrix}, one has
\begin{equation}\label{eq:gain}
 \int_M\frac{|f_2\nabla f_1-f_1\nabla f_2|^2}
 {\mathcal F^2}\,dV_g
 =\int_M\Gamma_\sigma\phi^4\,d\mu^M_\sigma,
\end{equation}
where
\[
 \Gamma_\sigma:=\frac{\rho}{m_{11}+2m_{12}\sqrt\rho+m_{22}\rho}
 \left|\frac{G\nabla|A|}{|A||\nabla G|}-Z_\sigma\nu
 +\left(1-\frac\sigma2\right)\nu\right|^2.
\]
\end{proposition}
\begin{proof}
Where $f_1f_2>0$, we have $f_1/f_2=|\nabla G|/(G|A|)$.
Using \eqref{eq:Green-gradient-norm}, we compute
\[
\begin{aligned}
 f_2\nabla f_1-f_1\nabla f_2
 &=f_1f_2\nabla\log\frac{f_1}{f_2}
 =f_1f_2\left(\frac{\nabla|\nabla G|}{|\nabla G|}
 -\frac{\nabla G}{G}-\frac{\nabla|A|}{|A|}\right)\\
 &=-f_1f_2\frac{|\nabla G|}{G}
 \left[\frac{G\nabla|A|}{|A||\nabla G|}-Z_\sigma\nu
 +\left(1-\frac\sigma2\right)\nu\right].
\end{aligned}
\]
Since $f_2^2=\rho f_1^2$ and
$f_1^2|\nabla G|^2G^{-2}\,dV_g=\phi^4\,d\mu^M_\sigma$,
squaring this identity and dividing by $\mathcal F^2$ gives \eqref{eq:gain}.
The weighted identity extends almost everywhere by the zero-set conventions
above.
\end{proof}

To handle the mixed terms arising when $\mathcal F$ is used as a test function
in the stability inequality, we define the mixed integrands
\begin{align*}
 F_8&:=\sqrt\rho\,\biggl[\frac{G\langle\nabla|A|,\nu\rangle}{|A||\nabla G|}+Z_\sigma(\nu,\nu)+\left(1-\frac\sigma2\right)\biggr],\\
 F_9&:=\rho^{3/2}-\sqrt\rho\,\left\langle\frac{G\nabla|A|}{|A||\nabla G|}+\frac{3-\sigma}{2}\nu,Z_\sigma\nu+\frac12\nu\right\rangle.
\end{align*}

\begin{proposition}\label{prop:mixed-identities}
For $0\leq\phi\in C_c^\infty((0,t_0))$, define
\begin{align}
 \mathcal R_8&:=\int_M F_8\phi^4\,d\mu^M_\sigma,
 \label{eq:mixdivsource}\\
 \mathcal R_9&:=\mathcal Q(f_1,f_2)+\int_M F_9\phi^4\,d\mu^M_\sigma.
 \label{eq:mixedsource}
\end{align}
For $\sigma$ in a fixed compact interval and every $\varepsilon>0$,
\begin{equation}\label{eq:mixed-source-bound}
 |\mathcal R_8|+|\mathcal R_9|
 \leq\varepsilon\mathcal M_\sigma[\phi]
 +C_\varepsilon\bigl(\mathcal M_\sigma[|D\phi|]
 +\mathcal M_\sigma[|D^2\phi|]\bigr).
\end{equation}
\end{proposition}
\begin{proof}
\textbf{Step 1. The identity for $F_8$.}
Let $v=v(G)$ be any smooth function supported in a compact interval of positive
Green values. Since $\Delta G=0$ on its support, the product rule and
\eqref{eq:Green-gradient-norm} give
\[
\begin{aligned}
 &\operatorname{div}\bigl(|A||\nabla G|G^{1-\sigma}v\nabla G\bigr)\\
 ={}&|\nabla G|G^{1-\sigma}v\langle\nabla|A|,\nabla G\rangle
 +|A|G^{1-\sigma}v\langle\nabla|\nabla G|,\nabla G\rangle\\
 &{}+(1-\sigma)|A||\nabla G|^3G^{-\sigma}v
 +|A||\nabla G|G^{1-\sigma}\langle\nabla v,\nabla G\rangle\\
 ={}&|A||\nabla G|^3G^{-\sigma}
 \biggl[\biggl(\frac{G\langle\nabla|A|,\nu\rangle}{|A||\nabla G|}
 +Z_\sigma(\nu,\nu)+1-\frac\sigma2\biggr)v+Dv\biggr].
\end{aligned}
\]
Using $\sqrt\rho\,d\mu^M_\sigma=|A||\nabla G|^3G^{-\sigma}\,dV_g$,
we obtain the identity
\[
 \int_M\sqrt\rho\,\biggl[\frac{G\langle\nabla|A|,\nu\rangle}{|A||\nabla G|}+Z_\sigma(\nu,\nu)+\left(1-\frac\sigma2\right)\biggr]v\,d\mu^M_\sigma
 =-\int_M\sqrt\rho\,Dv\,d\mu^M_\sigma.
\]
Taking $v=\phi^4$, we find that the integral defined in \eqref{eq:mixdivsource}
is given by
\[
 \mathcal R_8=-\int_M\sqrt\rho\,D(\phi^4)\,d\mu^M_\sigma.
\]

\medskip
\noindent\textbf{Step 2. The identity for $F_9$.}
Recalling \eqref{eq:gradient-f1} and differentiating $f_2$, we obtain
\begin{align}
 \nabla f_1
 &=|\nabla G|^2G^{-(1+\sigma)/2}
 \left[\phi^2\left(Z_\sigma \nu+\frac12\nu\right)+2\phi D\phi\,\nu\right],\notag\\
 \nabla f_2
 &=|A||\nabla G|G^{(1-\sigma)/2}
 \left[\phi^2\left(\frac{G\nabla|A|}{|A||\nabla G|}+\frac{3-\sigma}{2}\nu\right)+2\phi D\phi\,\nu\right].
 \label{eq:gradient-f2}
\end{align}
Since $|A||\nabla G|^3G^{-\sigma}\,dV_g=\sqrt\rho\,d\mu^M_\sigma$, we obtain
\[
\begin{aligned}
 &\mathcal Q(f_1,f_2)
 =\int_M\bigl(\langle\nabla f_1,\nabla f_2\rangle-|A|^2f_1f_2\bigr)\,dV_g\\
 ={}&\int_M\biggl\{\biggl[\sqrt\rho\left\langle Z_\sigma\nu+\frac12\nu,
 \frac{G\nabla|A|}{|A||\nabla G|}+\frac{3-\sigma}{2}\nu\right\rangle
 -\rho^{3/2}\biggr]\phi^4\\
 &{}+2\sqrt\rho\,\phi^3D\phi
 \biggl(Z_\sigma(\nu,\nu)+\frac{G\langle\nabla|A|,\nu\rangle}{|A||\nabla G|}
 +\frac{4-\sigma}{2}\biggr)
 +4\sqrt\rho\,\phi^2(D\phi)^2\biggr\}\,d\mu^M_\sigma\\
 ={}&-\int_M F_9\phi^4\,d\mu^M_\sigma
 +2\int_M\bigl[F_8\phi^3D\phi+\sqrt\rho\bigl(\phi^3D\phi+2\phi^2(D\phi)^2\bigr)\bigr]\,d\mu^M_\sigma.
\end{aligned}
\]
Applying the identity from Step~1 with $v=2\phi^3D\phi$ to the $F_8$ term
gives an expression for the quantity defined in \eqref{eq:mixedsource},
in which every term involves $D\phi$ or $D^2\phi$.
Together with the expression for $\mathcal R_8$, this gives
\[
 |\mathcal R_8|+|\mathcal R_9|
 \leq C\int_M\sqrt\rho\,
 \bigl(\phi^3|D\phi|+\phi^2|D\phi|^2+\phi^3|D^2\phi|\bigr)\,d\mu^M_\sigma.
\]
Since $\sqrt\rho\leq1+\rho$, Corollary~\ref{cor:cutoff-remainders}
gives \eqref{eq:mixed-source-bound}.
\end{proof}

For real coefficients $c_3,c_4,c_6,c_7,c_8$ and the matrix entries $m_{11},m_{12},m_{22}$
from \eqref{eq:norm-id}, we use $c_i$ as the coefficient of $F_i$
except for the fixed $F_2$ coefficient and the matrix terms, and define the combined stability expression
\begin{equation}\label{eq:nonlinear-combination}
 \begin{split}
 \mathscr P_\sigma:={}&F_2+c_3F_3+c_4F_4+c_6F_6+c_7F_7+c_8F_8\\
 &+m_{11}F_1+2m_{12}F_9+m_{22}F_5+(\det\mathsf M)\Gamma_\sigma.
 \end{split}
\end{equation}

For \eqref{eq:nonlinear-combination}, the preceding identities and stability
give the following integral estimate.
\begin{proposition}\label{prop:nonlinear-integral}
For every $0\leq\phi\in C_c^\infty((0,t_0))$ and $\varepsilon>0$, stability gives
\begin{equation}\label{eq:sourcePnl}
 \int_M \mathscr P_\sigma\phi^4\,d\mu^M_\sigma
 \leq\varepsilon\mathcal M_\sigma[\phi]
 +C_\varepsilon\bigl(\mathcal M_\sigma[|D\phi|]
 +\mathcal M_\sigma[|D^2\phi|]\bigr).
\end{equation}
The constant is uniform when $\sigma$ ranges over a fixed compact interval
and the coefficients remain bounded.
\end{proposition}
\begin{proof}
Combining \eqref{eq:F1-stability}--\eqref{eq:F4-integral},
\eqref{eq:F5-stability}, \eqref{eq:F67-integrals},
\eqref{eq:mixdivsource}, and \eqref{eq:mixedsource} with
\eqref{eq:norm-id} and \eqref{eq:gain} cancels the $(\det\mathsf M)\Gamma_\sigma$ integral
and leaves the cutoff terms minus $\mathcal Q(\mathcal F,\mathcal F)$.
Stability gives $\mathcal Q(\mathcal F,\mathcal F)\geq0$, and
\eqref{eq:first-four-source-bound}, \eqref{eq:additional-source-bound}, and
\eqref{eq:mixed-source-bound} give the estimate after choosing the error
parameters sufficiently small.
\end{proof}

To obtain \eqref{eq:abstract-cutoff-ineq}, it remains to choose the coefficients so
that $\mathscr P_\sigma$ has a uniform positive lower bound, as established in
Subsections~\ref{sec:nonlinear-reduction}--\ref{sec:nonlinear-coercivity}.

\subsection{Reduction to a scalar inequality}\label{sec:nonlinear-reduction}
We define a scalar function of the normalized principal curvature and $\rho$
and prove a lower bound for $\mathscr P_\sigma$ in terms of this function.

Throughout this subsection, assume
\begin{equation}\label{eq:matrix-signs}
\begin{gathered}
 m_{22}>4c_{6}>0,\quad 0<m_{11}<4/5,\quad m_{12}\geq0,\quad \det\mathsf M>0,\\
 m_{12}^2<\frac{m_{22}-4c_{6}}{3}\left(\frac65-m_{11}\right).
\end{gathered}
\end{equation}
We use $q\in[0,\sqrt{5/6}]$ as an independent scalar parameter for the
absolute value of a normalized principal curvature.
With the matrix $\mathsf M$ from \eqref{eq:stability-matrix}, we introduce the following notation:
\[
\begin{aligned}
 \beta(\rho)&:=\frac{\det\mathsf M}{m_{11}+2m_{12}\sqrt\rho+m_{22}\rho},\\
 \widetilde m_{11}(\rho)&:=m_{11}-\rho\beta(\rho),\qquad&
 \widetilde m_{12}(\rho)&:=m_{12}+\sqrt\rho\beta(\rho),\\
 \alpha(q,\rho)&:=\frac{m_{22}-c_6}{\frac13+\frac{q^2}2}-m_{22}+\beta(\rho),&
 \eta(q,\rho)&:=\frac65-\widetilde m_{11}(\rho)
              -\frac{\widetilde m_{12}(\rho)^2}{\alpha(q,\rho)}.
\end{aligned}
\]
Define the scalar coefficients by
\begin{equation}\label{eq:nonlinear-coefficients}
\begin{aligned}
 B_0(u):={}&\frac{\sigma(5-2\sigma)}{10}-\frac{m_{11}}4
 +c_6\rho^2+2m_{12}\rho^{3/2}\\
 &{}+\left[c_8\left(1-\frac\sigma2\right)-\frac{m_{12}(3-\sigma)}2\right]\sqrt\rho
 +\rho\beta(\rho)\left(1-\frac\sigma2\right)^2\\
 &{}+\rho\biggl[\left(2c_3\left(1-\frac\sigma5\right)-1\right)u
 -c_3\left(1-\frac{2\sigma}5\right)+m_{11}\\
 &\qquad{}-\frac{c_6}2(3-\sigma)(\sigma-2)
 -\frac{m_{22}}4(3-\sigma)^2-\frac{c_7}2(\sigma-2)\biggr],\\
 B_1:={}&c_7\sqrt\rho+c_8-m_{12}
 +2\sqrt\rho\beta(\rho)\left(1-\frac\sigma2\right),\\
 B_2:={}&c_4+\bigl(c_8-m_{12}(3-\sigma)\bigr)\sqrt\rho
 -2\rho\beta(\rho)\left(1-\frac\sigma2\right).
\end{aligned}
\end{equation}
For $q=|\widehat\lambda_i|\in[0,\sqrt{5/6}]$, Lemma~\ref{lem:sharp-spectra} gives
\begin{equation}\label{eq:envelope}
 \operatorname{tr} \widehat A^4-\frac16\leq\Lambda(q)
 =\frac{207}{125}q^4-\frac{33}{25}q^2+\frac{29}{60}
  +\frac6{125}q(5-6q^2)^{3/2}.
\end{equation}
Define the scalar function
\begin{align}\label{eq:scalarF}
 \mathscr F_\sigma(q,\rho):={}& B_0(q^2)-\frac{B_1^2}{4\alpha(q,\rho)}
 -\frac{c_{3}^2}4\left[\Lambda(q)-\frac65(q^2-\frac16)^2\right]\rho^2\notag\\
 &-\frac{\left[\frac65c_{3}(q^2-\frac16)\rho+B_2+\frac{\widetilde m_{12}(\rho) B_1}{\alpha(q,\rho)}\right]^2}{4\eta(q,\rho)}.
\end{align}
The denominators in \eqref{eq:scalarF} are positive by
Lemma~\ref{prop:nonlinear-denominators} below.

\begin{proposition}\label{prop:nonlinear-all-directions}
Let $(E_i)$ be an orthonormal eigenbasis of $\widehat A$, with eigenvalues
$\widehat\lambda_i$, and write $\nu_i:=\langle\nu,E_i\rangle$ for the
components of the geometric direction $\nu$. Then
\begin{equation}\label{eq:all-directions}
 \mathscr P_\sigma\geq\sum_i\nu_i^2\,
       \mathscr F_\sigma\bigl(|\widehat\lambda_i|,\rho\bigr).
\end{equation}
\end{proposition}

\begin{remark}
It suffices to choose bounded coefficients on $[\sigma_0,\sigma_1]$ such that,
for some uniform $m>0$,
\[
 \mathscr F_\sigma(q,\rho)\geq m(1+\rho+\rho^2)
 \qquad\text{for }0\leq q\leq\sqrt{5/6}\text{ and }\rho\geq0.
\]
Proposition~\ref{prop:nonlinear-all-directions} and \eqref{eq:sourcePnl},
with $\varepsilon<m$, then give \eqref{eq:abstract-cutoff-ineq} uniformly
on the interval. Corollary~\ref{cor:continuation} therefore extends
$\mathcal M_{\sigma_0}<\infty$ to $\mathcal M_\sigma<\infty$ for every
$\sigma\in[\sigma_0,\sigma_1]$.
\end{remark}

\begin{lemma}\label{prop:nonlinear-denominators}
Under \eqref{eq:matrix-signs}, for every $0\leq q\leq\sqrt{5/6}$ and $\rho\geq0$,
\[
 \alpha(q,\rho)>0,\qquad \eta(q,\rho)>0.
\]
\end{lemma}
\begin{proof}
Set $v:=(1,-\sqrt\rho)^{\mathsf T}$. Since
$1/3\leq 1/3+q^2/2\leq3/4$, the sign assumptions make the first matrix below
positive definite, while the second is positive semidefinite. Their sum satisfies
\[
 \mathsf S:=
 \begin{pmatrix}\frac{m_{22}-c_6}{\frac13+\frac{q^2}2}-m_{22}&-m_{12}\\
 -m_{12}&6/5-m_{11}\end{pmatrix}
 +\beta(\rho)vv^{\mathsf T}>0.
\]
Indeed, the first diagonal entry is at least $(m_{22}-4c_6)/3$,
and its determinant is positive by \eqref{eq:matrix-signs}.
The definitions give $\mathsf S_{11}=\alpha(q,\rho)$ and
$\det\mathsf S=\alpha(q,\rho)\eta(q,\rho)$, proving both assertions.
Here $\eta(q,\rho)$ is the Schur complement of $\alpha(q,\rho)$ in $\mathsf S$.
\end{proof}

In the chosen eigenbasis, define the diagonal matrix $\mathsf N$ by
\[
 \mathsf N_i:=c_3\rho(\widehat\lambda_i^2-1/6).
\]
Then $\operatorname{tr}\mathsf N=0$.

\begin{proof}[Proof of Proposition~\ref{prop:nonlinear-all-directions}]
Collecting the terms in $\nabla|A|$, $\nabla A$, and $Z_\sigma$, and using
$\operatorname{tr}Z_\sigma=0$, gives \eqref{eq:nonlinear-expansion}.
The terms independent of $\nabla A$ and $Z_\sigma$ form $B_0$.
The coefficient collection is shown in Appendix~\ref{sec:colored-coefficients}.
\begin{equation}\label{eq:nonlinear-expansion}
\begin{aligned}
 \mathscr P_\sigma
 ={}&B_0(|\widehat A\nu|^2)+|Z_\sigma|^2
 +c_3\rho\langle\widehat A^2-g/6,Z_\sigma\rangle
 +(m_{22}-c_6)\rho\frac{G^2|\nabla A|^2}{|A|^2|\nabla G|^2}\\
 &{}-\left(m_{22}-\beta(\rho)\right)
 \rho\frac{G^2|\nabla|A||^2}{|A|^2|\nabla G|^2}
 -\widetilde m_{11}(\rho)|Z_\sigma\nu|^2\\
 &{}-2\widetilde m_{12}(\rho)\sqrt\rho\frac{G\langle\nabla|A|,Z_\sigma\nu\rangle}{|A||\nabla G|}
 +B_1\sqrt\rho\frac{G\langle\nabla|A|,\nu\rangle}{|A||\nabla G|}
 +B_2Z_\sigma(\nu,\nu).
\end{aligned}
\end{equation}
In the chosen eigenbasis, we have the following pointwise estimate.

\smallskip
\noindent\textit{Claim.} We have
\begin{equation}\label{eq:directional-kato}
 |\nabla A|^2\geq\sum_{i=1}^6\frac{(\nabla_i|A|)^2}{1/3+\widehat\lambda_i^2/2}.
\end{equation}
We retain the eigenvalue dependence in the classical derivation of
the refined Kato inequality~\cite[pp.~279--280]{SchoenSimonYau1975}
from the Codazzi equation.
For $z\in\mathbb R^6$ with $\sum_{j=1}^6z_j=0$, the weighted Cauchy--Schwarz inequality gives
\begin{equation}\label{eq:weighted-trace-free-cauchy}
\begin{aligned}
 &\bigl(\sum_{j=1}^6\widehat\lambda_jz_j\bigr)^2
 =\Bigl[\frac34\widehat\lambda_i z_i
 +\sum_{j\ne i}\bigl(\widehat\lambda_j-\frac14\widehat\lambda_i\bigr)z_j\Bigr]^2\\
 \leq{}&\Bigl[\frac9{16}\widehat\lambda_i^2
 +\frac13\sum_{j\ne i}\bigl(\widehat\lambda_j-\frac14\widehat\lambda_i\bigr)^2\Bigr]
 \bigl(z_i^2+3\sum_{j\ne i}z_j^2\bigr)
 =\bigl(\frac13+\frac{\widehat\lambda_i^2}{2}\bigr)
 \bigl(z_i^2+3\sum_{j\ne i}z_j^2\bigr).
\end{aligned}
\end{equation}
Here we used $\sum_j\widehat\lambda_j=0$ and $\sum_j\widehat\lambda_j^2=1$.
Take $z_j=\nabla_iA_{jj}$. The Codazzi equation and minimality give
$\sum_jz_j=0$ and $\nabla_i|A|=\sum_j\widehat\lambda_jz_j$, while full symmetry gives
\[
 \sum_i\left[(\nabla_iA_{ii})^2+3\sum_{j\ne i}(\nabla_iA_{jj})^2\right]
 \leq|\nabla A|^2.
\]
Summing \eqref{eq:weighted-trace-free-cauchy} after division by
$1/3+\widehat\lambda_i^2/2$ proves \eqref{eq:directional-kato}.
This completes the proof of the claim.

For each $i$, define
\[
 \mathbf x_i:=\begin{pmatrix}
 \frac{\sqrt\rho\,G\nabla_i|A|}{|A||\nabla G|}\\(Z_\sigma\nu)_i
 \end{pmatrix},\quad
 \mathbf y_i:=\begin{pmatrix}
 B_1\\
 \frac65\mathsf N_i+B_2
 \end{pmatrix},\quad
 \mathsf S_i:=\begin{pmatrix}
 \alpha(|\widehat\lambda_i|,\rho)&-\widetilde m_{12}(\rho)\\
 -\widetilde m_{12}(\rho)&6/5-\widetilde m_{11}(\rho)
 \end{pmatrix}.
\]
Since $Z_\sigma+\mathsf N/2$ is symmetric and trace-free,
$|Z_\sigma+\mathsf N/2|^2\geq\frac65|(Z_\sigma+\mathsf N/2)\nu|^2$.
Hence
\[
\begin{aligned}
 &|Z_\sigma|^2+\langle\mathsf N,Z_\sigma\rangle
 =|Z_\sigma+\mathsf N/2|^2-\frac14|\mathsf N|^2\\
 \geq{}&-\frac14|\mathsf N|^2+\frac65|Z_\sigma\nu|^2
 +\frac65\langle\mathsf N\nu,Z_\sigma\nu\rangle
 +\frac3{10}\sum_i\mathsf N_i^2\nu_i^2.
\end{aligned}
\]
Substituting this bound into \eqref{eq:nonlinear-expansion} and using
\eqref{eq:directional-kato}, we obtain
\begin{equation}\label{eq:joint-vector-bound}
 \mathscr P_\sigma\geq B_0(|\widehat A\nu|^2)-\frac14|\mathsf N|^2
 +\sum_i\left[
 \mathbf x_i^{\mathsf T}\mathsf S_i\mathbf x_i
 +\nu_i\mathbf x_i^{\mathsf T}\mathbf y_i
 +\frac3{10}\mathsf N_i^2\nu_i^2\right].
\end{equation}
Lemma~\ref{prop:nonlinear-denominators} gives $\mathsf S_i>0$.
The Cauchy--Schwarz inequalities give
\[
\begin{aligned}
 |\nu_i\mathbf x_i^{\mathsf T}\mathbf y_i|
 &=|\nu_i|\left|(\mathsf S_i^{1/2}\mathbf x_i)^{\mathsf T}
 (\mathsf S_i^{-1/2}\mathbf y_i)\right|\\
 &\leq|\nu_i|\sqrt{\mathbf x_i^{\mathsf T}\mathsf S_i\mathbf x_i}
 \sqrt{\mathbf y_i^{\mathsf T}\mathsf S_i^{-1}\mathbf y_i}
 \leq\mathbf x_i^{\mathsf T}\mathsf S_i\mathbf x_i
 +\frac{\nu_i^2}{4}\mathbf y_i^{\mathsf T}\mathsf S_i^{-1}\mathbf y_i.
\end{aligned}
\]
Hence
\begin{equation}\label{eq:joint-matrix-cauchy}
 \mathbf x_i^{\mathsf T}\mathsf S_i\mathbf x_i
 +\nu_i\mathbf x_i^{\mathsf T}\mathbf y_i
 \geq-\frac{\nu_i^2}{4}\mathbf y_i^{\mathsf T}\mathsf S_i^{-1}\mathbf y_i.
\end{equation}
By the definition of $\eta$ and the $2\times2$ inverse formula,
\[
\begin{aligned}
 \det\mathsf S_i
 &=\alpha(|\widehat\lambda_i|,\rho)\eta(|\widehat\lambda_i|,\rho),\\
 \mathsf S_i^{-1}
 &=\frac{1}{\alpha(|\widehat\lambda_i|,\rho)\eta(|\widehat\lambda_i|,\rho)}
 \begin{pmatrix}
 6/5-\widetilde m_{11}(\rho)&\widetilde m_{12}(\rho)\\
 \widetilde m_{12}(\rho)&\alpha(|\widehat\lambda_i|,\rho)
 \end{pmatrix}.
\end{aligned}
\]
Using
$6/5-\widetilde m_{11}(\rho)=\eta(|\widehat\lambda_i|,\rho)
+\widetilde m_{12}(\rho)^2/\alpha(|\widehat\lambda_i|,\rho)$,
we obtain
\begin{equation}\label{eq:joint-inverse-quadratic}
\begin{aligned}
 &\mathbf y_i^{\mathsf T}\mathsf S_i^{-1}\mathbf y_i\\
 ={}&\frac{
 (6/5-\widetilde m_{11}(\rho))B_1^2
 +2\widetilde m_{12}(\rho)B_1(\frac65\mathsf N_i+B_2)
 }{\alpha(|\widehat\lambda_i|,\rho)\eta(|\widehat\lambda_i|,\rho)}
 +\frac{(\frac65\mathsf N_i+B_2)^2}{\eta(|\widehat\lambda_i|,\rho)}\\
 ={}&\frac{B_1^2}{\alpha(|\widehat\lambda_i|,\rho)}
 +\frac{\left[\frac65\mathsf N_i+B_2
 +\frac{\widetilde m_{12}(\rho)B_1}{\alpha(|\widehat\lambda_i|,\rho)}\right]^2}
 {\eta(|\widehat\lambda_i|,\rho)}.
\end{aligned}
\end{equation}
Substituting \eqref{eq:joint-inverse-quadratic} into
\eqref{eq:joint-matrix-cauchy} and applying these bounds term by term in
\eqref{eq:joint-vector-bound}, we obtain
\[
\begin{aligned}
 \mathscr P_\sigma\geq{}&B_0(|\widehat A\nu|^2)-\frac14|\mathsf N|^2\\
 &+\sum_i\nu_i^2\Biggl[
 \frac3{10}\mathsf N_i^2
 -\frac{B_1^2}{4\alpha(|\widehat\lambda_i|,\rho)}
 -\frac{\left[\frac65\mathsf N_i+B_2
       +\frac{\widetilde m_{12}(\rho)B_1}{\alpha(|\widehat\lambda_i|,\rho)}\right]^2}
       {4\eta(|\widehat\lambda_i|,\rho)}\Biggr].
\end{aligned}
\]
For fixed $\sigma$, $\rho$, and coefficients, $B_0(u)$ is affine in $u$.
Since $\sum_i\nu_i^2=1$, we have
\[
 B_0(|\widehat A\nu|^2)=\sum_i\nu_i^2B_0(\widehat\lambda_i^2),
 \qquad |\mathsf N|^2=c_3^2\rho^2(\operatorname{tr}\widehat A^4-1/6).
\]
Using \eqref{eq:envelope} with $q=|\widehat\lambda_i|$ and the definition
\eqref{eq:scalarF}, we obtain
\[
 \mathscr P_\sigma\geq\sum_i\nu_i^2\,
 \mathscr F_\sigma(|\widehat\lambda_i|,\rho).
\]
This proves \eqref{eq:all-directions}.
\end{proof}

In the next subsection, we use this reduction to prove a uniform positive lower bound for the chosen coefficient rows.

\subsection{Coercivity and moment extension}\label{sec:nonlinear-coercivity}
\begin{proposition}\label{prop:bridge}
For each of the eleven rational rows in Appendix~\ref{sec:rows}, and every $\sigma $ in
that row's interval,
\begin{equation}\label{eq:bridge-bound}
 \mathscr P_\sigma\geq10^{-7}(1+\rho+\rho^2).
\end{equation}
The union of the intervals is $[2.46,2.4962]$.
\end{proposition}
\begin{proof}
Finite decimals denote exact rational numbers throughout the paper.
Fix a row of Table~\ref{tab:nonlinear-rows-first},
and keep its coefficients fixed on $[\sigma_l,\sigma_r]$.
Direct substitution shows that every row satisfies \eqref{eq:matrix-signs}.

\textbf{Step 1. Reduction to the endpoints.}
Fix $0\leq q\leq\sqrt{5/6}$ and $\rho\geq0$.
We prove that $\mathscr F_\sigma(q,\rho)$ in \eqref{eq:scalarF} is concave in $\sigma$.
By \eqref{eq:nonlinear-coefficients}, the coefficient of $\sigma^2$
in its first term $B_0(q^2)$ is
\[
 -\frac15+\rho\left(\frac{c_{6}}2-\frac{m_{22}}{4}\right)+\frac{\rho\beta(\rho)}4.
\]
Since $m_{12}\geq0$ and $\det\mathsf M=m_{22}m_{11}-m_{12}^2\leq m_{22}m_{11}$,
\[
 0\leq \rho\beta(\rho)\leq m_{11}.
\]
Together with $m_{22}>4c_{6}>0$ and $m_{11}<4/5$, this gives
\[
 -\frac15+\rho\left(\frac{c_{6}}2-\frac{m_{22}}{4}\right)+\frac{\rho\beta(\rho)}4
 \leq-\frac15+\frac{m_{11}}{4}<0.
\]
Hence the first term in \eqref{eq:scalarF} is concave in $\sigma$.
In the second and fourth terms, the expressions being squared,
\[
 B_1,\qquad
 \frac65c_3\left(q^2-\frac16\right)\rho+B_2
 +\frac{\widetilde m_{12}(\rho)B_1}{\alpha(q,\rho)},
\]
are affine in $\sigma$ by \eqref{eq:nonlinear-coefficients}.
The denominators $4\alpha(q,\rho)$ and $4\eta(q,\rho)$ are positive
by Lemma~\ref{prop:nonlinear-denominators} and independent of $\sigma$.
These two terms are therefore concave in $\sigma$, while the third term
is independent of $\sigma$. Thus $\mathscr F_\sigma(q,\rho)$ is concave in $\sigma$.
If $\sigma=(1-t)\sigma_l+t\sigma_r$, $0\leq t\leq1$, then
\[
 \mathscr F_\sigma(q,\rho)
 \geq(1-t)\mathscr F_{\sigma_l}(q,\rho)+t\mathscr F_{\sigma_r}(q,\rho).
\]

\medskip
\noindent\textbf{Step 2. The endpoint bounds.}
Lemma~\ref{lem:nonlinear-endpoints} below gives the scalar endpoint
estimate verified in Lean. The interpolation inequality in Step 1 therefore yields
\[
 \mathscr F_\sigma(q,\rho)\geq10^{-7}(1+\rho+\rho^2)
 \qquad\text{for all }\sigma\in[\sigma_l,\sigma_r].
\]
For any unit direction $\nu$, Proposition~\ref{prop:nonlinear-all-directions}
and $\sum_i\nu_i^2=1$ now give
\[
 \mathscr P_\sigma
 \geq\sum_i\nu_i^2\mathscr F_\sigma\bigl(|\widehat\lambda_i|,\rho\bigr)
 \geq10^{-7}(1+\rho+\rho^2)\sum_i\nu_i^2
 =10^{-7}(1+\rho+\rho^2).
\]
Finally, the consecutive intervals in
Table~\ref{tab:nonlinear-rows-first} have union $[2.46,2.4962]$.
\end{proof}

The endpoint estimate is as follows.
\begin{lemma}\label{lem:nonlinear-endpoints}
For each row of Table~\ref{tab:nonlinear-rows-first}, use its coefficients
in \eqref{eq:scalarF} and let $[\sigma_l,\sigma_r]$ be its interval.
Then, for every $\sigma\in\{\sigma_l,\sigma_r\}$,
$0\leq q\leq\sqrt{5/6}$, and $\rho\geq0$,
\begin{equation}\label{eq:formal-bridge-endpoints}
 \mathscr F_\sigma(q,\rho)\geq10^{-7}(1+\rho+\rho^2).
\end{equation}
\end{lemma}
\begin{proof}
This result is verified in Lean. See Appendix~\ref{sec:coefficients-verification}.
\end{proof}

\begin{corollary}\label{cor:seed}
Under stability, $\mathcal M_{2.4962}<\infty$.
\end{corollary}
\begin{proof}
Combining \eqref{eq:sourcePnl} and \eqref{eq:bridge-bound} and choosing
$\varepsilon$ sufficiently small gives
\[
 \int\phi^4\,d\mu^M_\sigma\leq C\int(|D\phi|^4+|D^2\phi|^4)\,d\mu^M_\sigma
\]
uniformly for $\sigma\in[2.46,2.4962]$.
Corollary~\ref{cor:continuation} applies starting from \eqref{eq:ordinary-seed}.
This estimate was derived with compactly supported cutoff functions,
without assuming that the target integral is finite.
\end{proof}

\section{Finiteness of all subcritical moments}\label{sec:all-subcritical-moments}
\subsection{Homogeneous curvature-tensor identities}\label{sec:curvature-tensors}
We derive three additional integral identities with compactly supported
cutoff functions and estimate their cutoff terms. Throughout this subsection, we use the geometric quantities
$\nu=\nabla G/|\nabla G|$ and $Z_\sigma$ defined in Section~\ref{sec:first-estimates}.
For a symmetric $C^1$ tensor $T$ on $M\setminus\{o\}$, define
\begin{equation}\label{eq:weighted-divergence-integrand}
 \mathscr B_\sigma[T]:=\frac{G^2}{|\nabla G|^2}
 \left[\langle T,Z_\sigma\rangle+2\left(1-\frac\sigma5\right)T(\nu,\nu)
 -\frac{\sigma}{10}\operatorname{tr}T\right]
 +\frac{G^3}{|\nabla G|^3}\langle\operatorname{div}T,\nu\rangle.
\end{equation}
\begin{proposition}\label{prop:curvature-divergence}
Let $T$ be a symmetric $C^1$ tensor on $M\setminus\{o\}$.
For every smooth compactly supported function $w=w(G)$ away from the pole,
\begin{equation}\label{eq:Newton-general}
 \int_M\mathscr B_\sigma[T]w\,d\mu^M_\sigma
 =-\int_M\frac{G^2}{|\nabla G|^2}T(\nu,\nu)Dw\,d\mu^M_\sigma.
\end{equation}
\end{proposition}
\begin{proof}
By direct computation, we have
\begin{equation}\label{eq:weighted-divergence-expansion}
\begin{aligned}
 \operatorname{div}(G^{2-\sigma}wT\nabla G)
 ={}&G^{2-\sigma}w\langle\operatorname{div}T,\nabla G\rangle
 +G^{2-\sigma}w\langle T,\nabla^2G\rangle\\
 &+G^{1-\sigma}\bigl((2-\sigma)w+Dw\bigr)T(\nabla G,\nabla G).
\end{aligned}
\end{equation}
Recalling the definition \eqref{eq:shifted-hessian} of $Z_\sigma$ and that of
$d\mu^M_\sigma$, we have
\[
 \nabla^2G=\frac{|\nabla G|^2}{G}
 \left[Z_\sigma+\frac{\sigma}{10}(6\nu\otimes\nu-g)\right],\qquad
 d\mu^M_\sigma=|\nabla G|^4G^{-\sigma-1}\,dV_g.
\]
After separating the factors involving $w$ and $Dw$, the three terms
in the divergence expansion \eqref{eq:weighted-divergence-expansion} take the following forms
with respect to $dV_g$ and $d\mu^M_\sigma$.
\[
\begin{aligned}
 G^{2-\sigma}\langle\operatorname{div}T,\nabla G\rangle\,dV_g
 &=\frac{G^3}{|\nabla G|^3}\langle\operatorname{div}T,\nu\rangle\,d\mu^M_\sigma,\\
 G^{2-\sigma}\langle T,\nabla^2G\rangle\,dV_g
 &=\frac{G^2}{|\nabla G|^2}\left[\langle T,Z_\sigma\rangle
 +\frac{\sigma}{10}\bigl(6T(\nu,\nu)-\operatorname{tr}T\bigr)\right]d\mu^M_\sigma,\\
 G^{1-\sigma}T(\nabla G,\nabla G)\,dV_g
 &=\frac{G^2}{|\nabla G|^2}T(\nu,\nu)\,d\mu^M_\sigma.
\end{aligned}
\]
Integrating \eqref{eq:weighted-divergence-expansion} and using
\eqref{eq:weighted-divergence-integrand}, we obtain
\[
 0=\int_M\mathscr B_\sigma[T]w\,d\mu^M_\sigma
 +\int_M\frac{G^2}{|\nabla G|^2}T(\nu,\nu)Dw\,d\mu^M_\sigma.
\]
This is \eqref{eq:Newton-general}. At $|\nabla G|=0$, the normalized terms
are interpreted through the expressions multiplying $dV_g$ above,
as in Remark~\ref{rem:zeros}.
\end{proof}

To obtain the three additional integral identities, we use the curvature tensors
\begin{equation}\label{eq:curvature-tensors}
 T_1:=\frac{A^4}{|A|^2},\qquad
 T_2:=\frac{\operatorname{tr}(A^3)}{|A|^2}A,\qquad
 T_3:=\frac{\operatorname{tr}(A^4)}{|A|^2}g.
\end{equation}
At points where $A=0$, set $T_1=T_2=T_3=0$.
\begin{lemma}\label{lem:curvature-divergences}
The tensors $T_j$ are $C^1$ on $M$ and satisfy the following divergence
identities, with each term on the right extended continuously by zero at $A=0$:

\begin{align}
 \operatorname{div}T_1
 &=\frac{\nabla\operatorname{tr}(A^4)}{4|A|^2}
 +\frac{A\nabla\operatorname{tr}(A^3)}{3|A|^2}
 +\frac{A^2\nabla|A|}{|A|}
 -\frac{2A^4\nabla|A|}{|A|^3},\label{eq:div-quartic-tensor}\\
 \operatorname{div}T_2
 &=\frac{A\nabla\operatorname{tr}(A^3)}{|A|^2}
 -\frac{2\operatorname{tr}(A^3)}{|A|^3}A\nabla|A|,\label{eq:div-cubic-tensor}\\
 \operatorname{div}T_3
 &=\nabla\left(\frac{\operatorname{tr}(A^4)}{|A|^2}\right).\label{eq:div-scalar-tensor}
\end{align}
\end{lemma}
\begin{proof}
At a point where $|A|>0$, choose a normal orthonormal frame diagonalizing
$A$, with eigenvalues $\lambda_i$. The Codazzi equation and minimality give
$\nabla_aA_{ia}=\nabla_iA_{aa}$ and
$\operatorname{div}A=0$.
Differentiating the four factors in $A^4$ yields
\[
\begin{aligned}
 (\operatorname{div}A^4)_i
 &=\sum_a\sum_{h=0}^3\lambda_i^h\lambda_a^{3-h}\nabla_aA_{ia}\\
 &=\sum_a(\lambda_a^3+\lambda_i\lambda_a^2+\lambda_i^2\lambda_a)
       \nabla_iA_{aa}
   +\lambda_i^3\sum_a\nabla_iA_{aa}\\
 &=\sum_a(\lambda_a^3+\lambda_i\lambda_a^2+\lambda_i^2\lambda_a)
       \nabla_iA_{aa}.
\end{aligned}
\]
The product rule and cyclicity of the trace give, for $m=3,4$,
\[
 \nabla_i\operatorname{tr}(A^m)
 =m\operatorname{tr}(A^{m-1}\nabla_iA)
 =m\sum_a\lambda_a^{m-1}\nabla_iA_{aa}.
\]
Since $\sum_a\lambda_a\nabla_iA_{aa}=|A|\nabla_i|A|$, it follows that
\[
 \operatorname{div}A^4
 =\frac14\nabla\operatorname{tr}(A^4)
 +\frac13A\nabla\operatorname{tr}(A^3)
 +|A|A^2\nabla|A|.
\]
Applying the product rule to $|A|^{-2}A^4$ proves
\eqref{eq:div-quartic-tensor}.
The same rule, together with $\operatorname{div}A=0$, gives
\eqref{eq:div-cubic-tensor}, and $\nabla g=0$ gives
\eqref{eq:div-scalar-tensor}.

In \eqref{eq:curvature-tensors}, each numerator has degree four in $A$
and the denominator has degree two. Thus $T$ is smooth away from zero and
homogeneous of degree two, giving
\[
 |T(A)|\leq C|A|^2,\qquad |\mathrm dT(A)|\leq C|A|\qquad\text{for }A\ne0.
\]
The first bound gives $\mathrm dT(0)=0$, and the second shows continuity
of the derivative at zero. Hence the zero extension is $C^1$.
Each term in \eqref{eq:div-quartic-tensor}--\eqref{eq:div-scalar-tensor}
is bounded by $C|A||\nabla A|$ and extends continuously by zero at $A=0$,
where $\nabla T_j=0$. Thus the identities hold on $M$.
\end{proof}

For the normalized integrands below, let $\widehat{\boldsymbol\lambda}$ denote
the spectrum of $\widehat A$ and write
$s_j(\widehat{\boldsymbol\lambda}):=\operatorname{tr}\widehat A^j$.
Applying the weighted divergence identity to $T_1,T_2,T_3$ gives the
following integral relations.
\begin{proposition}\label{prop:additional-curvature-identities}
Define
\begin{equation}\label{eq:additional-curvature-integrands}
 F_{9+j}:=\mathscr B_\sigma[T_j]\qquad\text{for }j=1,2,3.
\end{equation}
Given the cutoff function $\phi$, define
\[
 \mathcal R_{9+j}:=\int_M F_{9+j}\phi^4\,d\mu^M_\sigma\qquad\text{for }j=1,2,3.
\]
These integrals are given by
\begin{equation}\label{eq:additional-curvature-errors}
 \mathcal R_{9+j}=-\int_M\frac{G^2}{|\nabla G|^2}
 T_j(\nu,\nu)D(\phi^4)\,d\mu^M_\sigma\qquad\text{for }j=1,2,3.
\end{equation}
These integral identities also hold with an arbitrary smooth compactly
supported weight in place of $\phi^4$.
For $0\leq\phi\in C_c^\infty((0,t_0))$ and $\sigma$ in a fixed compact
interval, every $\varepsilon>0$ admits a uniform $C_\varepsilon$ such that
\begin{equation}\label{eq:additional-curvature-source-bound}
 |\mathcal R_i|\leq\varepsilon\mathcal M_\sigma[\phi]
 +C_\varepsilon\bigl(\mathcal M_\sigma[|D\phi|]
                       +\mathcal M_\sigma[|D^2\phi|]\bigr),
 \qquad i=10,11,12.
\end{equation}
\end{proposition}
\begin{proof}
Apply \eqref{eq:Newton-general} to $T_j$, whose $C^1$ extension is given by
Lemma~\ref{lem:curvature-divergences}, to obtain
\eqref{eq:additional-curvature-errors}.
Since $|\widehat A|=1$, we have $|T_j(\nu,\nu)|\leq C|A|^2$.
Consequently $|\mathcal R_{9+j}|\leq C\int_M\rho\phi^3|D\phi|d\mu^M_\sigma$.
Corollary~\ref{cor:cutoff-remainders} therefore gives
\eqref{eq:additional-curvature-source-bound}.
\end{proof}

\begin{remark}\label{rem:explicit-curvature-integrands}
Substituting
\eqref{eq:div-quartic-tensor}--\eqref{eq:div-scalar-tensor} into
\eqref{eq:weighted-divergence-integrand} gives the explicit formulas
on $\{|A|>0,\ |\nabla G|>0\}$
\begin{align*}
 F_{10}={}&\rho\left[\langle\widehat A^4,Z_\sigma\rangle
 +2\left(1-\frac\sigma5\right)|\widehat A^2\nu|^2
 -\frac\sigma{10}s_4(\widehat{\boldsymbol\lambda})\right]\\
 &{}+\frac{G^3}{|A|^2|\nabla G|^3}
 \left\langle\frac14\nabla\operatorname{tr}(A^4)
 +\frac13 A\nabla\operatorname{tr}(A^3),\nu\right\rangle\\
 &{}+\frac{G^3}{|A||\nabla G|^3}
 \left\langle\nabla|A|,\left(A^2-\frac{2A^4}{|A|^2}\right)\nu\right\rangle,\\
 F_{11}={}&\rho s_3(\widehat{\boldsymbol\lambda})
 \left[\langle\widehat A,Z_\sigma\rangle
 +2\left(1-\frac\sigma5\right)\langle\widehat A\nu,\nu\rangle\right]\\
 &{}+\frac{G^3}{|A|^2|\nabla G|^3}
 \left\langle A\nabla\operatorname{tr}(A^3)
 -\frac{2\operatorname{tr}(A^3)}{|A|}A\nabla|A|,\nu\right\rangle,\\
 F_{12}={}&(2-\sigma)\rho s_4(\widehat{\boldsymbol\lambda})
 +\frac{G^3}{|A|^2|\nabla G|^3}
 \left\langle\nabla\operatorname{tr}(A^4)
 -\frac{2\operatorname{tr}(A^4)}{|A|}\nabla|A|,\nu\right\rangle.
\end{align*}
\end{remark}

\subsection{Spectral coercivity and moment extension}\label{sec:spectral-coercivity}
\begin{proposition}\label{prop:endpoint}
Define
\begin{equation}\label{eq:Psp}
 \widetilde{\mathscr P}_\sigma:=F_2+c_3F_3+c_5F_5+c_6F_6+c_7F_7+\sum_{i=10}^{12}c_iF_i,
\end{equation}
where we choose the constants as
\begin{equation}\label{eq:spectral-coefficients}
 (c_3,c_5,c_6,c_7,c_{10},c_{11},c_{12})
 =(2.67,6.04,0.52,-4.18,-1.25,0.37,1.98).
\end{equation}
For $2.4962\leq\sigma\leq5/2$ and $\rho\geq0$, we have
\begin{equation}\label{eq:endpoint-bound}
 \widetilde{\mathscr P}_\sigma
 \geq\frac{5/2-\sigma}{3800000}+10^{-8}(\rho+\rho^2).
\end{equation}
\end{proposition}

\begin{corollary}\label{cor:all-moments}
Under stability,
\[
 \mathcal M_\sigma<\infty\qquad\text{for every }\sigma<5/2.
\]
\end{corollary}
\begin{proof}
Fix $\sigma_1\in(2.4962,5/2)$ and take
$0\leq\phi\in C_c^\infty((0,t_0))$.
For $\sigma\in[2.4962,\sigma_1]$, multiply \eqref{eq:endpoint-bound}
by $\phi^4$ and integrate with respect to $d\mu^M_\sigma$.
Using \eqref{eq:F23-integrals}, \eqref{eq:F5-stability},
\eqref{eq:F67-integrals}, and
Proposition~\ref{prop:additional-curvature-identities}, we obtain
\begin{equation}\label{eq:subcritical-integrated-bound}
\begin{aligned}
 &\frac{5/2-\sigma_1}{3800000}\mathcal M_\sigma[\phi]
 \leq\int_M\widetilde{\mathscr P}_\sigma\phi^4\,d\mu^M_\sigma\\
 \leq{}&\mathcal R_2+c_3\mathcal R_3+c_5\mathcal R_5+c_6\mathcal R_6
 +c_7\mathcal R_7+\sum_{i=10}^{12}c_i\mathcal R_i\\
 \leq{}&\varepsilon\mathcal M_\sigma[\phi]
 +C_\varepsilon\bigl(\mathcal M_\sigma[|D\phi|]
                       +\mathcal M_\sigma[|D^2\phi|]\bigr).
\end{aligned}
\end{equation}
The second inequality uses stability and $c_5>0$.
For the last inequality, \eqref{eq:first-four-source-bound} controls
$\mathcal R_2,\mathcal R_3$, \eqref{eq:additional-source-bound} controls
$\mathcal R_5,\mathcal R_6,\mathcal R_7$, and
\eqref{eq:additional-curvature-source-bound} controls
$\mathcal R_{10},\mathcal R_{11},\mathcal R_{12}$.
Choosing $\varepsilon<(5/2-\sigma_1)/3800000$ and absorbing the moment
term gives
\[
 \mathcal M_\sigma[\phi]\leq C_{\sigma_1}
 \bigl(\mathcal M_\sigma[|D\phi|]+\mathcal M_\sigma[|D^2\phi|]\bigr),
\]
which is \eqref{eq:abstract-cutoff-ineq} uniformly on this interval.
Corollary~\ref{cor:continuation} and Corollary~\ref{cor:seed} then give
$\mathcal M_\sigma<\infty$ for every $\sigma\in[2.4962,\sigma_1]$.
Since $\sigma_1<5/2$ is arbitrary, the proof is complete.
\end{proof}

We prove Proposition~\ref{prop:endpoint} by a tensor reduction and a scalar lower bound.
For the reduction, we introduce the following notation.
For a normalized trace-free spectrum, write
\[
 \widehat{\boldsymbol\lambda}:=(\widehat\lambda_1,\ldots,\widehat\lambda_6),
 \qquad \sum_j\widehat\lambda_j=0,\qquad \sum_j\widehat\lambda_j^2=1,
\]
and set $s_k(\widehat{\boldsymbol\lambda}):=\sum_j\widehat\lambda_j^k$.
Define
\begin{equation}\label{eq:spectral-N}
 \mathcal N_j:=c_3\widehat\lambda_j^2+c_{10}\widehat\lambda_j^4
 +c_{11}s_3(\widehat{\boldsymbol\lambda})\widehat\lambda_j-\frac{c_3+c_{10}s_4(\widehat{\boldsymbol\lambda})}{6}.
\end{equation}
For each index $i$, define
\begin{equation}\label{eq:spectral-V}
\begin{aligned}
 V_j^{(i)}:={}&(c_{10}+4c_{12})\widehat\lambda_j^3
 +(c_{10}+3c_{11})\widehat\lambda_i\widehat\lambda_j^2\\
 &+\bigl[c_7+c_{10}(\widehat\lambda_i^2-2\widehat\lambda_i^4)
 -2c_{11}s_3(\widehat{\boldsymbol\lambda})\widehat\lambda_i-2c_{12}s_4(\widehat{\boldsymbol\lambda})\bigr]\widehat\lambda_j.
\end{aligned}
\end{equation}
Set
\begin{align}
 L_{\sigma,i}:={}&2\left(1-\frac\sigma5\right)\mathcal N_i
 -\widehat\lambda_i^2-\frac{c_5}{4}(3-\sigma)^2\notag\\*
 \MoveEqLeft[3]{}+(\sigma-2)\left[\frac{c_3}{3}-\frac{c_7}{2}
 -\frac{c_6}{2}(3-\sigma)
 -\left(\frac{c_{10}}6+c_{12}\right)s_4(\widehat{\boldsymbol\lambda})\right]\label{eq:spectral-LQ}\\*
 \MoveEqLeft[3]{}+\min_{\substack{z\in\mathbb R^6\\\sum_{j=1}^6z_j=0}}
 \left\{(c_5-c_6)\left(z_i^2+3\sum_{j\ne i}z_j^2\right)
 -c_5\left(\sum_{j=1}^6\widehat\lambda_jz_j\right)^2
 +\sum_{j=1}^6V_j^{(i)}z_j\right\},
\notag\displaybreak[0]\\
 Q:={}&c_6-\frac14\sum_{j=1}^6\mathcal N_j^2.\label{eq:spectral-Q}
\end{align}
By \eqref{eq:weighted-trace-free-cauchy} and $\widehat\lambda_i^2\leq5/6$,
the quadratic part on $\sum_jz_j=0$ is bounded below
by $(c_5/4-c_6)(z_i^2+3\sum_{j\ne i}z_j^2)$.
Since $c_5/4-c_6=0.99>0$, the minimum exists and is unique.
The scalar function used in this section is
\begin{equation}\label{eq:spectral-scalar-function}
 \widetilde{\mathscr F}_\sigma(\widehat{\boldsymbol\lambda},i,\rho)
 :=\frac{\sigma(5-2\sigma)}{10}+L_{\sigma,i}\rho+Q\rho^2.
\end{equation}

\begin{lemma}\label{lem:spectral-tensor-reduction}
Let $(E_i)$ be an orthonormal eigenbasis of $\widehat A$, with normalized
spectrum $\widehat{\boldsymbol\lambda}$, and write
$\nu_i:=\langle\nu,E_i\rangle$.
For $2.4962\leq\sigma\leq5/2$ and $\rho\geq0$, the normalized tensors defined above satisfy
\begin{equation}\label{eq:spectral-tensor-reduction}
 \widetilde{\mathscr P}_\sigma
 \geq\sum_{i=1}^6\nu_i^2\,
 \widetilde{\mathscr F}_\sigma(\widehat{\boldsymbol\lambda},i,\rho).
\end{equation}
\end{lemma}
\begin{proof}
Set $C_{ijk}:=G\nabla_iA_{jk}/(|A||\nabla G|)$ and
$\mathcal N:=\operatorname{diag}(\mathcal N_1,\ldots,\mathcal N_6)$.
We expand \eqref{eq:Psp} using \eqref{eq:weighted-divergence-integrand}
and \eqref{eq:div-quartic-tensor}--\eqref{eq:div-scalar-tensor}, collecting
terms quadratic and linear in $Z_\sigma$, terms quadratic and linear in
$C$, and terms independent of both tensors. This gives
\begingroup
\allowdisplaybreaks[2]
\begin{align}
 \widetilde{\mathscr P}_\sigma={}&\frac{\sigma(5-2\sigma)}{10}+c_6\rho^2
 +|Z_\sigma|^2+\rho\langle\mathcal N,Z_\sigma\rangle
 +(c_5-c_6)\rho\sum_{a,b,c=1}^6C_{abc}^2\notag\\
 &{}+\rho\sum_{i=1}^6\nu_i^2
 \left[2\left(1-\frac\sigma5\right)\mathcal N_i
 -\widehat\lambda_i^2-\frac{c_5}{4}(3-\sigma)^2\right]\notag\\
 &{}+(\sigma-2)\rho
 \left[\frac{c_3}{3}-\frac{c_7}{2}-\frac{c_6}{2}(3-\sigma)\right]
 -(\sigma-2)\rho\left(\frac{c_{10}}6+c_{12}\right)
 s_4(\widehat{\boldsymbol\lambda})
\notag\\
 &{}-c_5\rho\sum_{i=1}^6
 \left(\sum_{j=1}^6\widehat\lambda_jC_{ijj}\right)^2
 +\rho\sum_{i,j=1}^6\nu_iV_j^{(i)}C_{ijj}.\label{eq:spectral-collected-expansion}
\end{align}
\endgroup
The individual expansions and coefficient collection are shown
using colors and frames in Appendix~\ref{sec:spectral-colored-coefficients}.

The Codazzi equation and minimality imply that $C$ is fully symmetric and
$\sum_jC_{ijj}=0$ for every $i$, while
$G\nabla_i|A|/(|A||\nabla G|)=\sum_j\widehat\lambda_jC_{ijj}$.
Full~symmetry~gives
\begin{equation}\label{eq:spectral-codazzi-decomposition}
 |C|^2=\sum_{i=1}^6\left(C_{iii}^2+3\sum_{j\ne i}C_{ijj}^2\right)
       +6\sum_{a<b<c}C_{abc}^2.
\end{equation}
Using \eqref{eq:spectral-codazzi-decomposition}, the terms involving
$C$ in \eqref{eq:spectral-collected-expansion} are
\[
\begin{aligned}
 &\rho\sum_{i=1}^6\biggl[
 (c_5-c_6)\left(C_{iii}^2+3\sum_{j\ne i}C_{ijj}^2\right)
 -c_5\left(\sum_j\widehat\lambda_jC_{ijj}\right)^2
 +\nu_i\sum_jV_j^{(i)}C_{ijj}\biggr]\\
 &\qquad{}+6\rho(c_5-c_6)\sum_{a<b<c}C_{abc}^2.
\end{aligned}
\]
The groups $(C_{ijj})_{j=1}^6$ are independent, each with its own trace
constraint. For $\nu_i\ne0$, substituting $C_{ijj}=\nu_i z_j$ shows that
the $i$th bracket is at least $\nu_i^2$ times the minimum in
\eqref{eq:spectral-LQ}. For $\nu_i=0$, the same bound follows from
\eqref{eq:weighted-trace-free-cauchy}, since $\widehat\lambda_i^2\leq5/6$
and $c_5/4-c_6>0$.
The terms with three distinct indices are also nonnegative.
Applying these lower bounds to
\eqref{eq:spectral-collected-expansion} and using
\eqref{eq:spectral-LQ} therefore gives
\begin{equation}\label{eq:spectral-direction-bound}
\begin{aligned}
 \widetilde{\mathscr P}_\sigma
 &\geq\frac{\sigma(5-2\sigma)}{10}
 +\rho\sum_i\nu_i^2L_{\sigma,i}+c_6\rho^2
 +|Z_\sigma|^2+\rho\langle\mathcal N,Z_\sigma\rangle\\
 &=\sum_i\nu_i^2\,
 \widetilde{\mathscr F}_\sigma(\widehat{\boldsymbol\lambda},i,\rho)
 +\left|Z_\sigma+\frac\rho2\mathcal N\right|^2
 \geq\sum_i\nu_i^2\,
 \widetilde{\mathscr F}_\sigma(\widehat{\boldsymbol\lambda},i,\rho).
\end{aligned}
\end{equation}
Here $\operatorname{tr}\mathcal N=0$ by \eqref{eq:spectral-N}, and we used
$\sum_i\nu_i^2=1$ and \eqref{eq:spectral-Q}.
This proves the claim for every unit direction $\nu$.
\end{proof}

\begin{lemma}\label{lem:spectral-scalar-lower}
Let $\widehat{\boldsymbol\lambda}$ be a normalized trace-free spectrum.
For $i=1,\ldots,6$, $2.4962\leq\sigma\leq5/2$, and $\rho\geq0$, we have
\begin{equation}\label{eq:spectral-scalar-target}
 \widetilde{\mathscr F}_\sigma(\widehat{\boldsymbol\lambda},i,\rho)
 \geq\frac{5/2-\sigma}{3800000}+10^{-8}(\rho+\rho^2).
\end{equation}
\end{lemma}
\begin{proof}
This result is verified in Lean. See Appendix~\ref{sec:coefficients-verification}.
\end{proof}

\begin{proof}[Proof of Proposition~\ref{prop:endpoint}]
Use Lemmas~\ref{lem:spectral-tensor-reduction} and
\ref{lem:spectral-scalar-lower} with $\sum_{i=1}^6\nu_i^2=1$.
\end{proof}

\section{Critical energies and rigidity}\label{sec:critical-rigidity}
We first prove finiteness of the critical curvature energies and then show that they vanish.

\subsection{Critical energies}\label{sec:energy-section}
All tensor identities in this section are used at $\sigma=5/2$.
We keep the coefficients in \eqref{eq:spectral-coefficients} fixed throughout the rest of the proof. Set $\kappa:=10^{-8}$.
Proposition~\ref{prop:endpoint} at the critical exponent gives
\begin{equation}\label{eq:critical-coercivity}
 \widetilde{\mathscr P}_\sigma\geq\kappa(\rho+\rho^2).
\end{equation}
The subcritical moments enter through an exponentially decaying weight and
we use \eqref{eq:critical-coercivity} at $\sigma=5/2$ itself.
Recall that $d\mu^M_{5/2}=|\nabla G|^4G^{-7/2}\,dV_g$.
\begin{proposition}\label{prop:energies}
The two global curvature energies
\begin{align*}
 \mathcal E_1
 &:=\int |A|^2|\nabla G|^2G^{-3/2}\,dV_g
 =\int\rho\,d\mu^M_{5/2},\\
 \mathcal E_2
 &:=\int |A|^4G^{1/2}\,dV_g
 =\int\rho^2\,d\mu^M_{5/2}
\end{align*}
are finite. Here all the integrals are on $M\setminus\{o\}$.
\end{proposition}
\begin{proof}
Set $s:=\log(t_0/t)$, so $D=-\partial_s$ on scalar test functions.
On $M$, we evaluate $s$ at $t=G$.

For $\psi\in C_c^\infty(\mathbb R)$, the argument for the second inequality in
\eqref{eq:subcritical-integrated-bound}, applied at $\sigma=5/2$ with weight
$\psi^2$ in place of $\phi^4$, gives
\[
 \int\widetilde{\mathscr P}_{5/2}\psi^2\,d\mu^M_{5/2}
 \leq\mathcal R_2+c_3\mathcal R_3+c_5\mathcal R_5+c_6\mathcal R_6
 +c_7\mathcal R_7+\sum_{j=10}^{12}c_j\mathcal R_j.
\]
Unlike the subcritical argument, the lower bound now has no positive
constant term. We therefore keep the sign of $\mathcal R_2$ and the
factor $\rho$ in the other cutoff terms, rather than use the coarse
moment bounds from \eqref{eq:subcritical-integrated-bound}.
We use formula~\eqref{eq:sources} with $D=-\partial_s$ to obtain
\begin{equation}\label{eq:critical-exact-bochner-remainder}
 \mathcal R_2=\frac12\int\bigl[\partial_s^2(\psi^2)+\partial_s(\psi^2)\bigr]
                                    \,d\mu^M_{5/2}.
\end{equation}
Using \eqref{eq:F5-remainder} with $\phi^2$ replaced by $\psi$ and
$\sigma=5/2$, together with
\eqref{eq:curvature-direction-remainder}, \eqref{eq:simons-remainder},
\eqref{eq:curv-radial}, and \eqref{eq:additional-curvature-errors},
we obtain, using $|\widehat A|=1$ and $|T_j|\leq C|A|^2$, for
$i\in\{3,5,6,7,10,11,12\}$,
\begin{equation}\label{eq:critical-other-remainders}
 |\mathcal R_i|\leq C\int\rho\bigl(|\psi\partial_s\psi|+|\partial_s\psi|^2
                     +|\psi\partial_s^2\psi|\bigr)\,d\mu^M_{5/2}.
\end{equation}
Here we used $\partial_s(\psi^2)=2\psi\partial_s\psi$ and
$\partial_s^2(\psi^2)=2|\partial_s\psi|^2+2\psi\partial_s^2\psi$.
The constant $C$ is independent of $\psi$.
Combining this estimate with \eqref{eq:critical-coercivity} gives
\begin{equation}\label{eq:endpoint-source}
\begin{aligned}
 &\kappa\int(\rho+\rho^2)\psi^2\,d\mu^M_{5/2}
 \leq\frac12\int\bigl[\partial_s^2(\psi^2)+\partial_s(\psi^2)\bigr]\,d\mu^M_{5/2}\\
 &\qquad+C\int\rho\bigl(|\psi\partial_s\psi|+|\partial_s\psi|^2
                                  +|\psi\partial_s^2\psi|\bigr)\,d\mu^M_{5/2}.
\end{aligned}
\end{equation}
Choose $0\leq\theta,\zeta_R\leq1$ with $\theta=0$ on $s\leq-1$,
$\theta=1$ on $s\geq0$, $\zeta_R=1$ on $s\leq R$, and
$\zeta_R=0$ on $s\geq R+1$, where $R\geq1$.
Their first two derivatives are uniformly bounded. For $0<\varepsilon<1$, set
\[
 \psi_R:=\theta\zeta_R^2e^{-\varepsilon s/2},\qquad
 K:=\{-1<s<0\}=\{t_0<G<e\,t_0\}.
\]
The set $K$ has compact closure away from the pole.
For each fixed $0<\varepsilon<1$, the weight identity
\begin{equation}\label{eq:critical-weighted-measure}
 e^{-\varepsilon s}\,d\mu^M_{5/2}
 =t_0^{-\varepsilon}\,d\mu^M_{5/2-\varepsilon}
\end{equation}
and Corollary~\ref{cor:all-moments} give
\begin{equation}\label{eq:critical-weighted-mass}
 \int_{\{s\geq0\}}e^{-\varepsilon s}\,d\mu^M_{5/2}
 =t_0^{-\varepsilon}\mathcal M_{5/2-\varepsilon}<\infty.
\end{equation}
On $\{s\geq0\}$, direct differentiation gives
\[
\begin{aligned}
 &(\partial_s^2+\partial_s)(\zeta_R^4e^{-\varepsilon s})\\
 ={}&e^{-\varepsilon s}\bigl[(\varepsilon^2-\varepsilon)\zeta_R^4
 +4(1-2\varepsilon)\zeta_R^3\partial_s\zeta_R
 +12\zeta_R^2(\partial_s\zeta_R)^2+4\zeta_R^3\partial_s^2\zeta_R\bigr].
\end{aligned}
\]
The first term in brackets is nonpositive, and all the other terms
are supported in $\{R<s<R+1\}$. Also,
\[
 |\psi_R\partial_s\psi_R|+|\partial_s\psi_R|^2+|\psi_R\partial_s^2\psi_R|
 \leq C(\varepsilon+\varepsilon^2)\zeta_R^4e^{-\varepsilon s}
 +C\zeta_R^2e^{-\varepsilon s}\mathbf1_{\{R<s<R+1\}}.
\]
In the last term, Young's inequality gives
\[
 C\rho\zeta_R^2\leq\frac\kappa2\rho^2\zeta_R^4+C_\kappa.
\]
Apply \eqref{eq:endpoint-source} with $\psi=\psi_R$.
The integrals over $K$ are bounded by a constant $C_K$ independent
of $R$ and $0<\varepsilon<1$. Choose $\varepsilon$ sufficiently small
that $C(\varepsilon+\varepsilon^2)\leq\kappa/2$.
Absorbing the two curvature errors into the left-hand side yields
\begin{equation}\label{eq:uniform-critical}
 \frac\kappa2\int_{\{s\geq0\}}(\rho+\rho^2)\zeta_R^4e^{-\varepsilon s}
 \,d\mu^M_{5/2}
 \leq C_K+C_\kappa\int_{\{R<s<R+1\}}e^{-\varepsilon s}\,d\mu^M_{5/2}.
\end{equation}
Here $C_K$ collects all integral contributions from the fixed transition region
$\{-1\leq s\leq0\}$. This region is compact and lies away from the pole,
and the relevant densities are locally integrable, so $C_K<\infty$.
By \eqref{eq:critical-weighted-mass}, the tail integral on the right of \eqref{eq:uniform-critical}
tends to zero as $R\to\infty$. Fatou's lemma, first as $R\to\infty$
and then as $\varepsilon\downarrow0$, gives
\[
 \int_{\{G\leq t_0\}}\rho\,d\mu^M_{5/2}<\infty,
 \qquad \int_{\{G\leq t_0\}}\rho^2\,d\mu^M_{5/2}<\infty.
\]

Near the pole, write $r=d_g(o,\cdot)$. The Green asymptotics
\eqref{eq:pole-asymptotic} and local boundedness of $A$ give
\[
 |A|^2|\nabla G|^2G^{-3/2}=O(r^{-4}),\qquad
 |A|^4G^{1/2}=O(r^{-2}),\qquad dV_g\leq Cr^5\,dr\,d\omega.
\]
Thus both densities are integrable on $B_{r_0}(o)$ for sufficiently
small $r_0>0$. The set $\{G\geq t_0\}\setminus B_{r_0}(o)$ is compact
and avoids the pole. Together with the integrability on $\{G\leq t_0\}$,
this proves $\mathcal E_1,\mathcal E_2<\infty$.
\end{proof}

\subsection{The Green boundary term and sharp isoperimetry}\label{sec:flux-section}
The sharp isoperimetric inequality of Brendle~\cite{Brendle2021} gives
\begin{equation}\label{eq:isoperimetric}
 |\partial\Sigma|\geq6\left(\frac{\omega_5}{6}\right)^{1/6}|\Sigma|^{5/6}
\end{equation}
for every smooth compact immersed minimal hypersurface $\Sigma^6$ in $\mathbb R^7$.
We apply this to the Green superlevel sets.

Recall that $F_2=|Z_{5/2}|^2-\rho|\widehat A\nu|^2$ at $\sigma=5/2$.
\begin{lemma}\label{lem:flux}
The function $\mathcal J(t):=t^{-5/2}J(t)$ has finite limits
$\mathcal J_0$ as $t\downarrow0$ and $\mathcal J_{\rm pole}$ as $t\to\infty$.
Moreover,
\begin{equation}\label{eq:flux-sign}
 \int_M F_2\,d\mu^M_{5/2}
 =\frac{\mathcal J_{\rm pole}-\mathcal J_0}{2}\leq0,
 \qquad \mathcal J_{\rm pole}=32\sqrt{\omega_5}.
\end{equation}
\end{lemma}
\begin{proof}
\textbf{Step 1. The boundary identity.}
We first establish integrability of the Hessian energy density and prove that
\begin{equation}\label{eq:flux-boundary-identity}
 \mathcal J_{\rm pole}-\mathcal J_0
 =2\int_M F_2\,d\mu^M_{5/2}.
\end{equation}
Set $\mathcal K(t):=t\mathcal J(t)=t^{-3/2}J(t)$.
Recall also that $d\mu^M_{5/2}=|\nabla G|^4G^{-7/2}\,dV_g$.
For $v\in C_c^\infty((0,\infty))$, replacing $\phi^4$ by $v$ in
\eqref{eq:sources} gives
\begin{equation}\label{eq:flux-weak-t}
\begin{aligned}
 \int_M F_2v(G)\,d\mu^M_{5/2}
 &=\frac12\int_M[D^2v-Dv](G)\,d\mu^M_{5/2}\\
 &=\frac12\int_0^\infty t^2v''(t)J(t)t^{-7/2}\,dt
 =\frac12\int_0^\infty\mathcal K(t)v''(t)\,dt.
\end{aligned}
\end{equation}
Here $D=t\partial_t$, so $D^2v-Dv=t^2v''$.
The weighted densities $|Z_{5/2}|^2d\mu^M_{5/2}$ and
$\rho|\widehat A\nu|^2d\mu^M_{5/2}$ are locally integrable on
$M\setminus\{o\}$.
Applying the coarea formula to the left-hand side of
\eqref{eq:flux-weak-t} therefore gives, in distributions,
\begin{align}
 \mathcal K''&=\mathrm I_Z-\mathrm I_A,\label{eq:flux-second-derivative}\\
 \mathrm I_Z(t)&:=2t^{-7/2}\int_{\{G=t\}}|Z_{5/2}|^2|\nabla G|^3\,dA_g,
 \label{eq:flux-hessian-density}\\
 \mathrm I_A(t)&:=2t^{-7/2}\int_{\{G=t\}}
 \rho|\widehat A\nu|^2|\nabla G|^3\,dA_g.
 \label{eq:flux-curvature-density}
\end{align}
In particular, $\mathrm I_Z,\mathrm I_A\geq0$ and $\mathcal K\in W^{2,1}_{\rm loc}((0,\infty))$.
Since $\mathcal K'\in W^{1,1}_{\rm loc}((0,\infty))$, it has a locally
absolutely continuous representative. We use the resulting $C^1$
representative of $\mathcal K$, which agrees with $t^{-3/2}J(t)$
on intervals of regular values.
By \eqref{eq:flux-curvature-density}, Proposition~\ref{prop:energies}, and
$|\widehat A\nu|^2\leq5/6$, we have
\[
 \int_0^\infty \mathrm I_A(t)\,dt
 =2\int_M\rho|\widehat A\nu|^2\,d\mu^M_{5/2}
 \leq\frac53\mathcal E_1<\infty.
\]
Define
\[
 \mathcal W(t):=\mathcal K(t)+\int_0^t(t-u)\mathrm I_A(u)\,du.
\]
By \eqref{eq:flux-second-derivative}, $\mathcal W''=\mathrm I_Z\geq0$,
so $\mathcal W$ is convex.
Also, $\mathcal W\geq0$, and \eqref{eq:J-quadratic-bound} gives
$\mathcal K(t)=O(t^{1/2})\to0$ as $t\downarrow0$.
Since $0\leq\int_0^t(t-u)\mathrm I_A(u)\,du\leq t\int_0^t \mathrm I_A(u)\,du$,
the continuous extension satisfies $\mathcal W(0)=0$.
The pole asymptotics \eqref{eq:pole-asymptotic} give
$\mathcal K(t)/t\to\mathcal J_{\rm pole}:=32\sqrt{\omega_5}$ as $t\to\infty$.
For $0<u<t$, we have $0\leq1-u/t\leq1$, and for each fixed $u>0$,
$1-u/t\to1$ as $t\to\infty$. Since $\mathrm I_A\in L^1(0,\infty)$,
extending the integrand by zero for $u\geq t$ and applying dominated
convergence gives
\begin{equation}\label{eq:W-asymptotic-slope}
 \frac{\mathcal W(t)}t
 =\frac{\mathcal K(t)}t+\int_0^t\left(1-\frac ut\right)\mathrm I_A(u)\,du
 \xrightarrow[t\to\infty]{}\mathcal J_{\rm pole}+\int_0^\infty \mathrm I_A(u)\,du<\infty.
\end{equation}
Convexity now yields
\begin{equation}\label{eq:convex-slope-bound}
 0\leq\frac{\mathcal W(t)}t\leq\mathcal W'(t)
 \leq\lim_{T\to\infty}\frac{\mathcal W(T)-\mathcal W(t)}{T-t}
 =\mathcal J_{\rm pole}+\int_0^\infty \mathrm I_A(u)\,du.
\end{equation}
Thus $\mathcal W'$ is nondecreasing and bounded, so both
$\mathcal W'(0+):=\lim_{t\downarrow0}\mathcal W'(t)$ and
$\mathcal W'(\infty):=\lim_{t\to\infty}\mathcal W'(t)$ exist and are finite.
Combining \eqref{eq:W-asymptotic-slope} with \eqref{eq:convex-slope-bound} gives
$\mathcal W'(\infty)=\mathcal J_{\rm pole}+\int_0^\infty \mathrm I_A(u)\,du$.
Since $\mathcal W''=\mathrm I_Z\geq0$ and $\mathcal W'$ is bounded,
the coarea formula applied to \eqref{eq:flux-hessian-density} and the
fundamental theorem of calculus give
\[
 2\int_M|Z_{5/2}|^2\,d\mu^M_{5/2}
 =\int_0^\infty \mathrm I_Z(t)\,dt<\infty.
\]
In particular,
\begin{equation}\label{eq:critical-F2-absolute-integrability}
 \int_M|F_2|\,d\mu^M_{5/2}
 \leq\int_M|Z_{5/2}|^2\,d\mu^M_{5/2}+\frac56\mathcal E_1<\infty.
\end{equation}
Consequently, $\mathcal K''=\mathrm I_Z-\mathrm I_A\in L^1(0,\infty)$, and
$\mathcal K'(t)=\mathcal W'(t)-\int_0^t \mathrm I_A(u)\,du$ gives
\[
 \mathcal K'(0+)=\mathcal W'(0+),\qquad
 \mathcal K'(\infty)=\mathcal J_{\rm pole}.
\]
L'H\^opital's rule at $0$ and $\infty$ therefore gives the finite limits
\[
 \mathcal J_0=\mathcal K'(0+),\qquad
 \lim_{t\to\infty}\mathcal J(t)=\mathcal J_{\rm pole}.
\]
Integrating \eqref{eq:flux-second-derivative} and using
\eqref{eq:flux-hessian-density}--\eqref{eq:flux-curvature-density} gives
\[
\begin{aligned}
 \mathcal J_{\rm pole}-\mathcal J_0
 &=\lim_{t\to\infty}\mathcal K'(t)-\lim_{t\downarrow0}\mathcal K'(t)
 =\int_0^\infty\mathcal K''(t)\,dt\\
 &=2\int_M|Z_{5/2}|^2\,d\mu^M_{5/2}
 -2\int_M\rho|\widehat A\nu|^2\,d\mu^M_{5/2}
 =2\int_M F_2\,d\mu^M_{5/2},
\end{aligned}
\]
which proves \eqref{eq:flux-boundary-identity}.

\medskip
\textbf{Step 2. Comparison of $\mathcal J_{\rm pole}$ and $\mathcal J_0$.}
We will show that
\[
 \mathcal J_{\rm pole}\leq\mathcal J_0.
\]
Let $\mathscr V(t):=|\{G>t\}\cup\{o\}|$.
For $0<a<b$, the coarea formula and Remark~\ref{rem:zeros} give
\[
 \mathscr V(a)-\mathscr V(b)
 =\int_{\{a<G<b\}}dV_g
 =\int_a^b\left(\int_{\{G=u\}}|\nabla G|^{-1}\,dA_g\right)du.
\]
Thus $\mathscr V$ is locally absolutely continuous on $(0,\infty)$, and
\[
 -\mathscr V'(t)=\int_{\{G=t\}}|\nabla G|^{-1}\,dA_g
 \qquad\text{for almost every }t>0.
\]
At regular values, the unit flux identity \eqref{eq:Green-flux} and
the Cauchy--Schwarz inequality give
\[
\begin{aligned}
 |\partial\{G>t\}|^2
 &=\left(\int_{\{G=t\}}|\nabla G|^{1/2}|\nabla G|^{-1/2}\,dA_g\right)^2\\
 &\leq\left(\int_{\{G=t\}}|\nabla G|\,dA_g\right)
       \left(\int_{\{G=t\}}|\nabla G|^{-1}\,dA_g\right)
 =-\mathscr V'(t).
\end{aligned}
\]
Combining this with \eqref{eq:isoperimetric} and differentiating
$\mathscr V^{-2/3}$ yields, almost everywhere,
\[
\begin{aligned}
 -\mathscr V'(t)&\geq|\partial\{G>t\}|^2
 \geq36(\omega_5/6)^{1/3}\mathscr V(t)^{5/3},\\
 (\mathscr V^{-2/3})'(t)
 &=\frac23\mathscr V(t)^{-5/3}(-\mathscr V'(t))
 \geq24(\omega_5/6)^{1/3}.
\end{aligned}
\]
Integrating the last inequality from $a$ to $t$, discarding
$\mathscr V(a)^{-2/3}\geq0$, and letting $a\downarrow0$ gives
\begin{equation}\label{eq:sharp-G-volume}
 \mathscr V(t)\leq\frac{t^{-3/2}}{48\sqrt{\omega_5}}.
\end{equation}
We apply the Cauchy--Schwarz inequality again to obtain
\[
\begin{aligned}
 1&=\left(\int_{\{G=t\}}|\nabla G|\,dA_g\right)^2
 \leq\left(\int_{\{G=t\}}|\nabla G|^3\,dA_g\right)
      \left(\int_{\{G=t\}}|\nabla G|^{-1}\,dA_g\right)\\
 &=J(t)(-\mathscr V'(t)).
\end{aligned}
\]
Given $\eta>0$, there exists $t_1>0$ such that
$J(u)\leq(\mathcal J_0+\eta)u^{5/2}$ for almost every $0<u<t_1$,
by the result of Step~1.
Hence $-\mathscr V'(u)\geq 1/J(u)\geq u^{-5/2}/(\mathcal J_0+\eta)$
for almost every $0<u<t_1$. Integrating this inequality with respect to $u$
over $[t,t_1]$, where $0<t<t_1$, gives
\[
 \mathscr V(t)\geq\mathscr V(t_1)
 +\frac{1}{\mathcal J_0+\eta}\int_t^{t_1}u^{-5/2}\,du
 =\mathscr V(t_1)+\frac{2(t^{-3/2}-t_1^{-3/2})}
                         {3(\mathcal J_0+\eta)}.
\]
Multiplying by $t^{3/2}$, using \eqref{eq:sharp-G-volume}, and
letting $t\downarrow0$ gives
\[
 \frac{2}{3(\mathcal J_0+\eta)}
 \leq\liminf_{t\downarrow0}t^{3/2}\mathscr V(t)
 \leq\frac1{48\sqrt{\omega_5}}.
\]
Letting $\eta\downarrow0$ proves
$\mathcal J_0\geq32\sqrt{\omega_5}=\mathcal J_{\rm pole}$
and completes the proof.
\end{proof}

\subsection{Rigidity}\label{sec:rigidity-section}
\begin{proof}[Proof of Theorem~\ref{thm:main}]
By Lemma~\ref{lem:reduction}, we may assume that $M$ is nonflat,
has bounded curvature, and satisfies $|A|(o)=1$.
Use $s=\log(t_0/t)$, evaluated at $t=G$ on $M$.
Choose $0\leq\chi\in C_c^\infty(\mathbb R)$ with
$\chi\leq1$ and $\chi=1$ on $[-1,1]$, and set
\[
 \psi_R(s):=\chi(s/R),\qquad R\geq1.
\]
Then $\psi_R^2\to1$ pointwise, $|\partial_s\psi_R|\leq C/R$, and
$|\partial_s^2\psi_R|\leq C/R^2$.
By \eqref{eq:critical-exact-bochner-remainder} and the definition of $F_2$,
the first integral on the right of
\eqref{eq:endpoint-source}, with $\psi=\psi_R$, equals
$\int_MF_2\psi_R^2\,d\mu^M_{5/2}$.
Since $\mathcal E_1<\infty$ by Proposition~\ref{prop:energies},
\eqref{eq:endpoint-source} therefore gives
\begin{equation}\label{eq:rigidity-cutoff-bound}
 \kappa\int_M(\rho+\rho^2)\psi_R^2\,d\mu^M_{5/2}
 \leq\int_MF_2\psi_R^2\,d\mu^M_{5/2}+\frac C R\mathcal E_1.
\end{equation}
Letting $R\to\infty$, using Fatou's lemma on the left, dominated
convergence by \eqref{eq:critical-F2-absolute-integrability} on the right,
and \eqref{eq:flux-sign}, we obtain
\begin{equation}\label{eq:rigidity-limit-bound}
 0\leq\kappa(\mathcal E_1+\mathcal E_2)
 \leq\int_MF_2\,d\mu^M_{5/2}
 =\frac{\mathcal J_{\rm pole}-\mathcal J_0}{2}\leq0.
\end{equation}
Thus $\mathcal E_2=\int_{M\setminus\{o\}}|A|^4G^{1/2}\,dV_g=0$.
The positivity of $G$ and smoothness of $A$ imply $A\equiv0$,
contrary to $|A|(o)=1$.
The original immersion is therefore totally geodesic.
\end{proof}

\appendix
\section{Algebraic estimates and formal verification}\label{sec:coefficients-verification}
The next two subsections give the parameters, complete scalar definitions,
and lower bounds proved in Lean and independently verified in Mathematica.
All variables here are algebraic, all finite decimals denote exact rational
numbers, and square roots are nonnegative.

The Lean and Mathematica files are available in the supplementary repository:
\begin{center}
\mbox{\url{https://github.com/wgaom/stable-bernstein-R7}}.
\end{center}

Install \texttt{elan}, then run the following commands.
\begin{codeblock}
git clone https://github.com/wgaom/stable-bernstein-R7.git
cd stable-bernstein-R7/lean
lake exe cache get
lake build
\end{codeblock}
A successful build
checks both theorem proofs and prints their axiom dependencies.
The definitions and theorem statements are in
\nolinkurl{lean/R7/Definitions.lean} and \nolinkurl{lean/R7/Theorems.lean},
respectively. After completing this build, compare these files with the
paper using the checks following each lemma below.

\subsection{Nonlinear estimate}\label{sec:rows}
For each row of Table~\ref{tab:nonlinear-rows-first}, keep the coefficients
\[
 c_3,c_4,c_6,c_7,c_8,m_{11},m_{22},m_{12}
\]
fixed and denote its interval by
$[\sigma_l,\sigma_r]$. The variables have ranges
\begin{equation}\label{eq:app-nonlinear-domain}
 \sigma\in[\sigma_l,\sigma_r],\qquad
 0\leq q\leq\sqrt{5/6},\qquad \rho\geq0.
\end{equation}
\begin{table}[htbp]
\caption{Intervals and coefficients for the nonlinear estimate.}
\label{tab:nonlinear-rows-first}
\centering
\setlength{\tabcolsep}{4pt}
\renewcommand{\arraystretch}{1.15}
\begin{tabular}{@{}r@{\hspace{6pt}$[$}
 S[table-format=1.4]@{\,$,$\,}S[table-format=1.4]@{$]$\hspace{6pt}}
 S[table-format=1.3] S[table-format=1.4] S[table-format=1.3]
 S[table-format=-1.3] S[table-format=-1.4]
 S[table-format=1.4] S[table-format=1.3] S[table-format=1.4]@{}}
\toprule
\multicolumn{1}{c}{Row} & \multicolumn{2}{c}{$[\sigma_l,\sigma_r]$}
 & {$c_3$} & {$c_4$} & {$c_6$} & {$c_7$} & {$c_8$}
 & {$m_{11}$} & {$m_{22}$} & {$m_{12}$} \\
\midrule
1 & 2.46 & 2.4845 & 1.323 & 0.0304 & 0.249 & -0.518 & -0.0339 & 0.0282 & 1.667 & 0.0097 \\
2 & 2.4845 & 2.4895 & 1.606 & 0.0196 & 0.344 & -1.22 & -0.0229 & 0.0199 & 3.335 & 0.0075 \\
3 & 2.4895 & 2.492 & 1.708 & 0.0154 & 0.389 & -1.485 & -0.0199 & 0.0155 & 4.046 & 0.005 \\
4 & 2.492 & 2.4933 & 1.762 & 0.0131 & 0.414 & -1.628 & -0.0186 & 0.013 & 4.441 & 0.0034 \\
5 & 2.4933 & 2.4941 & 1.793 & 0.0104 & 0.429 & -1.717 & -0.0184 & 0.0115 & 4.693 & 0.0023 \\
6 & 2.4941 & 2.4947 & 1.816 & 0.01 & 0.44 & -1.783 & -0.0182 & 0.0104 & 4.884 & 0.0014 \\
7 & 2.4947 & 2.4951 & 1.837 & 0.0092 & 0.45 & -1.835 & -0.0178 & 0.0096 & 5.026 & 0.0007 \\
8 & 2.4951 & 2.4954 & 1.841 & 0.0091 & 0.452 & -1.864 & -0.017 & 0.009 & 5.119 & 0.0002 \\
9 & 2.4954 & 2.4957 & 1.857 & 0.0078 & 0.46 & -1.896 & -0.0169 & 0.0084 & 5.202 & 0 \\
10 & 2.4957 & 2.496 & 1.861 & 0.0065 & 0.462 & -1.921 & -0.0157 & 0.0079 & 5.276 & 0 \\
11 & 2.496 & 2.4962 & 1.873 & 0.0068 & 0.468 & -1.939 & -0.016 & 0.0074 & 5.325 & 0 \\
\bottomrule
\end{tabular}
\end{table}

Define
\begin{equation}\label{eq:app-nonlinear-matrix}
 \mathsf M:=\begin{pmatrix}m_{11}&m_{12}\\m_{12}&m_{22}\end{pmatrix},
 \qquad \det\mathsf M=m_{11}m_{22}-m_{12}^2.
\end{equation}
The coefficients in every row satisfy
\begin{equation}\label{eq:app-matrix-signs}
\begin{gathered}
 m_{22}>4c_{6}>0,\quad 0<m_{11}<4/5,\quad m_{12}\geq0,\quad \det\mathsf M>0,\\
 m_{12}^2<\frac{m_{22}-4c_{6}}{3}\left(\frac65-m_{11}\right).
\end{gathered}
\end{equation}
The scalar functions used below are
\begin{equation}\label{eq:app-nonlinear-denominators}
\begin{aligned}
 \beta(\rho)&:=\frac{\det\mathsf M}{m_{11}+2m_{12}\sqrt\rho+m_{22}\rho},\\
 \widetilde m_{11}(\rho)&:=m_{11}-\rho\beta(\rho),\qquad&
 \widetilde m_{12}(\rho)&:=m_{12}+\sqrt\rho\beta(\rho),\\
 \alpha(q,\rho)&:=\frac{m_{22}-c_6}{\frac13+\frac{q^2}2}-m_{22}+\beta(\rho),&
 \eta(q,\rho)&:=\frac65-\widetilde m_{11}(\rho)
 -\frac{\widetilde m_{12}(\rho)^2}{\alpha(q,\rho)}.
\end{aligned}
\end{equation}
The sign conditions give $\beta(\rho)>0$, $\alpha(q,\rho)>0$, and
$\eta(q,\rho)>0$ throughout \eqref{eq:app-nonlinear-domain}.
For fixed coefficients, $\sigma$, and $\rho$, and for $u\in\mathbb R$, define
\begin{equation}\label{eq:app-nonlinear-coefficients}
\begin{aligned}
 B_0(u):={}&\frac{\sigma(5-2\sigma)}{10}-\frac{m_{11}}4
 +c_6\rho^2+2m_{12}\rho^{3/2}\\
 &{}+\left[c_8\left(1-\frac\sigma2\right)-\frac{m_{12}(3-\sigma)}2\right]\sqrt\rho
 +\rho\beta(\rho)\left(1-\frac\sigma2\right)^2\\
 &{}+\rho\biggl[\left(2c_3\left(1-\frac\sigma5\right)-1\right)u
 -c_3\left(1-\frac{2\sigma}5\right)+m_{11}\\
 &\qquad{}-\frac{c_6}2(3-\sigma)(\sigma-2)
 -\frac{m_{22}}4(3-\sigma)^2-\frac{c_7}2(\sigma-2)\biggr],\\
 B_1:={}&c_7\sqrt\rho+c_8-m_{12}
 +2\sqrt\rho\beta(\rho)\left(1-\frac\sigma2\right),\\
 B_2:={}&c_4+\bigl(c_8-m_{12}(3-\sigma)\bigr)\sqrt\rho
 -2\rho\beta(\rho)\left(1-\frac\sigma2\right).
\end{aligned}
\end{equation}
The remaining scalar function is
\begin{equation}\label{eq:app-nonlinear-envelope}
 \Lambda(q):=\frac{207}{125}q^4-\frac{33}{25}q^2+\frac{29}{60}
 +\frac6{125}q(5-6q^2)^{3/2}.
\end{equation}
With the functions just defined, set
\begin{align}\label{eq:app-nonlinear-function}
 \mathscr F_\sigma(q,\rho):={}& B_0(q^2)-\frac{B_1^2}{4\alpha(q,\rho)}
 -\frac{c_{3}^2}4\left[\Lambda(q)-\frac65(q^2-\frac16)^2\right]\rho^2\notag\\
 &-\frac{\left[\frac65c_{3}(q^2-\frac16)\rho+B_2+\frac{\widetilde m_{12}(\rho) B_1}{\alpha(q,\rho)}\right]^2}{4\eta(q,\rho)}.
\end{align}

\begin{nonlinearendpointlemma}[restated]
For each of the eleven rows of Table~\ref{tab:nonlinear-rows-first}, use its
coefficients in the definitions above. These coefficients satisfy
\eqref{eq:app-matrix-signs}. For every
$\sigma\in\{\sigma_l,\sigma_r\}$, $0\leq q\leq\sqrt{5/6}$, and $\rho\geq0$,
\begin{equation}\label{eq:app-nonlinear-target}
 \mathscr F_\sigma(q,\rho)\geq10^{-7}(1+\rho+\rho^2).
\end{equation}
\end{nonlinearendpointlemma}
\smallskip
\noindent\emph{Formal verification.}
The correspondence with the Lean theorem can be checked as follows.
\begin{enumerate}
\item In \nolinkurl{Definitions.lean}, check that the intervals
$[\sigma_l,\sigma_r]$ encoded by \texttt{rows} cover $[2.46,2.4962]$.
The coefficients are not uniquely determined, and Table~\ref{tab:nonlinear-rows-first}
lists the values used here.
\item Compare \texttt{detM}, \texttt{beta}, \texttt{m11tilde},
\texttt{m12tilde}, \texttt{alpha}, \texttt{eta}, \texttt{B0},
\texttt{B1}, \texttt{B2}, \texttt{Lambda}, and \texttt{scalarF}
with \eqref{eq:app-nonlinear-matrix} and
\eqref{eq:app-nonlinear-denominators}--\eqref{eq:app-nonlinear-function}.
\item In \nolinkurl{Theorems.lean}, compare the statement of
\nolinkurl{R7.Paper.scalar_bound} with both \eqref{eq:app-matrix-signs}
and the lower bound in the lemma above. For each \texttt{row}, the first
part of the conclusion gives the sign conditions, and the second gives
the lower bound for both endpoints. The indices \texttt{row} and
\texttt{side} range over all eleven rows and their two endpoints, respectively. 
\item In the build output, inspect the axiom dependencies printed by
\begin{Verbatim}[fontsize=\small]
#print axioms R7.Paper.scalar_bound
\end{Verbatim}
Check that the only dependencies are \nolinkurl{propext},
\nolinkurl{Classical.choice}, and \nolinkurl{Quot.sound}, with no
\nolinkurl{sorryAx} or native-computation axiom.
\end{enumerate}
The first three checks identify the formal statement with this lemma,
and compilation checks its proof in the Lean kernel.

Proposition~\ref{prop:bridge} uses concavity in $\sigma$ to extend this
endpoint estimate to each whole interval.

\subsection{Spectral estimate}\label{app:terminal-verification}
In this subsection the coefficients are fixed by
\begin{equation}\label{eq:app-terminal-coefficients}
 (c_3,c_5,c_6,c_7,c_{10},c_{11},c_{12})
 =(2.67,6.04,0.52,-4.18,-1.25,0.37,1.98).
\end{equation}
Let $\widehat{\boldsymbol\lambda}=(\widehat\lambda_1,\ldots,\widehat\lambda_6)
\in\mathbb R^6$ satisfy
\begin{equation}\label{eq:app-terminal-constraints}
 \sum_{j=1}^6\widehat\lambda_j=0,\qquad
 \sum_{j=1}^6\widehat\lambda_j^2=1.
\end{equation}
The other variables have ranges
\begin{equation}\label{eq:app-terminal-domain}
 i\in\{1,\ldots,6\},\qquad 2.4962\leq\sigma\leq5/2,\qquad \rho\geq0.
\end{equation}
All indices $j$ below range from $1$ to $6$. Define
\begin{equation}\label{eq:app-terminal-moments}
 s_k(\widehat{\boldsymbol\lambda}):=\sum_{j=1}^6\widehat\lambda_j^k
 \qquad(k=3,4).
\end{equation}
For each $j$ and $i$, set
\begin{equation}\label{eq:app-spectral-N}
 \mathcal N_j:=c_3\widehat\lambda_j^2+c_{10}\widehat\lambda_j^4
 +c_{11}s_3(\widehat{\boldsymbol\lambda})\widehat\lambda_j-\frac{c_3+c_{10}s_4(\widehat{\boldsymbol\lambda})}{6}.
\end{equation}
\begin{equation}\label{eq:app-spectral-V}
\begin{aligned}
 V_j^{(i)}:={}&(c_{10}+4c_{12})\widehat\lambda_j^3
 +(c_{10}+3c_{11})\widehat\lambda_i\widehat\lambda_j^2\\
 &+\bigl[c_7+c_{10}(\widehat\lambda_i^2-2\widehat\lambda_i^4)
 -2c_{11}s_3(\widehat{\boldsymbol\lambda})\widehat\lambda_i-2c_{12}s_4(\widehat{\boldsymbol\lambda})\bigr]\widehat\lambda_j.
\end{aligned}
\end{equation}
Define
\begin{align}
 L_{\sigma,i}:={}&2\left(1-\frac\sigma5\right)\mathcal N_i
 -\widehat\lambda_i^2-\frac{c_5}{4}(3-\sigma)^2\notag\\*
 \MoveEqLeft[3]{}+(\sigma-2)\left[\frac{c_3}{3}-\frac{c_7}{2}
 -\frac{c_6}{2}(3-\sigma)
 -\left(\frac{c_{10}}6+c_{12}\right)s_4(\widehat{\boldsymbol\lambda})\right]\label{eq:app-spectral-LQ}\\*
 \MoveEqLeft[3]{}+\min_{\substack{z\in\mathbb R^6\\\sum_{j=1}^6z_j=0}}
 \left\{(c_5-c_6)\left(z_i^2+3\sum_{j\ne i}z_j^2\right)
 -c_5\left(\sum_{j=1}^6\widehat\lambda_jz_j\right)^2
 +\sum_{j=1}^6V_j^{(i)}z_j\right\},\notag\\
 Q:={}&c_6-\frac14\sum_{j=1}^6\mathcal N_j^2.\label{eq:app-spectral-Q}
\end{align}
It is easy to verify that the minimum is attained, so it agrees with
the infimum used in the Lean definition.
Finally, define
\begin{equation}\label{eq:app-spectral-function}
 \widetilde{\mathscr F}_\sigma(\widehat{\boldsymbol\lambda},i,\rho)
 :=\frac{\sigma(5-2\sigma)}{10}+L_{\sigma,i}\rho+Q\rho^2.
\end{equation}

\begin{terminalspectrallemma}[restated]
With the constants and functions defined in this subsection, let
$\widehat{\boldsymbol\lambda}\in\mathbb R^6$ satisfy
$\sum_{j=1}^6\widehat\lambda_j=0$ and
$\sum_{j=1}^6\widehat\lambda_j^2=1$.
For every $i\in\{1,\ldots,6\}$, $2.4962\leq\sigma\leq5/2$, and $\rho\geq0$,
\begin{equation}\label{eq:app-terminal-target}
 \widetilde{\mathscr F}_\sigma(\widehat{\boldsymbol\lambda},i,\rho)
 \geq\frac{5/2-\sigma}{3800000}+10^{-8}(\rho+\rho^2).
\end{equation}
\end{terminalspectrallemma}
\smallskip
\noindent\emph{Formal verification.}
\begin{enumerate}
\item In \nolinkurl{Definitions.lean}, compare
\texttt{c3}, \texttt{c5}, \texttt{c6}, \texttt{c7}, \texttt{c10},
\texttt{c11}, and \texttt{c12} with \eqref{eq:app-terminal-coefficients}.
\item Compare \texttt{s}, \texttt{N}, \texttt{V}, \texttt{minimum},
\texttt{L}, \texttt{Q}, and \texttt{spectralF} with
\eqref{eq:app-terminal-moments}--\eqref{eq:app-spectral-Q}
and \eqref{eq:app-spectral-function}.
\item In \nolinkurl{Theorems.lean}, compare the statement of
\nolinkurl{R7.Paper.spectral_bound} with the lemma above.
Direction $i$ corresponds to Lean index $i-1$.
\item In the build output, inspect the axiom dependencies printed by
\texttt{\#print axioms R7.Paper.spectral\_bound}.
Check that its only dependencies are \nolinkurl{propext},
\nolinkurl{Classical.choice}, and \nolinkurl{Quot.sound}.
\end{enumerate}

\begingroup\emergencystretch=3em
\subsection{Independent verification in Mathematica}\label{sec:bridge-polys}
Mathematica provides an independent verification of
Lemmas~\ref{lem:nonlinear-endpoints} and~\ref{lem:spectral-scalar-lower}
using exact symbolic arithmetic and real quantifier elimination.
The parameters, definitions, and statements are those in the preceding
two subsections.

\smallskip
\noindent\emph{Verification of Lemma~\ref{lem:nonlinear-endpoints}.}
The notebook checks the matrix signs \eqref{eq:app-matrix-signs} for
every row and uses $\rho=y^2$ and
$q=\sqrt5\,x/\sqrt{1+6x^2}$, with $x,y\geq0$.
It then verifies the denominator signs and the inequalities at both
endpoints of every row. The boundary $q=\sqrt{5/6}$ is checked separately.

\smallskip
\noindent\emph{Verification of Lemma~\ref{lem:spectral-scalar-lower}.}
The notebook checks an explicit minimizer and a nonnegative quadratic
remainder for \eqref{eq:app-spectral-LQ}, then rewrites the resulting
homogeneous polynomials in nonnegative gaps between the ordered, unmarked
eigenvalues. Exact coefficient reconstruction and real quantifier elimination
verify the inequalities, with zero-gap cases checked separately.
A concavity identity extends the endpoint bounds to the whole interval in
$\sigma$. The verification includes repeated eigenvalues, all directions,
and both interval endpoints.

\smallskip
\noindent\emph{Reproduction.}
Evaluate the self-contained notebook
\nolinkurl{mathematica/Verification.nb} to run both Mathematica
verifications. The notebook reports \texttt{PASS} for each successfully
verified estimate.

\par\endgroup
\begingroup
\definecolor{collectBzero}{RGB}{0,76,153}
\definecolor{collectBone}{RGB}{164,75,0}
\definecolor{collectBtwo}{RGB}{121,45,145}
\definecolor{collectGrad}{RGB}{0,112,58}
\definecolor{collectRadial}{RGB}{122,83,32}
\definecolor{collectMixed}{RGB}{176,35,65}
\definecolor{collectHessian}{RGB}{0,112,124}
\tikzset{collect term/.style={line width=.55pt,inner xsep=2pt,inner ysep=1.3pt,
 outer sep=.3pt},collect zero/.style={sharp corners},
 collect one/.style={rounded corners=3pt,dashed},
 collect two/.style={sharp corners,double,double distance=.65pt},
 collect full/.style={rounded corners=3pt},
 collect norm/.style={sharp corners,densely dashed},
 collect mixed/.style={rounded corners=3pt,dash dot},
 collect Hessian/.style={sharp corners,densely dotted}}
\newcommand{\collectbox}[3]{\tikz[baseline=(term.base)]
 \node[collect term,#2,draw=#1,text=#1,fill=#1!3]
 (term){$\displaystyle #3$};}
\newcommand{\cBz}[1]{\collectbox{collectBzero}{collect zero}{#1}}
\newcommand{\cBo}[1]{\collectbox{collectBone}{collect one}{#1}}
\newcommand{\cBt}[1]{\collectbox{collectBtwo}{collect two}{#1}}
\newcommand{\cGa}[1]{\collectbox{collectGrad}{collect full}{#1}}
\newcommand{\cGr}[1]{\collectbox{collectRadial}{collect norm}{#1}}
\newcommand{\cMx}[1]{\collectbox{collectMixed}{collect mixed}{#1}}
\newcommand{\cHz}[1]{\collectbox{collectHessian}{collect Hessian}{#1}}
\section{Expansion of the combined expressions}\label{sec:coefficient-collections}
\subsection{Nonlinear stability expression}\label{sec:colored-coefficients}
We show how the individual integrands combine to give
\eqref{eq:nonlinear-expansion}.
Each color and frame identify the same family of terms throughout the calculation.
\begin{center}
\setlength{\tabcolsep}{8pt}
\begin{tabular}{ll}
\cBz{B_0} & Solid square frame, independent of $\nabla A$ and $Z_\sigma$ \\
\addlinespace[2pt]
\cBo{B_1} & Dashed round frame, radial gradient terms \\
\addlinespace[2pt]
\cBt{B_2} & Double square frame, $Z_\sigma(\nu,\nu)$ terms \\
\addlinespace[2pt]
\cGa{\text{Full curvature gradient}} & Terms proportional to $|\nabla A|^2$ \\
\addlinespace[2pt]
\cGr{\text{Gradient of the norm}} & Terms proportional to $|\nabla|A||^2$ \\
\addlinespace[2pt]
\cMx{\text{Mixed gradient}} & Terms proportional to $\langle\nabla|A|,Z_\sigma\nu\rangle$ \\
\addlinespace[2pt]
\cHz{\text{Hessian terms}} & Remaining quadratic or mixed terms in $Z_\sigma$
\end{tabular}
\end{center}

\subsubsection*{The six additional integrands}
The coefficients multiplying these integrands are
\[
 1,c_3,c_4,c_6,c_7,c_8,
\]
respectively. Their definitions are
\begin{align*}
 F_2={}&\cHz{|Z_\sigma|^2}
 +\cBz{\left[-\rho|\widehat A\nu|^2+\frac{\sigma(5-2\sigma)}{10}\right]},\\
 F_3={}&\cHz{\rho\langle\widehat A^2,Z_\sigma\rangle}
 +\cBz{\rho\left[2\left(1-\frac\sigma5\right)|\widehat A\nu|^2
 -\left(1-\frac{2\sigma}5\right)\right]},\\
 F_4={}&\cBt{Z_\sigma(\nu,\nu)},\\
 F_6={}&\cBz{\left[\rho^2-\frac12(3-\sigma)(\sigma-2)\rho\right]}
 -\cGa{\rho\frac{G^2|\nabla A|^2}{|A|^2|\nabla G|^2}},\\
 F_7={}&\cBo{\rho\frac{G\langle\nabla|A|,\nu\rangle}{|A||\nabla G|}}
 -\cBz{\frac{\sigma-2}{2}\rho},\\
 F_8={}&\cBo{\sqrt\rho\frac{G\langle\nabla|A|,\nu\rangle}{|A||\nabla G|}}
 +\cBt{\sqrt\rho\,Z_\sigma(\nu,\nu)}
 +\cBz{\sqrt\rho\left(1-\frac\sigma2\right)}.
\end{align*}
Since $\operatorname{tr}Z_\sigma=0$, the Hessian contraction in $F_3$
can also be written as
\[
 \cHz{\rho\langle\widehat A^2,Z_\sigma\rangle
 =\rho\langle\widehat A^2-g/6,Z_\sigma\rangle}.
\]

\subsubsection*{The matrix contributions}
The three integrands $F_1,F_5,F_9$ in \eqref{eq:matrix-block-square} are
\begin{align*}
 F_1={}&\cBz{\left[\rho-\frac14\right]}
 -\cHz{|Z_\sigma\nu|^2},\\
 F_5={}&\cGa{\rho\frac{G^2|\nabla A|^2}{|A|^2|\nabla G|^2}}
 -\cGr{\rho\frac{G^2|\nabla|A||^2}{|A|^2|\nabla G|^2}}
 -\cBz{\frac{(3-\sigma)^2}{4}\rho},\\
 F_9={}&\cBz{\left[\rho^{3/2}-\frac{3-\sigma}{4}\sqrt\rho\right]}
 -\cMx{\sqrt\rho\frac{G\langle\nabla|A|,Z_\sigma\nu\rangle}{|A||\nabla G|}}
 -\cBo{\frac{\sqrt\rho}{2}
 \frac{G\langle\nabla|A|,\nu\rangle}{|A||\nabla G|}}
 -\cBt{\frac{3-\sigma}{2}\sqrt\rho\,Z_\sigma(\nu,\nu)}.
\end{align*}
Expanding the square in the definition of $\Gamma_\sigma$ gives
\begin{align*}
 \Gamma_\sigma
 ={}&\frac{\rho}{m_{11}+2m_{12}\sqrt\rho+m_{22}\rho}\Biggl[
 \cGr{\frac{G^2|\nabla|A||^2}{|A|^2|\nabla G|^2}}
 +\cHz{|Z_\sigma\nu|^2}
 -\cMx{2\frac{G\langle\nabla|A|,Z_\sigma\nu\rangle}{|A||\nabla G|}}\\
 &{}+\cBo{2\left(1-\frac\sigma2\right)
 \frac{G\langle\nabla|A|,\nu\rangle}{|A||\nabla G|}}
 -\cBt{2\left(1-\frac\sigma2\right)Z_\sigma(\nu,\nu)}
 +\cBz{\left(1-\frac\sigma2\right)^2}\Biggr].
\end{align*}
The multipliers in \eqref{eq:matrix-block-square} are
$m_{11},m_{22},2m_{12},\det\mathsf M$, respectively.
Recall that
\begin{gather*}
 \beta(\rho)=\frac{\det\mathsf M}{m_{11}+2m_{12}\sqrt\rho+m_{22}\rho},\\
 \widetilde m_{11}(\rho)=m_{11}-\rho\beta(\rho),\qquad
 \widetilde m_{12}(\rho)=m_{12}+\sqrt\rho\beta(\rho).
\end{gather*}
Multiplying the expansion of $\Gamma_\sigma$ by $\det\mathsf M$ replaces its
prefactor by $\rho\beta(\rho)$.
The resulting \cGr{\text{brown}}, \cHz{\text{cyan}}, and \cMx{\text{red}} coefficients are
\begin{gather*}
 \cGr{-m_{22}\rho+\rho\beta(\rho)
 =-(m_{22}-\beta(\rho))\rho},\\
 \cHz{-m_{11}+\rho\beta(\rho)=-\widetilde m_{11}(\rho)},\qquad
 \cMx{-2m_{12}\sqrt\rho-2\rho\beta(\rho)
 =-2\widetilde m_{12}(\rho)\sqrt\rho}.
\end{gather*}

\subsubsection*{The expanded matrix contribution}
Combining $m_{11}F_1$, $2m_{12}F_9$, $m_{22}F_5$, and $(\det\mathsf M)\Gamma_\sigma$ gives
\begin{equation}\label{eq:matrix-block-square}
\begin{aligned}
 &m_{11}F_1+2m_{12}F_9+m_{22}F_5+(\det\mathsf M)\Gamma_\sigma\\
 ={}&\cGa{m_{22}\rho\frac{G^2|\nabla A|^2}{|A|^2|\nabla G|^2}}
 -\cGr{(m_{22}-\beta(\rho))\rho
 \frac{G^2|\nabla|A||^2}{|A|^2|\nabla G|^2}}
 -\cHz{\widetilde m_{11}(\rho)|Z_\sigma\nu|^2}\\
 &{}-\cMx{2\widetilde m_{12}(\rho)\sqrt\rho
 \frac{G\langle\nabla|A|,Z_\sigma\nu\rangle}{|A||\nabla G|}}
 +\cBo{\left[(2-\sigma)\rho\beta(\rho)-m_{12}\sqrt\rho\right]
 \frac{G\langle\nabla|A|,\nu\rangle}{|A||\nabla G|}}\\
 &{}-\cBt{\left[m_{12}(3-\sigma)\sqrt\rho
 +(2-\sigma)\rho\beta(\rho)\right]Z_\sigma(\nu,\nu)}\\
 &{}+\cBz{m_{11}\rho+2m_{12}\rho^{3/2}-\frac{m_{11}}4
 -\frac{m_{12}(3-\sigma)}2\sqrt\rho}\\
 &{}+\cBz{-\frac{m_{22}}4(3-\sigma)^2\rho
 +\rho\beta(\rho)\left(1-\frac\sigma2\right)^2}.
\end{aligned}
\end{equation}
Adding $F_2+c_3F_3+c_4F_4+c_6F_6+c_7F_7+c_8F_8$ gives
\begin{align*}
 \mathscr P_\sigma
 ={}&\cBz{B_0(|\widehat A\nu|^2)}
 +\cHz{|Z_\sigma|^2+c_3\rho\langle\widehat A^2-g/6,Z_\sigma\rangle}
 +\cGa{(m_{22}-c_6)\rho\frac{G^2|\nabla A|^2}{|A|^2|\nabla G|^2}}\\
 &{}-\cGr{(m_{22}-\beta(\rho))\rho
 \frac{G^2|\nabla|A||^2}{|A|^2|\nabla G|^2}}
 -\cHz{\widetilde m_{11}(\rho)|Z_\sigma\nu|^2}\\
 &{}-\cMx{2\widetilde m_{12}(\rho)\sqrt\rho
 \frac{G\langle\nabla|A|,Z_\sigma\nu\rangle}{|A||\nabla G|}}
 +\cBo{B_1\sqrt\rho\frac{G\langle\nabla|A|,\nu\rangle}{|A||\nabla G|}}
 +\cBt{B_2Z_\sigma(\nu,\nu)}.
\end{align*}
In particular, the \cGa{\text{green terms}} combine by $m_{22}-c_6$.
The \cBo{\text{orange terms}} produce $B_1\sqrt\rho$, so one factor $\sqrt\rho$
is removed when defining $B_1$. The \cBt{\text{purple terms}} directly produce $B_2$.

\subsubsection*{The three collected coefficients}
Writing $u=|\widehat A\nu|^2$, the \cBz{\text{blue terms}} give
\begin{align*}
 &\cBz{B_0(u)}=\cBz{\frac{\sigma(5-2\sigma)}{10}-\frac{m_{11}}4
 +c_6\rho^2+2m_{12}\rho^{3/2}
 +\rho\beta(\rho)\left(1-\frac\sigma2\right)^2}\\
 &{}+\cBz{\left[c_8\left(1-\frac\sigma2\right)
 -\frac{m_{12}(3-\sigma)}2\right]\sqrt\rho}
 +\cBz{\rho\left[\left(2c_3\left(1-\frac\sigma5\right)-1\right)u\right.}\\
 &\cBz{\left.-c_3\left(1-\frac{2\sigma}5\right)+m_{11}
 -\frac{c_6}2(3-\sigma)(\sigma-2)
 -\frac{m_{22}}4(3-\sigma)^2-\frac{c_7}2(\sigma-2)\right]}.
\end{align*}
The \cBo{\text{orange terms}} give
\[
 \cBo{B_1=c_7\sqrt\rho+c_8-m_{12}
 +2\sqrt\rho\beta(\rho)\left(1-\frac\sigma2\right)}.
\]
The \cBt{\text{purple terms}} give
\[
 \cBt{B_2=c_4+\bigl(c_8-m_{12}(3-\sigma)\bigr)\sqrt\rho
 -2\rho\beta(\rho)\left(1-\frac\sigma2\right)}.
\]
These are exactly the definitions in \eqref{eq:nonlinear-coefficients}.
The sole term depending on $u$ is
\[
 \cBz{\rho\left[2c_3\left(1-\frac\sigma5\right)-1\right]u}.
\]
\subsection{Spectral expression}\label{sec:spectral-colored-coefficients}
We expand the integrands in \eqref{eq:Psp}, using their definitions
in Subsection~\ref{subsec:compact-identities} and the explicit formulas in
Remark~\ref{rem:explicit-curvature-integrands}, and collect their terms to obtain the expression used in
Lemma~\ref{lem:spectral-tensor-reduction}.
In an orthonormal eigenbasis of $\widehat A$, write
\[
 C_{ijk}:=\frac{G\nabla_iA_{jk}}{|A||\nabla G|},\qquad
 Z_{jk}:=Z_\sigma(E_j,E_k),\qquad \nu_i:=\langle\nu,E_i\rangle.
\]
All indices range from $1$ to $6$. We use
\[
 \sum_j\widehat\lambda_j=0,\quad
 \sum_j\widehat\lambda_j^2=1,\quad
 \sum_i\nu_i^2=1,\quad \sum_jZ_{jj}=0,\quad \sum_jC_{ijj}=0.
\]
Thus $|Z_\sigma|^2=\sum_{j,k}Z_{jk}^2$, and the normalized derivative satisfies
\[
 \frac{G\nabla_i|A|}{|A||\nabla G|}
 =\sum_j\widehat\lambda_jC_{ijj}.
\]
The same color and frame are used for each family throughout this appendix.
\begin{center}
\setlength{\tabcolsep}{8pt}
\begin{tabular}{ll}
\cBz{\rho\nu_i^2} & Scalar terms, independent of $C$ and $Z_\sigma$ \\
\addlinespace[2pt]
\cHz{Z_{jk}} & Quadratic and linear Hessian terms \\
\addlinespace[2pt]
\cGa{C_{abc}^2} & The full three-tensor norm \\
\addlinespace[2pt]
\cGr{(\sum_j\widehat\lambda_jC_{ijj})^2} & The squared derivative of $|A|$ \\
\addlinespace[2pt]
\cBo{\nu_iC_{ijj}} & Terms linear in the three-tensor
\end{tabular}
\end{center}

\subsubsection*{The five lower-order integrands}
The definitions in Section~\ref{sec:first-estimates} and
Subsection~\ref{subsec:compact-identities} give
\begin{align*}
 F_2={}&\cHz{\sum_{j,k=1}^6Z_{jk}^2}
 +\cBz{\frac{\sigma(5-2\sigma)}{10}}
 -\cBz{\rho\sum_{i=1}^6\widehat\lambda_i^2\nu_i^2},\displaybreak[0]\\
 F_3={}&\cHz{\rho\sum_{j=1}^6\widehat\lambda_j^2Z_{jj}}
 +\cBz{2\left(1-\frac\sigma5\right)\rho
 \sum_{i=1}^6\widehat\lambda_i^2\nu_i^2}
 -\cBz{\rho\left(1-\frac{2\sigma}5\right)},\displaybreak[0]\\
 F_5={}&\cGa{\rho\sum_{a,b,c=1}^6C_{abc}^2}
 -\cGr{\rho\sum_{i=1}^6
 \left(\sum_{j=1}^6\widehat\lambda_jC_{ijj}\right)^2}
 -\cBz{\frac{(3-\sigma)^2}{4}\rho},\\
 F_6={}&\cBz{\rho^2}
 -\cGa{\rho\sum_{a,b,c=1}^6C_{abc}^2}
 -\cBz{\frac{(3-\sigma)(\sigma-2)}2\rho},\\
 F_7={}&\cBo{\rho\sum_{i,j=1}^6\nu_i\widehat\lambda_jC_{ijj}}
 -\cBz{\frac{\sigma-2}{2}\rho}.
\end{align*}

\subsubsection*{The three higher-order curvature integrands}
Substituting \eqref{eq:div-quartic-tensor}--\eqref{eq:div-scalar-tensor}
into \eqref{eq:weighted-divergence-integrand} gives the full expansions
\begin{align*}
 F_{10}={}&\cHz{\rho\sum_{j=1}^6\widehat\lambda_j^4Z_{jj}}
 +\cBz{2\left(1-\frac\sigma5\right)\rho
 \sum_{i=1}^6\widehat\lambda_i^4\nu_i^2}
 -\cBz{\frac\sigma{10}\rho s_4(\widehat{\boldsymbol\lambda})}\\
 &{}+\cBo{\rho\sum_{i,j=1}^6\nu_i
 (\widehat\lambda_j^3+\widehat\lambda_i\widehat\lambda_j^2)C_{ijj}}
 +\cBo{\rho\sum_{i,j=1}^6\nu_i
 (\widehat\lambda_i^2\widehat\lambda_j
 -2\widehat\lambda_i^4\widehat\lambda_j)C_{ijj}},\\
 F_{11}={}&\cHz{\rho s_3(\widehat{\boldsymbol\lambda})
 \sum_{j=1}^6\widehat\lambda_jZ_{jj}}
 +\cBz{2\left(1-\frac\sigma5\right)\rho s_3(\widehat{\boldsymbol\lambda})
 \sum_{i=1}^6\widehat\lambda_i\nu_i^2}\\
 &{}+\cBo{\rho\sum_{i,j=1}^6\nu_i
 \bigl(3\widehat\lambda_i\widehat\lambda_j^2
 -2s_3(\widehat{\boldsymbol\lambda})\widehat\lambda_i\widehat\lambda_j\bigr)C_{ijj}},\\
 F_{12}={}&\cBz{(2-\sigma)\rho s_4(\widehat{\boldsymbol\lambda})}
 +\cBo{4\rho\sum_{i,j=1}^6\nu_i\widehat\lambda_j^3C_{ijj}}
 -\cBo{2\rho s_4(\widehat{\boldsymbol\lambda})
 \sum_{i,j=1}^6\nu_i\widehat\lambda_jC_{ijj}}.
\end{align*}
Here $s_k(\widehat{\boldsymbol\lambda})=\sum_j\widehat\lambda_j^k$.
The trace term in $F_{11}$ vanishes because
$\operatorname{tr}\widehat A=0$, and the Hessian term in $F_{12}$ vanishes
because $\operatorname{tr}Z_\sigma=0$.
Their multipliers in \eqref{eq:Psp} are $c_{10},c_{11},c_{12}$, respectively.

\subsubsection*{Hessian terms}
Recall from \eqref{eq:spectral-N} that
\[
 \cHz{\mathcal N_j=c_3\widehat\lambda_j^2+c_{10}\widehat\lambda_j^4
 +c_{11}s_3(\widehat{\boldsymbol\lambda})\widehat\lambda_j
 -\frac{c_3+c_{10}s_4(\widehat{\boldsymbol\lambda})}{6}}.
\]
Since $\sum_jZ_{jj}=0$, the constant term in $\mathcal N_j$ does not
contribute to $\sum_j\mathcal N_jZ_{jj}$. Thus the \cHz{\text{cyan terms}} in
$F_2,c_3F_3,c_{10}F_{10},c_{11}F_{11}$ give
\begin{align*}
 &\cHz{\sum_{j,k=1}^6Z_{jk}^2}
 +\cHz{c_3\rho\sum_{j=1}^6\widehat\lambda_j^2Z_{jj}}
 +\cHz{c_{10}\rho\sum_{j=1}^6\widehat\lambda_j^4Z_{jj}}
 +\cHz{c_{11}\rho s_3(\widehat{\boldsymbol\lambda})
 \sum_{j=1}^6\widehat\lambda_jZ_{jj}}\\
 ={}&\cHz{|Z_\sigma|^2+\rho\sum_{j=1}^6\mathcal N_jZ_{jj}}.
\end{align*}

\subsubsection*{Squared norm of the curvature gradient}
The \cGa{\text{green terms}} from $c_5F_5$ and $c_6F_6$ combine as
\[
 \cGa{c_5\rho\sum_{a,b,c=1}^6C_{abc}^2}
 -\cGa{c_6\rho\sum_{a,b,c=1}^6C_{abc}^2}
 =\cGa{(c_5-c_6)\rho\sum_{a,b,c=1}^6C_{abc}^2}.
\]

\subsubsection*{Terms linear in the curvature gradient}
For the \cBo{\text{orange terms}}, the coefficient of $\rho\nu_iC_{ijj}$ in each
weighted integrand is
\begin{center}
\setlength{\tabcolsep}{8pt}
\renewcommand{\arraystretch}{1.6}
\begin{tabular}{cl}
\toprule
Source & Coefficient of $\rho\nu_iC_{ijj}$ \\
\midrule
$c_7F_7$ & \cBo{c_7\widehat\lambda_j} \\
\addlinespace[2pt]
$c_{10}F_{10}$ & \cBo{c_{10}\bigl(\widehat\lambda_j^3
 +\widehat\lambda_i\widehat\lambda_j^2
 +\widehat\lambda_i^2\widehat\lambda_j
 -2\widehat\lambda_i^4\widehat\lambda_j\bigr)} \\
\addlinespace[2pt]
$c_{11}F_{11}$ & \cBo{3c_{11}\widehat\lambda_i\widehat\lambda_j^2
 -2c_{11}s_3(\widehat{\boldsymbol\lambda})\widehat\lambda_i\widehat\lambda_j} \\
\addlinespace[2pt]
$c_{12}F_{12}$ & \cBo{4c_{12}\widehat\lambda_j^3
 -2c_{12}s_4(\widehat{\boldsymbol\lambda})\widehat\lambda_j} \\
\bottomrule
\end{tabular}
\end{center}
Collecting the cubic, quadratic, and linear powers of $\widehat\lambda_j$
in this column gives \eqref{eq:spectral-V} explicitly as
\[
\begin{aligned}
 V_j^{(i)}={}&\cBo{(c_{10}+4c_{12})\widehat\lambda_j^3
 +(c_{10}+3c_{11})\widehat\lambda_i\widehat\lambda_j^2}\\
 \MoveEqLeft[3]{}+\cBo{\bigl[c_7+c_{10}(\widehat\lambda_i^2-2\widehat\lambda_i^4)
 -2c_{11}s_3(\widehat{\boldsymbol\lambda})\widehat\lambda_i
 -2c_{12}s_4(\widehat{\boldsymbol\lambda})\bigr]\widehat\lambda_j}.
\end{aligned}
\]
Thus the \cBo{\text{orange terms}} sum to
$\rho\sum_{i,j=1}^6\nu_iV_j^{(i)}C_{ijj}$.

Combining the \cGa{\text{green}}, \cGr{\text{brown}}, and \cBo{\text{orange terms}} and using
\eqref{eq:spectral-codazzi-decomposition} gives
\begin{align*}
 &\rho\sum_{i=1}^6\Biggl[
 \cGa{(c_5-c_6)\left(C_{iii}^2+3\sum_{j\ne i}C_{ijj}^2\right)}
 -\cGr{c_5\left(\sum_{j=1}^6\widehat\lambda_jC_{ijj}\right)^2}\\
 &\hspace{32mm}{}+\cBo{\nu_i\sum_{j=1}^6V_j^{(i)}C_{ijj}}\Biggr]
 +\cGa{6\rho(c_5-c_6)\sum_{a<b<c}C_{abc}^2}.
\end{align*}
This identifies each independent group and its linear term in the proof of
Lemma~\ref{lem:spectral-tensor-reduction}.

\subsubsection*{Scalar terms}
The \cBz{\text{blue terms}} first give
$\sigma(5-2\sigma)/10+c_6\rho^2$.
Using $\sum_i\nu_i^2=1$, the remaining scalar terms are
$\rho\sum_i\nu_i^2$ times the sum of the coefficients in the following table, evaluated at $i$.
\begin{center}
\setlength{\tabcolsep}{6pt}
\renewcommand{\arraystretch}{1.6}
\begin{tabular}{@{}cl@{\hspace{18pt}}cl@{}}
\toprule
Source & Coefficient of $\rho\nu_i^2$ & Source & Coefficient of $\rho\nu_i^2$ \\
\midrule
$F_2$ & \cBz{-\widehat\lambda_i^2}
 & $c_3F_3$ & \cBz{2c_3(1-\sigma/5)\widehat\lambda_i^2-c_3(1-2\sigma/5)} \\
\addlinespace[3pt]
$c_5F_5$ & \cBz{-\frac{c_5}{4}(3-\sigma)^2}
 & $c_6F_6$ & \cBz{-\frac{c_6}{2}(3-\sigma)(\sigma-2)} \\
\addlinespace[3pt]
$c_7F_7$ & \cBz{-\frac{c_7}{2}(\sigma-2)}
 & $c_{10}F_{10}$ & \cBz{c_{10}\bigl[2(1-\sigma/5)\widehat\lambda_i^4-\sigma s_4(\widehat{\boldsymbol\lambda})/10\bigr]} \\
\addlinespace[3pt]
$c_{11}F_{11}$ & \cBz{2c_{11}(1-\sigma/5)s_3(\widehat{\boldsymbol\lambda})\widehat\lambda_i}
 & $c_{12}F_{12}$ & \cBz{c_{12}(2-\sigma)s_4(\widehat{\boldsymbol\lambda})} \\
\bottomrule
\end{tabular}
\end{center}
Substituting \eqref{eq:spectral-N} into the sum of these coefficients and combining
the collected terms with the \cGr{\text{brown term}} from $c_5F_5$ gives
\eqref{eq:spectral-collected-expansion}, with each family marked by its color and frame,
\begin{equation*}
\begin{aligned}
 \widetilde{\mathscr P}_\sigma={}&\cBz{\frac{\sigma(5-2\sigma)}{10}+c_6\rho^2}
 +\cHz{|Z_\sigma|^2+\rho\sum_{j=1}^6\mathcal N_jZ_{jj}}
 +\cGa{(c_5-c_6)\rho\sum_{a,b,c=1}^6C_{abc}^2}\\
 &{}+\cBz{\rho\sum_{i=1}^6\nu_i^2
 \left[2\left(1-\frac\sigma5\right)\mathcal N_i
 -\widehat\lambda_i^2-\frac{c_5}{4}(3-\sigma)^2\right]}\\
 &{}+\cBz{(\sigma-2)\rho
 \left[\frac{c_3}{3}-\frac{c_7}{2}-\frac{c_6}{2}(3-\sigma)\right]}
 -\cBz{(\sigma-2)\rho\left(\frac{c_{10}}6+c_{12}\right)
 s_4(\widehat{\boldsymbol\lambda})}
\\
 &{}-\cGr{c_5\rho\sum_{i=1}^6
 \left(\sum_{j=1}^6\widehat\lambda_jC_{ijj}\right)^2}
 +\cBo{\rho\sum_{i,j=1}^6\nu_iV_j^{(i)}C_{ijj}}.
\end{aligned}
\end{equation*}
The \cBz{\text{blue terms}} proportional to $\rho$ are precisely the terms outside the
minimum in \eqref{eq:spectral-LQ}, summed with weights $\nu_i^2$.
The minimum in \eqref{eq:spectral-LQ} gives the part of $L_{\sigma,i}$ coming from the three-tensor terms.
Finally, the \cHz{\text{cyan terms}} together with $c_6\rho^2$ equal the nonnegative
square in \eqref{eq:spectral-direction-bound} plus $Q\rho^2$.

\par\endgroup
\bibliographystyle{amsalpha}
\bibliography{R7}

\end{document}